%% file: scht1.tex
\documentclass[reqno]{amsart}
\newcommand{\papertitle}{Semiclassical Hodge theory for log Poisson manifolds}

\input{poisson-mhs-preamble}

\title{\papertitle}
\author{Aidan Lindberg}
\address{University of Toronto}
\email{aidan.lindberg@mail.utoronto.ca}

\author{Brent Pym}
\address{McGill University}
\email{brent.pym@mcgill.ca}
\begin{document}
           
\maketitle

\begin{abstract}
 We construct a mixed Hodge structure on the topological K-theory of smooth   Poisson varieties, depending weakly on a choice of compactification.  We  establish a package of tools for calculations with these structures, such as functoriality statements, projective bundle formulae, Gysin sequences and Torelli properties.  We show that for varieties with trivial A-hat class, the corresponding period maps for families can be written as exponential maps for bundles of tori, which we call the ``quantum parameters''.  As justification for the terminology, we show that in many interesting examples, the quantum parameters of a Poisson variety coincide with the parameters appearing in its known deformation quantizations.  In particular, we give a detailed implementation of an argument of Kontsevich, to  prove that his canonical quantization formula, when applied to Poisson tori, yields noncommutative tori with parameter ``$q = e^\hbar$''.
\end{abstract}

{
\setcounter{tocdepth}{1} 
\tableofcontents 
}

\input{poisson-mhs-intro}

\input{poisson-mhs-log-varieties}

\input{poisson-mhs-deRham}

\input{poisson-mhs-periods}

\input{poisson-mhs-adams}

\input{poisson-mhs-qtori}

\bibliographystyle{hyperamsalpha}

\bibliography{poisson-mhs}
       
\end{document}

%% file: poisson-mhs-preamble.tex
\usepackage{xargs}
\usepackage{amsfonts,amsmath,amssymb,amsthm,comment,bbm}
\usepackage{tikz,tikz-cd}
\usepackage{mathscinet}
\usepackage{mathrsfs}

\usepackage[font=small,labelfont=bf]{caption}
\numberwithin{equation}{section}

\usepackage{subcaption}

\usepackage{todonotes}

\newcommand{\sbt}{\,\begin{picture}(-1,1)(-1,-2)\circle*{2}\end{picture}\ }

\usepackage{xcolor,hyperref}
\definecolor{darkred}{rgb}{.8,0,0}
\definecolor{tocolor}{rgb}{.1,.1,.1}
\definecolor{urlcolor}{rgb}{.2,.2,.6}
\definecolor{linkcolor}{rgb}{.1,.1,.5}
\definecolor{citecolor}{rgb}{.4,.2,.1}
\definecolor{gray}{rgb}{.8,.8,.8}

\hypersetup{
	backref=true,
	colorlinks=true,
	urlcolor=urlcolor,
	linkcolor=linkcolor,
	citecolor=citecolor,
	pdfauthor = {Brent Pym},
	pdftitle = {\papertitle}
	}
	
\usepackage{aliascnt}

\newcommand{\thdef}[2]{
	\newaliascnt{#1}{theorem}  
	\newtheorem{#1}[#1]{#2}
	\aliascntresetthe{#1}  
	\newtheorem*{#1*}{#2}
	\expandafter\newcommand\expandafter{\csname #1autorefname\endcsname}{#2}
}

\newtheorem{theorem}{Theorem}[section]
\newtheorem*{theorem*}{Theorem}
\newtheorem{mainthm}{Theorem}

\thdef{conjecture}{Conjecture}
\thdef{problem}{Problem}
\thdef{lemma}{Lemma}
\thdef{corollary}{Corollary}
\thdef{proposition}{Proposition}
\theoremstyle{definition}
\thdef{definition}{Definition}
\thdef{notation}{Notation}
\thdef{question}{Question}

\theoremstyle{remark}
\thdef{remarkx}{Remark}
\thdef{examplex}{Example}

\newenvironment{example}
  {\pushQED{\qed}\examplex}
  {\popQED\endexamplex}

\newenvironment{remark}
  {\pushQED{\qed}\remarkx}
  {\popQED\endremarkx}

\newcounter{abvarcount}

\newcounter{toriccount}

\newcommand{\toricex}{\stepcounter{toriccount}Toric varieties, part \arabic{toriccount}}

\newcounter{twopurecount}

\newcommand{\twopureex}{\stepcounter{twopurecount}Two-pure varieties, part \arabic{twopurecount}}

\newcounter{sklyanincount}

\newcommand{\sklyaninex}{\stepcounter{sklyanincount}Sklyanin's elliptic Poisson structures, part \arabic{sklyanincount}}

\newcounter{surfacecount}

\newcommand{\surfaceex}{\stepcounter{surfacecount}Surfaces, part \arabic{surfacecount}}

\newcounter{threefoldcount}

\newcounter{logsympcount}

\newcommand{\logsympex}{\stepcounter{logsympcount}Log symplectic manifolds, part \arabic{logsympcount}}

\newcommand{\defn}[1]{\textbf{\emph{#1}}}

\usepackage{xypic}

\usepackage{tikz}
\usetikzlibrary{decorations.markings,cd}

\newcommand{\abrac}[1]{\left\langle#1\right\rangle}

\newcommand{\rbrac}[1]{\left(#1\right)}
\newcommand{\pairnaive}[1]{\rbrac{#1}}
\newcommand{\pair}[1]{\abrac{#1}}
\newcommand{\sbrac}[1]{\left[#1\right]}
\newcommand{\set}[2]{\left\{#1 \,\middle|\, #2 \right\}}

\newcommand{\mapdef}[5]{\begin{array}{ccccc}
#1 &:& #2 &\rightarrow& #3 \\
&& #4 &\mapsto& #5
\end{array}}

\newcommand{\CC}{\mathbb{C}}
\newcommand{\PP}{\mathbb{P}}
\newcommand{\RR}{\mathbb{R}}

\newcommand{\ZZ}{\mathbb{Z}}
\newcommand{\QQ}{\mathbb{Q}}

\newcommand{\bA}{\mathbb{A}}

\newcommand{\iu}{\mathrm{i}}

\newcommand{\famX}{\mathcal{X}}
\newcommand{\famXo}{\mathcal{X}^\circ}
\newcommand{\famXc}{\overline{\mathcal{X}}}
\newcommand{\fambdX}{\partial\mathcal{X}}
\newcommand{\XX}{\X}
\newcommand{\X}{\mathsf{X}}
\newcommand{\cX}{\overline{\mathsf{X}}}
\newcommand{\tX}{\widetilde{\X}}

\newcommand{\bdX}[1][]{\partial^{#1}\X}
\newcommand{\Xo}{\X^\circ}

\newcommand{\Q}[2][]{\mathsf{Q}_{#1}(#2)}
\newcommand{\QP}[1][]{\mathsf{Q}}  
\newcommand{\allones}[1][]{\mathbf{1}} 
\newcommand{\Aut}[1]{\mathsf{Aut}\rbrac{#1}}
\newcommand{\Gm}{\CC^\times}
\newcommand{\QQx}{\QQ^\times}
\newcommand{\CCx}{\CC^\times}
\newcommand{\symgp}[1]{S_{#1}}
\newcommand{\U}{\mathsf{U}}

\newcommand{\V}{\mathsf{V}}

\newcommand{\Bl}[2]{\mathsf{Bl}_{#2}(#1)}

\newcommand{\Y}{\mathsf{Y}}
\newcommand{\tY}{\widetilde{\mathsf{Y}}}
\newcommand{\cY}{\overline{\mathsf{Y}}}
\newcommand{\bdY}{\partial\Y}
\newcommand{\Yo}{\Y^\circ}
\newcommand{\Z}{\mathsf{Z}}

\newcommand{\cZ}{\overline{\mathsf{Z}}}
\newcommand{\N}{\mathsf{N}}

\newcommand{\bZ}[1][]{\overline{\Z}}
\newcommand{\Zo}{\mathsf{Z}^\circ}

\newcommand{\bS}{\mathsf{S}}

\newcommand{\sing}[1]{{#1}_{\mathrm{sing}}}

\newcommand{\per}{\wp}
\newcommand{\KodSpenc}{\kappa}
\newcommand{\quant}{\widehat{\mathrm{quant}}}

\newcommand{\ctb}[1]{\mathsf{T}^*#1}

\DeclareMathOperator{\rank}{rank}
\DeclareMathOperator{\codim}{codim}

\DeclareMathOperator{\img}{image}
\DeclareMathOperator{\res}{Res}
\DeclareMathOperator{\Sym}{Sym}

\newcommand{\id}{\mathrm{id}}
\DeclareMathOperator{\sgn}{sgn}

\newcommand{\adm}{\diamond} 
\newcommand{\cl}[2][\mathrm{cl}]{{#2}^{(#1)}} 
\newcommand{\ps}{\sigma}
\newcommand{\pss}{\ps^\sharp}

\newcommand{\cT}[1]{\mathcal{T}_{#1}} 
\newcommand{\cN}[1]{\mathcal{N}_{#1}} 
\newcommand{\coN}[1]{\mathcal{N}^\vee_{#1}} 

\newcommand{\cTlog}[1]{\mathcal{T}_{\overline{#1}}(-\log \partial #1)} 
\newcommand{\der}[2][\bullet]{\mathscr{X}^{#1}_{#2}} 
\newcommand{\forms}[2][\bullet]{\Omega^{#1}_{#2}} 
\newcommand{\rforms}[2][\bullet]{\Omega^{-#1}_{#2}}
\newcommand{\logforms}[2][\bullet]{\Omega^{#1}_{\overline{#2}}(\log \partial#2)} 

\newcommand{\dolb}[2][\bullet,\bullet]{\mathcal{A}^{#1}_{#2}}
\newcommand{\Dolb}[2][\bullet,\bullet]{\mathscr{A}^{#1}({#2})}
\newcommand{\Dolbc}[2][\bullet]{\mathscr{A}^{#1}_c({#2})}

\newcommand{\cO}[1]{\mathcal{O}_{#1}} 
\newcommand{\cI}{\mathcal{I}} 

\newcommand{\sL}{\mathcal{L}}

\newcommand{\coH}[2][\bullet]{\mathsf{H}^{#1}\!\rbrac{#2}}
\newcommand{\sHom}[2][]{\mathcal{H}om_{#1}(#2)}
\newcommandx{\sEnd}[3][1=,2=]{\mathcal{E}nd_{#1}^{#2}(#3)}
\newcommand{\HdR}[2][\bullet]{\mathsf{H}_{\mathrm{dR}}^{#1}(#2)}
\newcommand{\HdRc}[2][\bullet]{\mathsf{H}_{\mathrm{dR},c}^{#1}(#2)}

\newcommandx{\PcoHdR}[3][1=\bullet,2=j]{\bigoplus_{#2 \in \ZZ}\HdR[#1+2#2]{#3}}

\newcommand{\E}{\mathsf{E}}
\newcommand{\Eo}{\E^\circ}
\newcommand{\cE}{\overline{\E}}

\newcommand{\sMix}[2][]{\mathcal{M}ix_{#2}^{#1}}
\newcommand{\Mix}[2][]{\mathsf{Mix}^{#1}(#2)}
\newcommand{\sMixc}[2][]{\mathcal{M}ix_{#2;c}^{#1}}
\newcommand{\Mixc}[2][]{\mathsf{Mix}^{#1}_{c}(#2)}

\newcommand{\HP}[2][\bullet]{\mathsf{HP}_{#1}(#2)}
\newcommand{\Hlgy}[2][\bullet]{\mathsf{H}_{#1}(#2)}

\newcommand{\rsect}[1]{\mathsf{R\Gamma}\rbrac{#1}}
\newcommand{\K}[2][\bullet]{\mathsf{K}^{#1}(#2)}
\newcommand{\tK}[2][\bullet]{\widetilde{\mathsf{K}}^{#1}(#2)}
\newcommand{\Kc}[2][\bullet]{\mathsf{K}^{#1}_c(#2)}
\newcommand{\shfK}[2][\bullet]{\mathcal{K}^{#1}(#2)}

\newcommand{\KB}[2][\bullet]{\mathsf{K}^{#1}_{\mathrm{B}}(#2)}
\newcommand{\KBc}[2][\bullet]{\mathsf{K}^{#1}_{\mathrm{B},c}(#2)}
\newcommand{\sKB}[2][\bullet]{\mathcal{K}^{#1}_{\mathrm{B}}(#2)}

\newcommand{\KR}[2][\bullet]{\mathsf{K}^{#1}(#2)_{\RR}}
\newcommand{\KRc}[2][\bullet]{\mathsf{K}^{#1}_{c}(#2)_{\RR}}
\newcommand{\KdR}[2][\bullet]{\mathsf{K}^{#1}_{\mathrm{dR}}(#2)}
\newcommand{\KdRc}[2][\bullet]{\mathsf{K}^{#1}_{\mathrm{dR},c}(#2)}
\newcommand{\sKdR}[2][\bullet]{\mathcal{K}^{#1}_{\mathrm{dR}}(#2)}

\newcommand{\KDol}[2][\bullet]{\mathsf{K}^{#1}_{\mathrm{Dol}}(#2)}
\newcommand{\sKDol}[2][\bullet]{\mathcal{K}^{#1}_{\mathrm{Dol}}(#2)}

\newcommand{\Ktop}[2][\bullet]{\mathsf{K}^{#1}_{\mathrm{top}}(#2)}

\newcommand{\End}[2][]{\mathsf{End}^{#1}\rbrac{#2}}
\newcommand{\Hom}[2][]{\mathsf{Hom}_{#1}\rbrac{#2}}

\newcommand{\ExtZ}[2][]{\mathsf{Ext}_{\ZZ}^{#1}(#2)}
\newcommand{\ExtMHS}[2][1]{\mathsf{Ext}_{\mathrm{MHS}}^{#1}(#2)}

\newcommand{\Jac}[2][]{\mathsf{J}^{#1}(#2)}

\newcommand{\dd}{\mathrm{d}}
\newcommand{\dlog}[1]{\tfrac{\mathrm{d}#1}{#1}}

\newcommand{\del}{\partial}
\newcommand{\delb}{\overline{\partial}}
\newcommand{\delps}[1][]{\delta_{\ps_{#1}}}
\newcommand{\hook}[1]{\iota_{#1}}
\newcommand{\hookps}{\hook{\ps}}
\newcommand{\cvf}[1]{\partial_{#1}}
\newcommand{\logcvf}[1]{#1\partial_{#1}}

\newcommand{\QCoh}[1]{\mathsf{QCoh}(#1)}

\newcommand{\SF}[1][\bullet]{s_{#1}}
\newcommand{\grS}[1][\bullet]{\mathrm{gr}^{\SF[]}_{#1}}
\newcommand{\grW}[1][\bullet]{\mathsf{gr}^{\W[]}_{#1}}
\newcommand{\grF}[1][\bullet]{\mathsf{gr}_{\F[]}^{#1}}

\newcommand{\W}[1][\bullet]{W_{#1}}
\newcommand{\F}[1][\bullet]{F^{#1}}
\newcommand{\Fps}[1][\bullet]{F^{#1}_\ps}

\newcommand{\bFps}[1][\bullet]{\overline{F^{#1}_\ps}}

\newcommand{\MPois}{\mathfrak{M}_{\mathrm{Pois}}}
\newcommand{\MHodge}{\mathfrak{M}_{\mathrm{Hodge}}}

\newcommand{\Spec}[1]{\mathsf{Spec}(#1)}

\newcommand{\Bet}{\mathrm{B}}
\newcommand{\dR}{\mathrm{dR}}
\newcommand{\Flag}[2][\bullet]{\mathsf{Flag}_{#1}(#2)}

\newcommand{\tipi}{2 \pi \mathrm{i}}

\newcommand{\ch}[1][]{\mathrm{ch}_{#1}}
\newcommand{\charge}{c}

\newcommand{\Ahat}{\widehat{\mathrm{A}}}
\newcommand{\ahat}{\widehat{a}}
\newcommand{\modtwo}{\hspace{-0.7em}\mod 2}
\newcommand{\fg}{\mathfrak{g}}

\newcommand{\fj}{\mathfrak{j}}

\newcommand{\fq}{\mathfrak{q}}

\newcommand{\ft}{\mathfrak{t}}

\newcommand{\cA}{\mathcal{A}}

\newcommand{\calE}{\mathcal{E}}

\newcommand{\cM}{\mathcal{M}}

\newcommand{\sfS}{\mathsf{S}}

\newcommand{\zb}{\overline{z}}
\newcommand{\yb}{\overline{y}}

%% file: poisson-mhs-intro.tex
\section{Introduction}

\subsection{Overview}
This paper is the first in a series of works about ``semiclassical Hodge theory'', by which we mean the behaviour of Hodge theory under noncommutative deformations of smooth varieties.  This is a particular regime of Katzarkov--Kontsevich--Pantev's noncommutative Hodge theory~\cite{Katzarkov2008} in which everything can be phrased in terms of classical geometry and mixed Hodge structures, although these structures depend in an interesting way on the noncommutative deformation.  The ultimate aim, following a strategy of Kontsevich~\cite{Kontsevich2008a}, is to provide a tractable mechanism to calculate the result of his canonical deformation quantization formula~\cite{Kontsevich03} in many cases of interest.

We recall that the quantization formula associates, to every Poisson manifold, a noncommutative deformation of its algebra of functions.  The formula is a Feynman-style series expansion~\cite{Cattaneo2000}, whose beautiful structure presents several challenges.  For instance, it depends implicitly on a choice of ``gauge'', in the form of a Drinfel'd associator~\cite{Dolgushev2021,Kontsevich1999,Severa2011,Tamarkin1999,WW15,WWBraces}; its coefficients are multiple zeta values~\cite{Banks2020,Felder2010}, hence conjecturally transcendental; it can diverge, necessitating resummation~\cite{Garay2014,Li2023a}; and even when it converges, it is unclear how to calculate the sum.  Consequently, to our knowledge, the main cases in which the formula can be computed directly are present already in the original 1997 preprint of  \cite{Kontsevich03}.  

However, a deep insight of Kontsevich, outlined in \cite{Kontsevich2008a}, is that for \emph{holomorphic} Poisson brackets, many of these difficulties can be overcome by reinterpreting the problem as a comparison of certain Hodge-theoretic period maps. Our goal in this series of papers is to flesh this idea out.  As we shall explain, it allows one to establish the convergence and explicitly calculate the quantization  for a wide class of interesting Poisson structures, and provides a novel conceptual explanation for the recurrence of classical special functions in the relations defining well-known noncommutative algebras.  It also gives a useful tool to address questions about quantizability of modules, via the noncommutative Hodge conjecture~\cite{Lin2023}.

In this first paper, we develop the ``classical side'' of the story, i.e.~the Hodge theory of holomorphic Poisson manifolds.  Our main results concern the construction and properties of a natural mixed Hodge structure on the topological K-theory of a holomorphic Poisson manifold (equipped with a suitable equivalence class of compactifications).  Roughly speaking, this structure is obtained by deforming the classical mixed Hodge structure of Deligne~\cite{Deligne1971a} to incorporate information about the periods of the symplectic leaves of the Poisson bracket; it is closely related to the Hodge decomposition for generalized complex manifolds~\cite{Cavalcanti2006,Gualtieri2004}.  We examine the resulting period maps and explain how they can often be computed in terms of certain ``quantum parameters'' taking values in complex abelian Lie groups.  As justification for the terminology, and a preview of future developments, we prove that for toric Poisson varieties, these parameters correspond, under Kontsevich's quantization, to the usual ``$q$-parameters'' defining \emph{noncommutative} tori, thus providing an explicit non-perturbative calculation of the quantization, and verifying the noncommutative Hodge conjecture in this case.  All of the key obstructions to a direct calculation of the quantization formula are already present in this simple example---indeed it is the simplest example for which these complexities occur---so it nicely illustrates the efficacy of the Hodge-theoretic approach.

In the sequel, we will construct corresponding mixed Hodge structures on the ``quantum side'' using recent advances in noncommutative algebraic geometry, and prove in full generality that they agree with the structures presented here, so that many other natural examples of Poisson varieties can be treated in the same manner as tori.  In additional forthcoming work of the second author with Matviichuk--Lapointe and with Matviichuk--Schedler, we will explain how to further extend the class of examples that can be treated by leveraging orbifold resolution of singularities, and filtered deformations of toric structures, respectively.

We now give a detailed overview of contents of the present paper.
\subsection{Logarithmic Poisson manifolds}
The basic constructions of mixed Hodge theory for noncompact smooth varieties rely on the existence and uniqueness (up to ``zig-zags'') of smooth compactifications with normal crossings boundary, obtained using Nagata's compactification theorem~\cite{Nagata1962} and Hironaka's desingularization theorem~\cite{Hironaka1964}.   To adapt these methods to Poisson varieties,  we further require the compactifications to be compatible with Poisson brackets.  However, Poisson structures do not always lift along blowups, so the obvious analogues of Nagata and Hironoka's theorems fail in the Poisson category.  The existence and uniqueness also fail in the analytic category, since there are smooth varieties that are biholomorphic as complex manifolds, but have non-isomorphic mixed Hodge structures. As a result, we are led to carry around the compactifications as part of the data.  Thus, our main objects of study are the following:
\begin{definition}
A \defn{log manifold} is a pair $\X = (\cX,\bdX)$, where $\cX$ is a compact complex manifold that admits a K\"ahler metric (e.g.~a smooth projective $\CC$-variety), and $\bdX\subset \cX$ is a normal crossings  divisor.  A \defn{Poisson structure} on a log manifold $\X$ is a holomorphic Poisson bivector $\ps$ on $\cX$ that is tangent to $\bdX$.
\end{definition}

We allow the possibility that the normal crossings divisor $\bdX$ is smooth, or even empty.  We think of the interior $\X^\circ := \cX \setminus \bdX$ as the primary geometric object, so that two log Poisson manifolds are \defn{weakly equivalent} if they differ by a zig-zag of morphisms that restrict to isomorphisms of the interiors.  In particular, our main constructions will be invariant under weak equivalences.  Our notation throughout reflects the classical notation for the interiors, so for instance, the cohomology $\coH{\X}$ is canonically identified with the cohomology $\coH{\Xo}$ of the interior $\Xo$, which we have found convenient for calculations.  (In other words, we have adopted some notation from logarithmic algebraic geometry, although we do not work directly with log schemes in this paper.)

In \autoref{sec:poisson}, we recall and develop some basic geometry of log Poisson manifolds.  In particular, we highlight the important role played by Poisson submanifolds whose conormal Lie algebras are abelian; we call them ``coabelian'' for short.  As a consequence of Polishchuk's results on blowups of Poisson brackets~\cite{Polishchuk1997}, these are exactly the submanifolds whose complement has an obvious logarithmic Poisson compactification, given by blowing up the submanifold and adding the exceptional divisor to the boundary.  This allows us to produce natural logarithmic counterparts for many important noncompact Poisson varieties, related to tori, del Pezzo surfaces, elliptic Feigin--Odesskii--Sklyanin algebras, Fano 3-folds, etc., which we use as running examples throughout.

\subsection{Hodge--de Rham theory}
Classical Hodge theory is rooted in singular (Betti) and de Rham cohomology, but in noncommutative geometry, we lose direct access to these invariants, e.g.~since there is no underlying topological space in which to form simplices.    However, there are closely related invariants that persist: the topological K-theory and the periodic cyclic homology, which can be thought of roughly as the ``two-periodization'' of cohomology thanks to the Atiyah--Hirzebruch and Hochschild--Kostant--Rosenberg theorems, respectively.  Thus for a log manifold $\X$, we denote by $\KB{\X} := \Ktop{\Xo}$ the topological K-theory and by 
\[
\KdR{\X} := \bigoplus_{j \in \ZZ} \HdR[\bullet+2j]{\X} \cong \bigoplus_{j \in \ZZ} \HdR[\bullet+2j]{\Xo}
\]
the graded vector space obtained by the 2-periodization of the logarithmic de Rham cohomology of $\X$ (or equivalently its interior $\Xo$).

Here, there is an important subtlety, which is well known, but bears repeating: when defining the Chern character in de Rham cohomology via curvature of connections, one obtains classes in cohomology whose periods are not rational, but rather rational multiples of powers of $\tipi$, i.e.~Tate twists. This forces us to adjust the notion of complex conjugation on $\KdR{\X}$ relative to $\HdR{\X}$, and to reindex the the Hodge filtration by the columns of the Hodge diamond instead of the diagonals; see \autoref{fig:surface-hodge-filtration}.  This  matches the re-indexing that occurs when we relate $\KdR{\X}$ to periodic cyclic homology, or to the Hodge decomposition for generalized K\"ahler manifolds~\cite{Cavalcanti2006,Gualtieri2004}. 

In \autoref{sec:deRham}, following Kontsevich~\cite{Kontsevich2008a}, we explain how a Poisson structure $\ps$ on $\X$ induces a deformation of the usual Hodge filtration on differential forms, via the Brylinski--Koszul Poisson homology mixed complex~\cite{Brylinski1988,Koszul1985,Stienon2011}.  Then, using the classical theorems on degeneracy of the Hodge and weight spectral sequences, together with a suitable equivariance under the operation of rescaling the Poisson bivector, we prove that this deformed filtration (and its Poincar\'e dual in compactly supported cohomology) defines an $\RR$-Hodge structure:
\begin{mainthm}[see \autoref{thm:R-MHS}]\label{thm:mainA}
	Let $(\X,\ps)$ be a log Poisson manifold.  Then the weight filtration and $\ps$-deformed Hodge filtration on the periodized de Rham complex define Poincar\'e dual mixed $\RR$-Hodge structures $\KR{\X,\ps}$ and $\KRc{\X,\ps}$. 
\end{mainthm}

\subsection{Lattices}
To obtain the full data of a mixed Hodge structure, one requires, in addition, an integral lattice.  For this, there are several possibilities: the de Rham isomorphism from periodized singular cohomology; the Chern character $\ch$ from topological K-theory $\KB{\X}$; and (variants of) the Caldararu--Mukai ``charge'' vector
\[
c_\X := \Ahat_\X^{1/2} \ch : \KB{\X}\to \KdR{\X}\,,
\]
which corrects the failure of the Chern character to respect Poincar\'e duality. Here $\Ahat^{1/2}_\X \in \KdR[0]{\X} \cong \bigoplus_j \HdR[2j]{\Xo}$ is the square root of the $\Ahat$-class of the interior $\Xo$.

The de Rham and Chern lattices are the easiest to work with from the perspective of geometry, in that they are functorial for pullbacks along arbitrary morphisms of log Poisson manifolds. This  naive functoriality fails in an interesting way for the charge lattice, but the latter is the most natural lattice from the physical/noncommutative perspective: it corresponds to the charge of D-branes \cite{CY98,MRM97}, and in our subsequent work, we will see that it is functorial for quantizable Fourier--Mukai transforms, which will allow us to use our Hodge structures to construct explicit, computable obstructions to deformations of morphisms and modules.  This motivates an examination of its functoriality in purely classical terms, which we carry out in \autoref{sec:lattices} and summarize as follows.

Let $\K{\X,\ps}$ denote the mixed Hodge structure obtained by equipping the mixed $\RR$-Hodge structure $\KR{\X,\ps}$ with the charge lattice.  We say that a morphism of log Poisson manifolds $(\X,\ps) \to (\Y,\eta)$ is \defn{K-quantizable} if the pullback in K-theory defines a morphism of mixed Hodge structures.  Thus, by definition, the assignment $(\X,\ps) \mapsto \K{\X,\ps}$ is contravariantly functorial for K-quantizable morphisms. We show that \'etale maps, immersions and blowups of coabelian submanifolds, coabelian vector bundles and their projectivizations are all K-quantizable.   This allows us to extend many of the classical tools for calculation of mixed Hodge structures on cohomology to our setting, which we collect in the following.
\begin{mainthm}[see \autoref{sec:lattices}]\label{thm:mainB}
The functor $(\X,\ps) \mapsto \K{\X,\ps}$ on the category of K-quantizable morphisms has the following properties:
\begin{enumerate}
\item It is invariant under weak equivalences and multiplication by classes of coabelian Poisson vector bundles.
\item It satisfies Bott periodicity, namely $\K{\X,\ps} \cong \K[\bullet+2]{\X,\ps}(1)$ where $(-)(1)$ is the Tate twist.
\item It is compatible with the classical Gysin sequence, blowup formula, and projective formula, provided the vector bundles and submanifolds involved are coabelian.
\end{enumerate}
\end{mainthm}

\subsection{Period maps and quantum parameters}\label{sec:intro-quantum-param}
In \autoref{sec:periods}, we consider the behaviour of the mixed Hodge structure $\K{\X, \ps}$ as the log Poisson manifold $(\X,\ps)$ varies in a holomorphic family, i.e. we study the \defn{period map}
\[
\mapdef{\per}{\MPois}{\MHodge}
{(\X,\ps)}{\K{\X,\ps}},
\]
from the moduli space of log Poisson manifolds to the moduli space of mixed Hodge structures.  We shall not make explicit use of these moduli spaces in the body of the paper, where we refer instead to general holomorphic families $(\X_s,\ps_s)_{s \in \bS}$ parameterized by a complex analytic space $\bS$.  But we shall stick with this suggestive picture in the introduction, in order to better convey the key ideas.

First off, a straightforward analogue of the classical argument for (generalized) complex manifolds~\cite{Baraglia14, Griffiths68}, yields a formula for the differential of $\wp$ in terms of the Kodaira--Spencer map and the contraction of polyvectors into forms, which implies that $\wp$ is truly a period map in the Hodge-theoretic sense:
\begin{mainthm}[see \autoref{prop:InfPeriodMap} and \autoref{rmk:moduli}]\label{thm:mainC}
    The period map $\wp$ is holomorphic and satisfies Griffiths' transversality condition, i.e.~defines a variation of mixed Hodge structures. 
\end{mainthm}

Hence, it is natural to consider the following \defn{global Torelli problem}:
\begin{quote}
    To what extent is a log Poisson manifold $(\X,\ps)$ determined by its period $\wp(\X,\ps)$?
\end{quote}

The period map $\wp$ is very far from being injective, even if we restrict our attention to Poisson structures on a fixed underlying log manifold $\X$.  For instance, if $\X = \PP^n$ is a projective space (viewed as a log manifold with empty boundary divisor), the space of Poisson structures on $\X$ has many irreducible components distinguished by the topology of the symplectic foliations, but the mixed Hodge structures $\K[0]{\PP^n,\ps} \cong \ZZ^{n+1}$ and $\K[1]{\PP^n,\ps} = 0$ are trivial, so that $\per(\PP^n,\ps)$ is independent of $\ps$; see \autoref{cor:P^n-hodge}.  However, that framing is too defeatist: for many of these components, there is a canonically associated family of log Poisson manifolds \emph{with nontrivial boundary}, modelling the complements of certain subvarieties of $\PP^n$ determined by the topology of the foliation.  For these associated families, the period map is a covering onto its image that identifies two log Poisson manifolds if and only if  certain  cohomology classes are integral.

Hence, the best one  can hope for is to find components of $\MPois$ on which $\wp$ is an immersion, whose fibres are characterized in explicit topological terms.  As in the classical situation, one easy way to guarantee immersivity is the Calabi--Yau condition.  Indeed, the following result is a straightforward adaptation of results of Baraglia~\cite{Baraglia14} and Katzarkov--Kontsevich--Pantev~\cite{Katzarkov2008}; we merely take the opportunity to make it explicit:
\begin{mainthm}[\cite{Baraglia14,Katzarkov2008}; see \autoref{ex:logCY}]\label{thm:mainD}
    The period map $\wp$ is an immersion over the locus of log Calabi--Yau Poisson manifolds, i.e.~those whose  underlying compactification divisor $\bdX \subset \cX$ is anticanonical.
\end{mainthm}

In general, the characterization of the fibres seems to be subtle.  We initiate the study by considering here the special case in which $\Ahat_\X = 1$; this condition, while restrictive, is satisfied in many natural and interesting examples, and corresponds to a special symmetry of the period map.  Namely, $\MPois$ carries a natural action of the multiplicative group $\Gm$ by rescaling the Poisson bivector, while $\MHodge$ carries an action of the multiplicative monoid $\ZZ\setminus \{0\}$ by the Adams operations on topological K-theory.  The triviality of $\Ahat_\X$ is equivalent to the statement these actions on $\MPois$ and $\MHodge$ are suitably intertwined by the period map $\wp$, in a neighbourhood of the point $(\X,\ps=0) \in \MPois$.

We develop the basic linear algebra of such ``Adams--equivariant'' mixed Hodge structures, relating them to filtered mixed Hodge structures by a variant of the Rees construction, so that taking the  limit $\lim_{\hbar \to 0}\hbar \ps$ corresponds to passing to an associated graded of $\K{\X,\ps}$.  This mechanism, together with Carlson's theory~\cite{Carlson1980} of extensions of mixed Hodge structures,  allows us to associate, to each log manifold $\X$ with trivial $\Ahat$-class, a complex abelian Lie group $\Q{\X}$, and to each Poisson structure $\ps$ on $\X$ an element
\[
q(\ps) \in \Q{\X}\,,
\]
which we call the \defn{quantum parameter}; see \autoref{sec:qparam} for the explicit construction in terms of operations on cohomology. It depends only on the mixed Hodge structure $\K{\X,\ps}$, and the filtration induced by the Adams action; sometimes the latter is just a re-indexing of the weight filtration, hence completely determined by the period $\per(\X,\ps)$, but in general, it is more information.

The global structure of $\Q{\X}$ is determined by the Hodge numbers of $\X$.  In many interesting examples, it is either a compact torus $\CC^n/\Lambda$ or an affine algebraic torus $(\CCx)^n$, e.g.~this holds under suitable purity assumptions on the cohomology of $\X$; see \autoref{sec:purity}. Meanwhile, the behaviour of the quantum parameter is characterized by \autoref{thm:qparam-torelli}, which we paraphrase as follows:
\begin{mainthm}[see \autoref{thm:qparam-torelli}]\label{thm:mainE}
Let $\X$ be a log manifold with $\Ahat_\X=1$.  Then for every Poisson structure $\ps$ on $\X$, the map $\hbar \mapsto q(\hbar \ps)$ for $\hbar \in \CC$ defines a one-parameter subgroup of $\Q{\X}$.  Two Poisson bivectors $\ps,\ps'$ on $\X$ have the same quantum parameter if and only if the contraction operators $\hookps,\hook{\ps'}$ differ by an integral endomorphism of K-theory.
\end{mainthm}

This result means that, for fixed $\X$ with $\Ahat_\X=1$, the quantum parameter  is an ``exponential function'' of $\ps$.  In \autoref{sec:examples-of-periods}, we show by way of example that this function can be readily computed in terms of classical Poisson geometry in several cases of interest, mirroring the parameters arising in known deformation quantizations of these structures, thus justifying the name.

\autoref{thm:mainE} effectively reduces the global Torelli problem to the classical one for the underlying log manifolds (when the $\Ahat$ class is trivial): up to some discrete information, the period $\wp(\X,\ps)$ is determined by the integral Hodge structure $\coH{\X;\ZZ}$ and the quantum parameter $q(\ps)$, and the behaviour of the latter is understood.

\subsection{Quantum tori}
In the final section, as a preview of our subsequent work, we consider the simplest case in which the period map can be used to extract a nontrivial result in deformation quantization, by computing the canonical deformation quantization of the ``log canonical'' bracket $\{x,y\} = xy$ and its higher-dimensional counterparts, following the outline given by Kontsevich in \cite{Kontsevich2008a}. As stated above, the canonical quantization depends implicitly on a choice of Drinfel'd associator, or more precisely, a stable formality morphism for Hochschild cochains, in the sense of \cite{Dolgushev2021}, and we will address this dependence as well.

Note that the log canonical bracket defines an invariant Poisson bracket on the torus group $(\Gm)^2$ with algebra of functions $\cO{}((\Gm)^2) = \CC[x^{\pm1},y^{\pm 1}]$; it has a natural logarithmic model, given by any toric compactification.  Since the quantization formula is equivariant for linear changes of coordinates, it follows easily that it gives an associative star product
\[
\star : \cO{}((\Gm)^2) \times \cO{}((\Gm)^2) \to \cO{}((\Gm)^2)[[\hbar]]
\]
which on generators takes the form
\[
x \star y =w(\hbar) xy \qquad\qquad\qquad y \star x = w(-\hbar) xy
\]
for some formal power series $w(\hbar) \in \CC[[\hbar]]$ whose exact form depends on the choice of stable formality morphism.  It is thus determined by a quantum-torus type relation
\begin{align*}
x \star y &= q(\hbar) y \star x  & q(\hbar) &:= \frac{w(\hbar)}{w(-\hbar)},
\end{align*}
where the parameter $q(\hbar)$ is again a formal power series, which determines the star product up to isomorphism. 

The problem is therefore to compute $q(\hbar)$.  The brute force method is to first compute $w(\hbar)$ directly using the Feynman expansion.  Using the software from \cite{Banks2020} we can compute the first handful of terms for some different choices of stable formality morphism. If we use Kontsevich's original formula (corresponding to the Alekseev--Torossian associator), we find
\[
w(\hbar) = 1 + \frac{\hbar}{2}  + \frac{\hbar^2}{24} - \frac{\hbar^3}{48}  - \frac{\hbar^4}{1440}  + \frac{\hbar^5}{480} + \left(\frac{251  \zeta(3)^{2}}{2048 \pi^{6}} -  \frac{17}{184320} \right)\hbar^6 + \cdots\,.
\]
where $\zeta(3) = \sum_{n \ge 1}\frac{1}{n^3}$ is a Riemann zeta value.   Meanwhile, using the logarithmic formality morphism of \cite{Alekseev2016,Kontsevich1999} (corresponding to the Knizhnik--Zamolodzhikov associator) we find a different result:
\[
w(\hbar) =1 + \frac{\hbar}{2} + \frac{\hbar^2}{24} - \frac{\hbar^3}{48} - \frac{13\hbar^4}{5760}  + \frac{\hbar^5}{768} +  \frac{505\hbar^6}{4032} + \cdots\, .
\]
Despite the lack of an obvious pattern in the coefficients, and the presence of the conjecturally transcendental number $\zeta(3)^2/\pi^6$, when we compute the parameter $q(\hbar)$, we find the beginnings of a familiar series:
\begin{align}
q(\hbar)  =  1+ \hbar + \frac{\hbar^2}{2} + \frac{\hbar^3}{3!} + \frac{\hbar^4}{4!}  + \frac{\hbar^5}{5!} +  \frac{\hbar^6}{6!} + \cdots \,.\label{eq:q-parameter}
\end{align}
This is as far as the computer can currently take us in a reasonable amount of time, due to the factorial growth in the number of graphs.  

Intriguingly, the formula for $w(\hbar)$ is sensitive not just to the choice of formality morphism, but also to the dimension of the torus.  For instance, if we consider a three-dimensional torus with coordinates $(x,y,z)$ and Poisson bracket of the form
\[
\{x,y\} = xy \qquad \{y,z\} = a yz \qquad \{z,x\} = bzx
\]
for some constants $a,b$, we still have $x \star y = w(\hbar) xy$ for some series $w(\hbar)$, but now it depends on both of the additional parameters $a$ and $b$, e.g.~using Kontsevich's original formula we find
\begin{align*}
w(\hbar) &= 1 + \tfrac{1}{2} \hbar +  \tfrac{ab+1}{24}\hbar^2  - \tfrac{1-ab}{48}\hbar^3 \\
&\ \ \ - \tfrac{16 - 54ab + 11(a^2+b^2)+37(a^3+b^3)-34(a^2b+ab^2)+7(a^3b+ab^3)-36a^2b^2}{23040}\hbar^4 + \cdots .
\end{align*}
Nevertheless, we still have $x \star y = q(\hbar) y \star x$ with $q(\hbar)$ as in \eqref{eq:q-parameter}.

Despite all these subtleties, the Hodge-theoretic approach outlined by Kontsevich puts the problem to rest, as follows.  In \autoref{sec:qtori}, we construct, for each noncommutative torus
\begin{align*}
A_q &:= \frac{\CC\abrac{x^{\pm1},y^{\pm 1}}}{(xy-q\,yx)} & q &\in \Gm
\end{align*}
a canonical mixed Hodge structure $\K{A_q}$, giving a variation of mixed Hodge structures over the parameter space $q \in \Gm$.  We establish a global Torelli property: $A_q$ is determined up to isomorphism by $\K{A_q}$. Replacing $q\in \Gm$ with the formal series $q(\hbar)\in\CC[[\hbar]]$ obtained from canonical quantization as above, we obtain a variation $\K{A_{q(\hbar)}}$ over the formal affine line $\hbar \in \widehat{\bA^1}$.  We prove it is isomorphic to the variation $\K{(\CCx)^2,\hbar \ps}$, provided that the quantization is performed using a stable formality morphism that is compatible with the Getzler's Gauss--Manin connection on cyclic chains~\cite{Getzler93}.  In particular, this condition holds for Kontsevich's original formula by the methods of \cite{CFW11,DTT08,DTT09,Shoikhet03,Tamarkin2001,Tsygan99,WWChains}, and is conjectured to hold in complete generality; see \autoref{rmk:grt}.

As a result, we identify the parameter $q(\hbar)$ of the quantization with the quantum parameter of the Poisson structure in the sense of \autoref{sec:intro-quantum-param}, yielding the following folklore result that was sketched by Kontsevich~\cite{Kontsevich2008a} and worked out in detail in the first author's MSc thesis~\cite{Lindberg2020}:
\begin{mainthm}[see \autoref{thm:q=e^h} and \autoref{rmk:grt}] \label{thm:mainF}
The canonical quantization of $\{x,y\} = xy$ is determined up to isomorphism by the quantum parameter
\begin{align*}
q(\hbar) = q(\hbar\sigma) = e^\hbar \in \CC[[\hbar]]^\times,\label{eq:torus}
\end{align*}
for any stable formality morphism compatible with the Gauss--Manin connection on cyclic chains.  Similar results hold for Poisson tori of higher dimension.
\end{mainthm}

Note that \autoref{thm:mainF} implies a sort of convergence of the star product: it is gauge equivalent to an associative product on the vector space $\cO{}(\Gm)=\CC[x^{\pm1},y^{\pm 1}]$ that is defined by an entire function of the parameter $\hbar$. It also gives a novel Hodge-theoretic interpretation for the well-studied condition that $q$ is a root of unity: as we explain in \autoref{sec:qtori-hodge-conj}, it exactly corresponds to the existence of nontrivial Hodge classes, giving a nontrivial instance of the noncommutative Hodge conjecture~\cite{Lin2023}.

\subsection{Acknowledgements} We are grateful to many people for interesting conversations and correspondence (even brief ones) over the years that have helped shaped the contents of this paper, including Arend Bayer, Pieter Belmans, Cl\'ement Dupont, George Elliott, Robert Friedman, Ezra Getzler, Phillip Griffiths, Marco Gualtieri, Richard Hain, Nigel Hitchin, Andrei Konovalov, Maxim Kontsevich, Alexander Kupers, Mykola Matviichuk, Tony Pantev, Sam Payne, Pavel Safronov, Travis Schedler, Nick Sheridan, Goncalo Tabuada, Michel Van den Bergh and Thomas Willwacher.  Special thanks are due to the participants of the noncommutative Hodge theory seminar held in Edinburgh in 2018; to Theo Raedschelders and Sue Sierra for extensive discussion about Hodge theory of Sklyanin algebras; to Alejandro Cabrera, whose careful reading of the first author's related MSc thesis led to several improvements; to Peter Banks and Erik Panzer for their collaboration in \cite{Banks2020}, without which this project never would have started; and to Nicole Zolkavich for catching innumerable typos in earlier drafts of this paper.  A.~L.\ was supported by a Doctoral Scholarship (CGS-D) from the Natural Sciences and Engineering Research Council of Canada and a FAST Fellowship from the University of Toronto.  B.~P.\ was supported by a Discovery Grant from the Natural Sciences and Engineering Research Council of Canada, a New university researchers startup grant from the Fonds de recherche du Qu\'ebec -- Nature et technologies (FRQNT), and a startup grant at McGill University.

%% file: poisson-mhs-log-varieties.tex
\section{Log Poisson manifolds}
\label{sec:poisson}

\subsection{Log manifolds}  We shall work in the following category:

\begin{definition}
A \defn{log manifold} is a pair $\X=(\cX,\bdX)$ where $\cX$ is a compact K\"ahler  manifold called the \defn{compactification of $\X$}, and $\bdX\subset \cX$ is a normal crossings divisor called the \defn{boundary of $\X$}. The \defn{interior of $\X$} is the open submanifold $\Xo := \cX \setminus \bdX$.
\end{definition}

\begin{definition}
If $\X$ and $\Y$ are log manifolds, a \defn{morphism} $\phi : \X \to \Y$ is a holomorphic map $\overline{\phi} : \cX \to \cY$  that respects the interiors, in the sense that  $\overline{\phi}(\Xo) \subset \Yo$. Equivalently, the divisor $\phi^*\bdY$ is supported on $\bdX$.  A morphism $\phi$ is a \defn{weak equivalence} if the induced map on interiors $\phi^\circ := \overline{\phi}|_{\Xo} : \Xo \to \Yo$  is an isomorphism.
\end{definition}

\begin{remark}
    Here and throughout, by a K\"ahler manifold, we mean a complex manifold that admits a K\"ahler metric.  The metric and symplectic structure play no role whatsoever.
\end{remark}

\begin{remark}
Every compact K\"ahler manifold can be viewed as a log manifold with empty boundary divisor, in which case a morphism is just a holomorphic map, and a weak equivalence is just a biholomorphism.  We make no notational distinction between such a manifold $\X$ and its associated log manifold $(\X,\varnothing)$. 
\end{remark}

One should think of a log manifold $\X$ as a ``nice'' model for its interior $\Xo$, hence the notion of weak equivalence.  We may also seek nice models for submanifolds:

\begin{definition}
    If $\X$ is a log manifold, an \defn{(embedded) log submanifold} of $\X$ is a morphism of log manifolds $\Y \to \X$ whose underlying map $\cY\to\cX$ is a closed embedding.
\end{definition}

\begin{example}
    Let $\cZ \subset \cX$ be a closed submanifold such that $\bdX \cap \cZ \subset \cZ$ is a normal crossing divisor in $\cZ$, and let $\Z := (\cZ,\bdX\cap \cZ)$.  Then the inclusion gives a log submanifold $\Z\hookrightarrow \X$ modelling the inclusion of the interiors $\Zo\hookrightarrow \Xo$.  We call such a submanifold a \defn{closed log submanifold}.
\end{example}

\begin{example}
    The identity map on $\cX$ gives a morphism $\X \to \cX$, making $\X$ a log submanifold of $\cX$. It models the inclusion $\Xo \hookrightarrow \cX$.
\end{example}

\subsection{(Poly)vector fields and forms}\label{sec:calc}
The \defn{tangent sheaf} of a log manifold $\X = (\cX,\bdX)$, denoted
\[
\cT{\X} = \cTlog{\X} \subset \cT{\cX}
\]
is the sheaf of holomorphic vector fields on $\cX$ that are tangent to $\bdX$.  Note that $\cT{\X}|_{\Xo}$ is just the tangent sheaf of the interior. The \defn{polyvector fields on $\X$} are the exterior algebra
\[
\der{\X} := \wedge^\bullet \cT{\X}\,.
\]
Equipped with the wedge product and Schouten bracket, they form a sheaf of Gerstenhaber algebras.  

Dually,  we denote by 
\[
\forms{\X} := \wedge^\bullet \cT{\X}^\vee \cong \forms{\cX}(\log \bdX) 
\]
the sheaf of differential forms on $\cX$ with at worst logarithmic poles on $\bdX$.  Equipped with the wedge product and the de Rham differential, they form a sheaf of commutative differential graded algebras.  The \defn{de Rham cohomology of $\X$} is the hypercohomology $\HdR{\X} := \coH{\rsect{\forms{\X}}}$.  It is functorial for pullbacks along morphisms of log manifolds, and restriction to the interior gives a natural isomorphism $\HdR{\X}\cong\HdR{\Xo}$ with the usual de Rham cohomology of the interior; hence it is invariant under weak equivalence.

Note that $\forms{\X}$ forms a module over $\der{\X}$ by the contraction operator $\iota$.  For the avoidance of doubt about signs: the convention is that if $\xi \in \cT{\X}$ and $\alpha \in \forms{\X}$, then $\hook{\xi}\alpha = \alpha(\xi,-,\ldots,-)$ is the contraction in the first slot of $\alpha$, and the action of higher degree polyvectors is determined by $\hook{\xi\wedge\eta} = \hook{\xi}\hook{\eta}$.  

\subsection{Logarithmic Poisson structures}

Recall that a (holomorphic) Poisson structure on a complex manifold $\Y$ is a Poisson bracket $\{-,-\} : \cO{\Y} \times \cO{\Y} \to \cO{\Y}$ on its sheaf of regular functions, or equivalently a global section $\eta \in \coH[0]{\der[2]{\Y}}$ such that  $[\eta,\eta] = 0$.  A Poisson subvariety is a locally closed analytic subspace $\Z\subset \Y$ such that the bracket on $\cO{\Y}$ descends to $\cO{\Z}$. Every Poisson manifold has a foliation by symplectic leaves, which are maximal immersed analytic submanifolds whose tangent space spanned by the image of the induced map $\eta^\sharp : \forms[1]{\Y}\to\cT{\Y}$.  Note that the leaves need not be locally closed. For a thorough introduction, see e.g.~\cite{Dufour2005,LGPV2013,Polishchuk1997}.

\begin{definition}
A \defn{Poisson structure} on a log manifold $\X$ is a Poisson structure $\ps$ on the compactification $\cX$ such that the boundary divisor $\bdX$ is a Poisson subvariety.  A \defn{log Poisson manifold} is a pair $(\X,\ps)$, where $\X$ is a log manifold and $\ps$ is a Poisson structure on $\X$.
\end{definition}

There are many equivalent formulations of this definition that are useful:

\begin{lemma}\label{lem:log-poisson-equiv}
Let $\X$ be a log manifold and let $\{-,-\}$ be the Poisson structure on $\cX$ corresponding to a Poisson bivector $\ps \in \coH[0]{\der[2]{\cX}}$.  Then the following statements are equivalent:
\begin{enumerate}
\item $\{-,-\}$ defines a Poisson structure on $\X$.
\item The Poisson bivector $\ps$ is tangent to $\bdX$.
\item $\ps \in \coH[0]{\der[2]{\X}}$ is a logarithmic bivector.
\item\label{it:leaf-equiv} Every symplectic leaf of $\Xo$ is also a symplectic leaf of $\cX$.
\end{enumerate}
\end{lemma}

\begin{proof}
    The equivalences $(1) \Leftrightarrow (2) \Leftrightarrow (3)$ are well-known. For the equivalence of (2) and (4), note that the symplectic leaves partition $\cX$, so (4) is equivalent to the statement that the leaf through every point $p \in \bdX$ is entirely contained in $\bdX$, which is equivalent to the statement that $\ps$ is tangent to $\bdX$.
\end{proof}

\begin{definition}	
	A \defn{morphism of log Poisson manifolds} $(\X,\ps) \to (\Y,\eta)$ is a Poisson map $\cX \to \cY$ that defines a morphism of log manifolds $\X \to \Y$; it is a \defn{weak equivalence} if the map $\X\to\Y$ is a weak equivalence, i.e.~the map of interiors $\Xo\to\Yo$ is a Poisson isomorphism.  Two log Poisson manifolds $(\X,\ps)$ and $(\X',\ps')$ are \defn{weakly equivalent} if there exists a zig-zag of weak equivalences of log Poisson manifolds
 \[
 \begin{tikzcd}
 (\X,\ps) & (\X_1,\ps_1)\ar[l] \ar[r] & \cdots & (\X_j,\ps_j)\ar[l]\ar[r] & (\X',\ps')
 \end{tikzcd}
 \]
 for some $j \ge 0$.
\end{definition}

\begin{definition}
Let $(\Y,\eta)$ be a holomorphic Poisson manifold.  A \defn{logarithmic model for $\Y$} is a log Poisson manifold $(\X,\ps)$ manifold together with a holomorphic Poisson isomorphism $(\Xo,\ps|_{\Xo})\cong (\Yo,\eta)$.  
\end{definition}

\begin{example}[\toricex]\label{ex:toric-def}
Let $\X$ be a \defn{toric log manifold}, by which we mean a smooth projective toric variety $\cX$, equipped with its toric boundary divisor $\bdX \subset \cX$ (which is automatically normal crossings).  It is a log model for an affine algebraic torus $\Xo \cong (\Gm)^n$.   Let $\ft$ be the Lie algebra of the torus.  Then the infinitesimal action map  $\ft \to  \cT{\cX}$ identifies $\cT{\X}\subset \cT{\cX}$ with the trivial bundle $\ft\otimes\cO{\cX}$.  Since $\coH[0]{\cO{\cX}} = \CC$ and $\coH[q]{\cO{\cX}} = 0$ for $q>0$, we have a canonical isomorphism
\begin{align*}
\coH[q]{\der{\X}} &\cong \begin{cases}\wedge^\bullet \ft \ & q = 0  \\ 0 & q > 0 \end{cases}\, ,
\end{align*}
and since $\ft$ is abelian, the Schouten bracket is identically zero.
Thus a Poisson structure on $\X$ is equivalent to an element $\ps \in \coH[0]{\der[2]{\X}} \cong \wedge^2 \ft$.  Concretely, if $x_1,\ldots,x_n : \Xo \cong (\Gm)^n \to \Gm$ are toric coordinates (i.e.~a basis of the character group), then the torus action is generated by the vector fields $\logcvf{x_1},\ldots,\logcvf{x_n}$, which extend uniquely to the compactification $\cX$, and a Poisson structure has the form
\[
\ps = \sum_{i,j} \lambda_{ij} x_ix_j \cvf{x_i}\wedge\cvf{x_j}
\]
for a skew-symmetric matrix $(\lambda_{ij})_{i,j} \in \CC^{n\times n}$, i.e.~the Poisson bracket is given by
\[
\{x_i,x_j\} = \lambda_{ij}x_ix_j
\]
for all $i,j$.
\end{example}

\begin{example}[\logsympex]\label{ex:log-symp-def}
A \defn{log symplectic manifold in the sense of Goto}~\cite{Goto2002} is a triple $(\Y,\Z,\omega)$ where $\Y$ is a complex manifold, $\Z\subset\Y$ is a hypersurface and $\omega \in \coH[0]{\forms[2,cl]{\Y}(\log \Z)}$ is a global logarithmic form that is closed and non-degenerate, in the sense that its top power trivializes the log canonical bundle $(\det\forms[1]{\Y})(\Z)$. The inverse $\ps = \omega^{-1}$ is then a Poisson structure on $\Y$ that is tangent to $\Z$.  When $\Y$ is compact K\"ahler and $\Z$ is normal crossings, we obtain a log Poisson manifold in the sense of the present paper; see \autoref{ex:log-symp-blowup} below for a hint of what to do when $\Z$ is more singular.
\end{example}

In view of \autoref{lem:log-poisson-equiv} part (\ref{it:leaf-equiv}), symplectic leaves are treated as fundamental building blocks in the theory: we are not allowed to modify them in the course of compactification.  Consequently, not every holomorphic Poisson manifold admits  a logarithmic model, an obvious class of counterexamples being given by artificially removing subvarieties that are not Poisson:

\begin{example}
Let $\Y$ be a K3 surface equipped with the Poisson structure given by a holomorphic symplectic structure, and let $\Z \subset \Y$ be a non-empty closed subvariety of positive codimension.  Note that $\Z$ is not a Poisson subvariety, but the open complement $\Y \setminus \Z$ is.  Since K3 surfaces are the unique minimal representative of their birational class, the only possible compactifications of $\Y\setminus \Z$ as a complex manifold are iterated blowups of $\Y$ along the points of $\Z$. But the pullback of a two-form along a blowup vanishes on the exceptional divisor, so the pullback of the bivector $\ps$ always has poles. Hence the only possible compactification of $\Y\setminus \Z$ as a Poisson manifold is $\Y$ itself, for which the boundary $\Z$ is not a Poisson subvariety.  Therefore $\Y\setminus \Z$ does not admit a logarithmic compactification.
\end{example}

\subsection{Building logarithmic models via blowups}\label{sec:compactifications}

If $\X$ is a compact K\"ahler Poisson manifold, and $\Z\subset \X$ is any closed Poisson subvariety, then trivially $\X$ gives a Poisson compactification of $\X\setminus \Z$ for which the boundary is Poisson. If $\Z$ is not a hypersurface, this compactification is not logarithmic, and it is natural to try to rectify this by forming the blowup $b : \Bl{\Z}{\X} \to \X$.  We then have an isomorphism of $\X\setminus \Z \cong \Bl{\X}{\Z} \setminus \E$ where $\E := b^{-1}(\Z)$ is the exceptional divisor.   However, as explained by Polishchuk in \cite[\S8]{Polishchuk1997}, the Poisson structure on $\X$ need not lift to the blowup, and even if it does, the exceptional divisor $\E$ need not be a Poisson subvariety.    This leads us to restrict the class of Poisson submanifolds that we can remove, using the conditions in \emph{op.~cit.}, as follows.

\begin{definition}
	Let $\X$ be a log Poisson manifold and $\Z \subset \X$ a closed log Poisson submanifold. The \defn{conormal Lie algebra of $\Z$} is the conormal sheaf $\coN{\Z} := \cI/\cI^2$ where $\cI < \cO{\cX}$ is the defining ideal of $\cZ\subset \cX$, equipped with the Lie bracket induced by the Poisson bracket on $\cX$.
\end{definition}

Recall from \cite[\S8]{Polishchuk1997} that a finite-dimensional Lie algebra $\fg$ is called \defn{degenerate} if for every $x,y,z \in \fg$, the symmetric two-tensor $[x,y]z + [y,z]x+[z,x]y \in S^2\fg$ is equal to zero.  Equivalently, either $\fg$ is abelian or has a basis $e_1,\ldots,e_{n-1},f$ such that $[f,e_i] = e_i$ for all $i$ and $[e_i,e_j] = 0$ for all $i,j$.  

\begin{definition}
Let $\X$ be a log Poisson manifold.  A \defn{codegenerate} (respectively \defn{coabelian}) \defn{submanifold} is a closed log Poisson submanifold $\Z \subset \X$ for which the fibre of the conormal Lie algebra at every point is degenerate (resp.~abelian).
\end{definition}

\begin{example}
    Any Poisson hypersurface is a coabelian submanifold: the conormal Lie algebra is rank one, hence automatically abelian.
\end{example}

The following fundamental result is due to Polishchuk.
\begin{theorem}[{\cite[\S8]{Polishchuk1997}}]
Let $\X$ be a Poisson manifold  and let $\Z \subset \X$ be a closed Poisson submanifold.  Then the Poisson structure on $\X$ lifts to the blowup $\Bl{\X}{\Z}$ if and only if $\Z$ is codegenerate.  The exceptional divisor is then a Poisson submanifold if and only if $\Z$ is coabelian.
\end{theorem}

Note that when we blow up a closed log submanifold, the union of the exceptional divisor and the boundary divisor is a normal crossings divisor.  This allows us to make the following definition. 

\begin{definition}\label{def:blowup-and-complement}
Let $\X$ be a log Poisson manifold and let $\Z \subset \X$ be a closed log Poisson submanifold.
\begin{itemize}
\item If $\Z$ is codegenerate, then the \defn{blowup of $\X$ along $\Z$} is the log Poisson manifold $\Bl{\X}{\Z} := (\Bl{\cX}{\cZ},b^*\bdX)$.  
\item If, in addition, $\Z$ is coabelian, then the \defn{complement of $\Z$} is the log Poisson manifold $\X\setminus \Z := (\Bl{\cX}{\cZ}, b^*\bdX + \E)$.
\end{itemize}
\end{definition}
The nomenclature and notations $\Bl{\X}{\Z}$ and $\X \setminus \Z$ are chosen to be consistent with the usual notations for blowups and complements of manifolds upon passing to interiors.  Indeed, as log manifolds, both the blowup and the complement have the same compactification $\overline{\X\setminus \Z} = \overline{\Bl{\X}{\Z}} = \Bl{\cX}{\cZ}$, and there are natural morphisms $\Bl{\X}{\Z} \to \X$ and $\X\setminus \Z \to \X$ given by the usual blowdown  $\Bl{\cX}{\cZ}\to\cX$; the induced maps of interiors are the  blowdown $\Bl{\Zo}{\Xo} \twoheadrightarrow \Xo$ and the inclusion $\Xo\setminus \Zo \hookrightarrow \Xo$, respectively.  Moreover, in the coabelian case, the exceptional divisor gives a log Poisson hypersurface $\E\subset\Bl{\X}{\Z}$ and we have a canonical isomorphism of log manifolds $\Bl{\X}{\Z} \setminus \E \cong \X \setminus \Z$.

\begin{example}[\surfaceex]\label{ex:surface-constr}
Let $\X$ be a compact K\"ahler surface.  Then a Poisson structure $\ps$ is equivalent to a section of the anticanonical line bundle. Its vanishing set is an anticanonical divisor $\Y\subset \X$ which may be singular.   If $p \in \X$, then $p$ is codegenerate if and only if $p \in \Y$, and $p$ is coabelian if and only if $p \in \sing{\Y}$ is a singular point of $\Y$.  Hence by repeatedly blowing up the singular points, we obtain a log Poisson surface $(\X',\Y',\ps')$ that gives a  logarithmic model for the complement $\X\setminus \Y$.
\end{example}

\begin{example}[\sklyaninex]\label{ex:sklyanin-constr}
	Sklyanin~\cite{Sklyanin1982} defined a family of Poisson structures  on $\PP^3$ associated with elliptic curves.  For each such structure $\ps$, the closures of its two-dimensional symplectic leaves give a pencil of quadric surfaces whose base locus is an elliptic normal curve $\Z\subset \PP^3$.  The vanishing set of $\ps$ consists of the curve $\Z$ and four isolated points $p_1,\ldots,p_4$, given by the singular points of the singular quadrics in the pencil.  The curve $\Z$ is coabelian, so we may blow it up to form a canonical logarithmic model of the complement $\PP^3 \setminus \Z$.  However, the linearization of the Poisson structure at each of the four points $p_1,\ldots,p_4$ is isomorphic to $\mathfrak{sl}(2;\CC)$, so these points are not codegenerate, and can therefore not be blown up.
\end{example}

\begin{example}[More general threefolds]\label{ex:threefold-constr}
Let $(\X,\ps)$ be a compact K\"ahler Poisson threefold. There are two possibilities: either the bivector $\ps$ is nonvanishing, or it has at least one zero.

 If $\ps$ is non-vanishing, then all of its symplectic leaves have dimension two; such threefolds are classified by Druel in \cite{Druel1999}.  Note that the only nontrivial closed Poisson submanifolds are hypersurfaces given by unions of closed symplectic leaves.  Such hypersurfaces are automatically smooth, so they can be added as boundary divisors to obtain logarithmic models for their complements, with no need to blow up.  

On the other hand, if $\ps$ has a zero, then Druel's result shows that at least one irreducible component of the vanishing locus $\ps^{-1}(0)$ has positive dimension, i.e.~contains a curve.  (From the previous example of Sklyanin's structures, we see that there may  be additional isolated zero-dimensional components.) Let $\Z = \cup_i \Z_i$ be the union of the positive-dimensional components of $\ps^{-1}(0)$.  It is common to encounter examples where each $\Z_i$ is a smooth curve; for instance, according to the classification results in \cite{Cerveau1996,Loray2013,Pym2015}, this is the case when $\X$ is a Fano threefold of Picard rank one (such as $\PP^3$) and $\ps$ is generic,~i.e.~lies in a Zariski open dense set of the space of all Poisson structures on $\X$. In this situation, each $\Z_i$ is coabelian by \cite[Theorem 13.1]{Polishchuk1997}, so we can form a log model for its complement by blowing up.  If, in addition, the intersection points $\Z_i \cap \Z_j$ are coabelian and the tangent spaces of $\Z_i$ and $\Z_j$ are pairwise distinct, we may first blow up all intersection points $\Z_i \cap \Z_j$ and then blow up the strict transforms of all the components $\Z_i$ to obtain a log model for $\Y\setminus \Z$; this works for all generic Poisson Fano threefolds of Picard rank one, except those named ``Aff'' in \cite[Table 1]{Loray2013}.
\end{example}

\begin{example}[\logsympex]\label{ex:log-symp-blowup}
For most interesting log symplectic manifolds $(\Y,\Z,\omega)$ in the sense of Goto (see \autoref{ex:log-symp-def}), the polar divisor $\Z$ is not normal crossings. Sometimes, one can construct a log model for $\Y \setminus \Z$ by iteratively blowing up strata of $\Z$; this is the case, for instance, for Hilbert schemes of surfaces with a smooth anticanonical divisor~\cite{Bottacin1998,Ran2016} and for some of Feigin--Odesskii's elliptic Poisson structures~\cite{Feigin1989,Feigin1998}. However, for many singularity types, this does not work, and one instead needs to use \emph{weighted} blowups to produce a log model for $\Y\setminus \Z$, at the cost of working with orbifolds. This will be discussed in forthcoming work of the second author with Lapointe and Matviichuk.
\end{example}

\subsection{Vector bundles and projective bundles}\label{sec:coabelian-constructions}
All vector bundles in this paper are holomorphic unless otherwise stated.

Let $\X$ be a log manifold.  Given a vector bundle on the interior $\Eo \to \Xo$, we may seek to compactify it as follows.  First, we try to extend it to a vector bundle $\E' \to \cX$.  Such an extension may not exist, but if it does, we obtain a logarithmic compactification of $\Eo$ by adding a hyperplane bundle at infinity, i.e.~forming the projective bundle $\PP(\E'\oplus{\cO{\cX}})$ and identifying $\E'$ with the complement of the hyperplane bundle $\PP(\E')\subset \PP(\E'\oplus{\cO{\cX}})$.  This motivates the following definition.

\begin{definition}
Let $\X=(\cX,\bdX)$ be a log manifold.  A \defn{log vector bundle $\E\to\X$}  is a log manifold given by a projective bundle $\cE = \PP(\E'\oplus \cO{\cX})$ for some vector bundle $\E' \to \cX$, together with the divisor $\partial\E = \PP(\E') \cup \cE|_{\bdX}$ 
\end{definition}

Note that the $\Gm$-action on the vector bundle $\E^\circ$ extends to an action on the log vector bundle $\E$ by automorphisms that fix the hyperplane at infinity pointwise.  Moreover, the bundle projection and zero section give morphisms of log manifolds $p : \E \to \X$ and $i : \X \to \E$ such that $p\circ i = \id_{\X}$.

The vector fields on a log vector bundle $p:\E\to\X$ fit in an exact sequence
\begin{align*}
\xymatrix{
0 \ar[r] & \cT{\E/\X} \ar[r] & \cT{\E} \ar[r] & p^* \cT{\X} \ar[r] & 0
}
\end{align*}
of locally free sheaves on $\cE := \PP(\E'\oplus\cO{\cX})$, where $\cT{\E/\X}$ is the sheaf of vertical vector fields on the projective bundle $\cE$ that are tangent to the hyperplane at infinity. Note that such vector fields are exactly the ones that generate the affine transformations of the fibres of $\E'$, i.e.~translations and linear endomorphisms.  In particular $\cT{\E/\X}$ is independent of the boundary divisor $\bdX$.

Pushing forward to the base, we thus obtain a canonical exact sequence
\[
\xymatrix{
0 \ar[r] & \sEnd{\calE} \oplus \calE \ar[r] & p_*\cT{\E} \ar[r] &  \cT{\X} \ar[r] & 0\, .
}
\]
Concretely, in a local trivialization of $\E$ over an open set $\U\subset \cX$ with linear fibre coordinates  $y^1,\ldots,y^n$, a log vector field on $p^{-1}(\U)$ can be written uniquely in the form
\[
\xi = \xi_0 + \sum_i \rbrac{a^i + \sum_j b^{i}_{j}y^j}\cvf{y^i}\,,
\]
where $a^i,b^i_j \in \cO{\cX}$ are functions on the base and $\xi_0 \in \cT{\X}$ is a logarithmic vector field on $\X$.

We can construct polyvectors of higher degree on $\E$ by taking wedge products of such vector fields.  In fact this exhausts the possibilities:
\begin{lemma}
	The canonical map $\wedge^\bullet p_*\cT{\E} \to p_*\der{\E}$ is surjective.
\end{lemma}
\begin{proof}
The pair $(\PP(\E'\oplus\cO{\cX}),\PP(\E'))$ is a fibre bundle over $\cX$ with fibre $(\PP^n,H)$ where $H\subset \PP^n$ is a hyperplane.  Hence by the K\"unneth formula, the problem reduces to the case in which the base is a point, i.e.~to the statement that the canonical map $\wedge^\bullet \coH[0]{ \cT{\PP^n}(-\log H)} \to \coH[0]{\wedge^\bullet \cT{\PP^n}(-\log H)} $ is surjective. 
\end{proof}

Using this result, we may understand Poisson structures on log vector bundles as follows.
\begin{definition}
A \defn{log Poisson vector bundle} is a log vector bundle $\E \to \X$ equipped with a Poisson structure $\eta \in \coH[0]{\der[2]{\E}}$.  A \defn{codegenerate} (respectively \defn{coabelian}) \defn{vector bundle on $\X$} is a log Poisson vector bundle whose zero-section is a codegenerate (resp.~coabelian) submanifold.
\end{definition}

Thus, in a local trivialization with fibre coordinates $y^1,\ldots,y^n$ as above, the Poisson structure on a log vector bundle has the form
\[
\eta = \ps + (\alpha^i + \beta^{i}_j y^j) \cvf{y^i} + (a^{ij} + b^{ij}_k y^k + c^{ij}_{kl} y^ky^l) \cvf{y^i}\wedge\cvf{y^j}\,,
\]
where $\ps \in \der[2]{\X}$ is a log Poisson structure on the base, $\alpha^i, \beta^i_j \in \cT{\X}$ are log vector fields and $a^{ij},b^{ij}_k,c^{ij}_{kl} \in \cO{\cX}$ are  functions, and we have used the Einstein summation convention (summing over repeated indices). The further possible conditions on $(\E,\eta)$ then have the following explicit descriptions:
\begin{itemize}
\item $(\E,\eta)$ has the zero section as a Poisson submanifold if and only if $\alpha^i=a^{ij}=0$ for all indices $i,j$.
\item $(\E,\eta)$ is codegenerate if and only if $\alpha^i=a^{ij}=0$ for all $i,j$ and the Lie algebra with structure constants $b^{ij}_k$ is degenerate
\item $(\E,\eta)$ is coabelian if and only if $\alpha^i=a^{ij}=b^{ij}_k=0$ for all $i,j,k$, i.e.~it has the form
\[
\eta = \ps + \beta^{i}_j y^j \cvf{y^i} +  c^{ij}_{kl} y^ky^l \cvf{y^i}\wedge\cvf{y^j}\,,
\]
or equivalently, it is invariant under the action of $\Gm$ on the fibres.
\end{itemize}

\begin{example}
Let $\Y$ be a complex manifold.  Then its cotangent bundle $\ctb{\Y}$ carries a canonical symplectic (hence Poisson) structure. If $\X$ is a log model for $\Y$, we obtain a log Poisson vector bundle model for $\ctb{\Y}$ by taking $\E'$ to be the log cotangent bundle, i.e.~the vector bundle on $\cX$ associated to the locally free sheaf $\forms[1]{\X}=\logforms[1]{\X}$.  Note that the zero section is not a Poisson submanifold, so in particular, this log Poisson vector bundle is not codegenerate (hence also not coabelian). Also, the symplectic form has a pole of order greater than one at infinity, so it is \emph{not} a log symplectic structure in the sense of Goto.
\end{example}

\begin{example}\label{ex:log-line-bundle}
By the logarithmic counterpart of \cite[Proposition 5.2]{Polishchuk1997} a coabelian line bundle is equivalent to the data of an invertible sheaf $\sL \to \cX$ and an operator $\nabla : \sL \to \cT{\X}\otimes \sL$, making $\sL$ into a Poisson module.   In particular, the modular representation of \cite{Weinstein1997} makes the total space of the log canonical sheaf $\det \forms[1]{\X}$ a coabelian  line bundle on $\X$ in a natural way.  Similarly,  if $\mathsf{D}$ is any Poisson divisor on $\cX$, then $\cO{\cX}(\mathsf{D})$ is naturally a coabelian line bundle by \cite[Proposition 7.1]{Polishchuk1997}.
\end{example}

\begin{example}\label{ex:coabelian-normal}
If $(\X,\ps)$ is a log Poisson manifold and $\Z\subset \X$ is a coabelian submanifold, then the normal bundle $\N_{\Z} \to \Z$ naturally has the structure of a coabelian vector bundle.  Indeed, the sheaf of functions on the normal bundle is the symmetric algebra $\Sym{\coN{\Z}} \cong \bigoplus_{j \ge 0}\cI^j/\cI^{j+1}$, where $\cI\subset \cO{\cX}$ is the ideal defining $\cZ$.  Since $\{\cI,\cI\}\subset \cI^2$, the Poisson bracket on $\cO{\X}$ induces a canonical bracket on $\Sym{\coN{\Z}}$ that is homogeneous of weight zero, i.e.~a $\Gm$-invariant log Poisson structure on $\N_\Z$.  
\end{example}

For a coabelian vector bundle, the $\Gm$-invariance of the Poisson structure implies that it descends to a Poisson structure $\PP(\eta)$ on the projective bundle $\PP(\E')$, so that we may make the following construction.
\begin{definition}\label{def:projectivization}
Let $(\E,\eta) \to (\X,\ps)$ be a coabelian vector bundle.  Its \defn{projectivization} is the log manifold $\PP(\E) := (\PP(\E'),\PP(\cE')|_{\bdX})$ equipped with the induced Poisson structure $\PP(\eta)$.  Its \defn{tautological bundle} is the blowup $\cO{\PP(\E)}(-1) := \Bl{\E}{\X}$ along the zero section.
\end{definition}
Evidently, $\PP(\E)$ is a log model for the projective bundle $\PP(\Eo)\to \Xo$ and inherits many of the ``expected'' properties. For instance, $\cO{\PP(\E)}(-1)$ is a coabelian line bundle over $\PP(\E)$, and there is a natural commutative diagram of log Poisson manifolds
\[
\begin{tikzcd}
\PP(\E) \ar[rd] & \cO{\PP(\E)}(-1) \ar[l] \ar[d]\ar[r] & \E \ar[ld] \\
& \X
\end{tikzcd}
\]
which induces a weak equivalence $\cO{\PP(\E)}(-1)\setminus 0 \to \E \setminus 0$ between the complements of the zero sections in the sense of \autoref{def:blowup-and-complement}.  Similarly, if $\Z \subset \X$ is a coabelian submanifold, the exceptional divisor of the blowup $\Bl{\X}{\Z}$ is canonically identified, as a log Poisson manifold, with the projective bundle $\PP(\N_{\Z})$, giving the usual Cartesian diagram
\[
\begin{tikzcd}
\PP(\N_{\Z}) \ar[r]\ar[d] & \Bl{\X}{\Z} \ar[d] \\
\Z \ar[r] & \X
\end{tikzcd}\,.
\]

%% file: poisson-mhs-deRham.tex
\section{Hodge-de Rham theory}
\label{sec:deRham}

In this section, we describe the natural structure on the de Rham cohomology of a log Poisson manifold.  Namely, we will construct Hodge and weight filtrations on the two-periodized de Rham complex, and prove that the corresponding spectral sequences degenerate, giving a mixed $\RR$-Hodge structure.  We also establish Poincar\'e duality and study the action of characteristic classes of coabelian line bundles.  Our key tools are Kassel's notion of a mixed complex (originating from cyclic homology), Deligne's results on mixed Hodge theory, and the equivariance of the Hodge filtration  under rescaling of the Poisson bivector.

\subsection{Mixed complexes}
Our basic algebraic tool is the following notion, due to Kassel~\cite{Kassel1987}. Note that the literature typically uses homological grading conventions, but we use cohomological gradings throughout.

\begin{definition}
A (cohomologically graded) \defn{ mixed complex} is a triple  $(M,b,B)$ where $(M,b)$ is a cochain complex (called the \defn{underlying complex}), and $B : M \to M[-1]$ is a cochain map such that $B^2=0 : M \to M[-2]$.  We call $b$ the \defn{primary differential} and $B$ the \defn{secondary differential}.  A morphism of mixed complexes is a morphism of the underlying cochain complexes that intertwines the secondary differentials.    Such a morphism is a \defn{quasi-isomorphism} if the map of underlying complexes is a quasi-isomorphism.
\end{definition}

\begin{definition}
Let $(M,b,B)$ be a mixed complex.
\begin{itemize}
\item The \defn{cohomology of $M$} is the cohomology of the underlying complex $(M,b)$.
\item The \defn{periodic cyclic complex of $M$} is the complex of formal Laurent series $(M((u)),b+uB)$ where $u$ is a formal variable of degree two; its cohomology is the \defn{periodic (cyclic) cohomology of $M$}.
\end{itemize}
\end{definition}
\begin{remark}
    The variable $u$ is simply a bookkeeping device to keep track of the degrees and filtrations. 
    In noncommutative Hodge theory \cite{Katzarkov2008}, $u$ is reinterpreted as the coordinate on an auxiliary projective line via the Rees construction.
\end{remark}

The periodic cyclic complex comes equipped with subcomplexes
\[
\F[p]M((u)) = u^pM[[u]]
\]
given by bounding the power of $u$ from below, giving a decreasing filtration
\begin{align*}
M((u)) \supset \cdots \supset \F[p]M((u)) \supset \F[p+1]M((u)) \supset \cdots
\end{align*}
with associated graded given by ``turning off $B$'':
\[
\grF[] (M((u)),b+uB) \cong (M((u)),b)\,.
\]
Passing to cohomology, we obtain the decreasing \defn{Hodge filtration}
\[
\F[p]\coH{M((u))} = \img\rbrac{ \coH{\F[p]M((u))} \to \coH{M((u))}},
\]
and a spectral sequence 
\[
E_1 = \coH{M,b}((u)) \Rightarrow \grF \coH{M((u)),b+uB}\,.
\]
Multiplication by $u$ gives the \defn{Bott periodicity} isomorphism
\begin{equation}
\begin{tikzcd}
 \F[p]\coH[n]{M((u)),b+uB} \ar[r,"\sim"] & \F[p+1]\coH[n+2]{M((u)),b+uB}
\end{tikzcd} \label{eq:bott}
\end{equation}
for all $p,n \in \ZZ$.

Note that if $(\cM,b,B)$ is a sheaf of bounded below mixed complexes on a space, the operator $B$ induces an operator on the derived global sections $M := \rsect{\cM,b}$, making the latter into a mixed complex.  The cohomology of $M$ is then the hypercohomology of $\rsect{\cM,b}$ so that we have also have a hypercohomology spectral sequence
\[
E_1^{p,q} := \coH[q]{\cM^p} \Rightarrow \coH[p+q]{M}.
\]

\subsection{The mixed complex of a log Poisson manifold}

If $(\X,\ps)$ is a log Poisson manifold, we define a sheaf of mixed complexes
\[
\sMix{\X,\ps} := (\forms[-\bullet]{\X},\delps,\dd)\,,
\]
where $\forms[-\bullet]{\X} = \logforms[-\bullet]{\X}$ denotes the sheaf of logarithmic differential forms with the degrees reversed, so that $\forms[p]{\X}$ lies in cohomological degree $-p$.  The secondary differential $B$ is the de Rham differential $\dd$ and the primary differential $b$ is the  Poisson homology operator
\[
\delps := [\dd,\hookps] = \dd \hookps - \hookps \dd
\]
introduced by Brylinski~\cite{Brylinski1988} and Koszul~\cite{Koszul1985}.  Here $\hookps : \forms{\X} \to \forms[\bullet-2]{\X}$ is the contraction by the Poisson bivector.

\begin{remark}
The sign convention for $\hook{}$ is noted in \autoref{sec:calc}.  With this convention, if $x,y$ are coordinates on $\CC^2$, the bivector $\ps = \cvf{x}\wedge\cvf{y}$ has inverse the symplectic form $\omega = \ps^{-1} = \dd y \wedge \dd x$, and we have the identity $\hookps \omega = 1$.
\end{remark}

\begin{definition}
The \defn{mixed complex} of $(\X,\ps)$ is the derived global sections
\[
\Mix{\X,\ps} := \rsect{\sMix{\X,\ps}}.
\]
\end{definition}

Here we allow ourselves the freedom to choose a convenient resolution to compute the derived global sections; all constructions we make will be independent of this choice, up to quasi-isomorphism of mixed complexes.  For instance, we may use the \v{C}ech resolution associated to a good cover, or the Dolbeault resolution
\[
\forms[p]{\X} = \forms[p]{\cX}(\log\bdX) \hookrightarrow \dolb[p,\bullet]{\X} \,,
\]
where
\[
\dolb[p,q]{\X} := \dolb[0,q]{\cX} \otimes_{\cO{\cX}} \forms[p]{\cX}(\log \bdX)
\]
is the sheaf of $C^\infty$-forms of type $(p,q)$ whose holomorphic factors have at worst logarithmic poles on $\bdX$ and whose anti-holomorphic factors are smooth.  Thus, if $z_1,\ldots,z_k,y_1,\ldots,y_j$ are local holomorphic coordinates on $\cX$ such that $\bdX$ is the vanishing set of $y_1y_2\cdots y_j$, then $\dolb{\X}$ is generated as a $C^\infty_{\cX}$-algebra by the forms $\dd z_i,\dd y_\ell/y_\ell, \dd \zb_i$ and $\dd \yb_\ell$.  Since these sheaves are $C^\infty_{\cX}$-modules, they are acyclic, and hence we may take
\[
\Mix[n]{\X,\ps} := \Dolb[n]{\X}\,,
\]
where
\[
\Dolb[n]{\X} := \bigoplus_{q-p=n} \Dolb[p,q]{\X}
\]
is the space of Dolbeault forms (with log poles) corresponding to the $n$th column of the Hodge diamond.   We emphasize that the total degree is $n = q-p$, not $n=q+p$, since $\Mix{\X}$ reverses the degree of the holomorphic forms.

In this Dolbeault model, the primary differential on the complex $\Mix{\X,\ps}$ is identified with the operator
\[
\delb + [\del,\hook{\ps}]  : \Dolb[n]{\X} \to \Dolb[n+1]{\X}
\]
and the secondary differential is identified with the operator
\[
\del : \Dolb[n]{\X} \to \Dolb[n-1]{\X}
\]
for all $n \in \ZZ$.

In some cases, one can find an alternate model for the mixed complex that is more amenable to direct calculations.  Here are some particularly simple examples of this phenomenon:

\begin{example}[\toricex]\label{ex:toric-mixed}
Let $\X$ be a toric log manifold as in \autoref{ex:toric-def}. The trivialization $\cT{\X} \cong \ft \otimes \cO{\cX}$ dualizes to a trivialization $\forms[1]{\X}\cong \ft^\vee\otimes\cO{\cX}$, so that $\coH{\forms{\X}} \cong \wedge^\bullet \ft^\vee$, and moreover the contraction of global polyvectors into forms is identified with the contraction $\wedge^\bullet \ft \times \wedge^{\bullet}\ft^\vee \to \wedge^{\bullet} \ft^\vee$.  Note that all global forms on $\X$ are $\dd$-closed.  Hence the inclusion of the global sections in the derived global sections gives a quasi-isomorphism of mixed complexes
\begin{align*}
\rbrac{\wedge^{-\bullet} \ft^\vee,0,0} \cong  \Mix{\X,\ps} \label{eq:toric-mix-formal}
\end{align*}
for all Poisson structures $\ps$ on $\X$.
\end{example}

\begin{example}[\twopureex]\label{ex:two-pure-formality}
    More generally, we recall that every global logarithmic form on a log manifold is closed thanks to Deligne's mixed Hodge theory.  Hence we always have a canonical map of mixed complexes
    \[
    (\coH[0]{\forms[-\bullet]{\X}},0,0) \to \Mix{\X,\ps}\,.
    \]
    Usually this map is not an isomorphism, but sometimes it is.  Namely, this is the case when $\Y$ is a smooth algebraic Poisson variety whose cohomology satisfies the 2-purity condition (see \autoref{sec:purity} below) and $\X$ is any algebraic log model for $\Y$. 
 It follows that in this case, the mixed complex is independent of the compactification, i.e.~depends only on $\Y$.
\end{example}

\begin{example}[\logsympex]
Let $\X$ be a log manifold equipped with a log symplectic form $\omega \in \coH[0]{\forms[2]{\X}}$ and denote the corresponding Poisson structure $\ps = \omega^{-1}$.  This includes all compact hyperk\"ahler manifolds as the case $\bdX=\varnothing$, but could also come from a log model for a noncompact holomorphic symplectic manifold as in \autoref{ex:log-symp-blowup}.  

The mixed complex of $(\X,\ps)$ has the following ``Fourier dual'' description, due to Kontsevich~\cite[Theorem 1.36.1]{Kontsevich2008a}.
The bivector $\ps$ defines a nondegenerate pairing  $\abrac{-,-}_\ps$ on $\forms[k]{\X}$ by the formula $\abrac{\alpha_1\wedge\cdots\wedge\alpha_k,\beta_1\wedge\cdots\wedge \beta_k}_\ps = \det \ps(\alpha_i,\beta_j)_{ij}$.  Suppose that $\dim \X = 2n$, and define the symplectic Hodge star $* : \forms[k]{\X} \to \forms[2n -k]{\X}$, by
\[
\alpha \wedge * \beta = \tfrac{1}{n!}\abrac{\alpha,\beta}_\ps \omega^{n}
\]
for $\alpha, \beta \in \forms[k]{\X}$.  Using the identities in \cite[\S 2.1]{Brylinski1988}, it is straightforward to show that $*$ gives an isomorphism of sheaves of mixed complexes
\[
(\sMix{\X},\delps,\dd) = (\forms[-\bullet]{\X},\delps,\dd) \cong (\forms{\X}[2n],\dd,\delps)
\]
where $[2n]$ denotes a shift in degree by the dimension of $\X$.  Therefore using the Dolbeault resolution, we obtain a quasi-isomorphism of mixed complexes
\[
\Mix{\X,\ps}^\bullet \cong \bigoplus_{p+q=\bullet+\dim\X}\left(\Dolb[p,q]{\X},\del+\delb,\delps\right)
\]
identifying the primary differential with the $C^\infty$ de Rham differential $\del + \delb$.
\end{example}

\subsection{Logarithmic Poisson homology}

Using the mixed complex $\Mix{\X,\ps}$ we may define the following invariants:
\begin{definition}
Let $(\X,\ps)$ be a log Poisson manifold. 
\begin{itemize}
\item The \defn{logarithmic Brylinski--Koszul Poisson homology} is the cohomology of its mixed complex, denoted
\[
\KDol{\X,\ps} := \coH{\rsect{\sMix{\X,\ps},\delps}}\, .
\]
\item The \defn{logarithmic periodic cyclic Poisson homology} is the periodic cyclic cohomology of its mixed complex, denoted
\[
\KdR{\X,\ps} := \coH{\rsect{\sMix{\X,\ps}((u)),\delps+u\dd}} \,.
\]
\end{itemize}
\end{definition}

\begin{remark}
We use cohomological gradings throughout, but note that Poisson homology is usually given a homological grading, in which case the degree is reversed.
\end{remark}

The notation $\KDol{-}$ and $\KdR{-}$ is chosen because their role in K-theory is analogous to the role of Dolbeault and de Rham cohomology in classical Hodge theory.  In particular, when $\ps = 0$, we have
\[
\KDol[n]{\X} := \KDol[n]{\X,\ps=0} \cong \bigoplus_{q-p=n} \coH[q]{\forms[p]{\X}} = \bigoplus_{q-p=n} \coH[q]{\logforms[p]{\X}}
\]
so that $\KDol[n]{\X}$ amounts to a totalization of the usual bigrading on Dolbeault cohomology.  Similarly, setting $u=1$, we see that
\begin{align}
\KdR[n]{\X} := \KdR[n]{\X,\ps=0} \cong \bigoplus_{j\in\ZZ}\HdR[n+2j]{\X} \label{eq:K-vs-H}
\end{align}
is the two-periodization of the de Rham cohomology.  In fact, $\KdR[n]{\X,\ps}$ is canonically isomorphic to $\KdR[n]{\X}$ for all $\ps$:
\begin{lemma}\label{lem:exp-ps}
We have the identity
\[
e^{-\hookps/u} \cdot u\dd \cdot e^{\hookps/u} = \delps + u\dd
\]
as operators on $\rforms{\X}((u))$, and hence we have an isomorphism
\begin{align}
\exp(\hookps/u) : \KdR{\X,\ps} \to \KdR{\X} \cong \bigoplus_{j \in \ZZ} \HdR[\bullet + 2j]{\X}  \label{eqn:HP-compare}
\end{align}
of graded vector spaces (even as $\CC((u))$-modules).
\end{lemma}

\begin{proof}
	See \cite[Lemma 4.6]{Polishchuk1997} or \cite[\S 1.34]{Kontsevich2008a}.
\end{proof}

\subsection{The Hodge filtration}
For a log Poisson manifold $(\X,\ps)$, 
the Hodge filtration on the mixed complex $\Mix{\X,\ps}((u))$ gives rise to a filtration on its periodic cohomology $\KdR{\X,\ps}$, and hence on $\KdR{\X}$ via \autoref{lem:exp-ps}. Note, however, that the isomorphism \eqref{eqn:HP-compare} does not respect the Hodge filtrations, since the operator $\exp(\hookps/u)$ involves negative powers of $u$.  Therefore, by transporting the Hodge filtration on $\KdR{\X,\ps}$ through this isomorphism, we obtain a filtration on $\KdR{\X}$ that depends on $\ps$, which will be more convenient  for computations.
\begin{definition}
The \defn{Poisson--Hodge filtration} $\Fps\KdR{\X}$ is the decreasing filtration
\[
\KdR{\X} \supset \cdots \supset \Fps[p]\supset \Fps[p+1]\supset \cdots
\]
obtained by transporting the filtration $\F$ on the periodic cohomology of the mixed complex $\Mix{\X,\ps}$ via the isomorphism of \autoref{lem:exp-ps}.
\end{definition}

Using the Dolbeault resolution to define $\Mix{\X}$, the filtration $\Fps$ has the following explicit description.  First, if $\ps=0$, an element in $\KdR[n]{\X}$ lies in the subspace $\F[p] = \F[p]_{\ps =0}$ if and only if it can be represented by a closed differential form of mixed degree lying in the subspace
\[
\Dolb[n-2p]{\X} \oplus \Dolb[n-2p-2]{\X} \oplus  \Dolb[n-2p-4]{\X} \oplus \cdots \subset \Dolb[\bullet]{\X}((u))^n
\]
where we have identified the left hand side as a subspace of the degree-$n$ part of $\Dolb[\bullet]{\X}((u))$ via evaluation at $u=1$.  Consequently, under the identification $\KdR[n]{\X} \cong \bigoplus_{j\in\ZZ}\HdR[n+2j]{\X}$ we have
\[
\F[p]\KdR[n]{\X} \cong \bigoplus_{j\in\ZZ}\F[p+j]\HdR[n+2j]{\X}
\]
where $\F\HdR{\X}$ is the usual Hodge filtration given by $\F[p]\HdR[l]{\X} = \img \coH[l]{\forms[\ge p]{\X}}$. Put differently, the filtration on $\KdR{\X}$ corresponds, up to an overall shift, to the columns of the Hodge diamond as in \autoref{fig:surface-hodge-filtration}.

For general $\ps$, this filtration is modified: elements of $\Fps[p]$ are represented by closed differential forms that lie in the subspace
\[
e^{\hookps}\rbrac{\Dolb[n-2p]{\X} \oplus \Dolb[n-2p-2]{\X} \oplus \Dolb[n-2p-4]{\X} \oplus \cdots}  \subset \Dolb[\bullet]{\X}((u))^n,
\]
or equivalently by a sequence of forms $\omega_p,\omega_{p+1},\omega_{p+2},\ldots$  where $\omega_j \in \Dolb[n-2j]{\X}$ for all $j$, satisfying the sequence of equations
\begin{align*}
 (\delb + \delps)\omega_p &= 0 \\
\del \omega_{p} + (\delb + \delps)\omega_{p+1} &= 0 \\
\del \omega_{p+1} + (\delb + \delps)\omega_{p+2} &= 0 \\
&\ \ \vdots
\end{align*}
Nevertheless, the \emph{dimension} of  $\Fps[p]$ is independent of $\ps$ by \autoref{prop:degen} below.  In our simple examples from the previous subsection, it can be computed directly as follows.

\begin{figure}
\begin{subfigure}{0.45\textwidth}
\begin{center}
\begin{tikzcd}[row sep=5,column sep=5]
 &\F[n+1] \ar[r,hook]& \F[n] \ar[r,hook] & \F[n-1] \\
&& \coH[2]{\forms[2]{\X}} \\
&\coH[0]{\forms[2]{\X}} & \coH[1]{\forms[1]{\X}} & \coH[2]{\forms[0]{\X}} \\
&& \coH[0]{\forms[0]{\X}}
\end{tikzcd} 
\caption{$\KdR[2n]{\X}$}
\end{center}
\end{subfigure}
\begin{subfigure}{0.45\textwidth}
\begin{center}
\begin{tikzcd}[row sep=5,column sep=5]
 & \F[n] \ar[r,hook] & \F[n-1] \\
& \coH[1]{\forms[2]{\X}} & \coH[2]{\forms[1]{\X}} \\
& \coH[0]{\forms[1]{\X}} & \coH[1]{\forms[0]{\X}} \\ \ & \ 
\end{tikzcd}
\caption{$\KdR[2n-1]{\X}$}
\end{center}
\end{subfigure}
\caption{$\KdR{\X}$, shown here in the case $\dim \X = 2$, is the two-periodization of the de Rham cohomology.  Its Hodge filtration corresponds to the columns of the classical Hodge diamond, with the indexing shifted according to the degree.}\label{fig:surface-hodge-filtration}
\end{figure}

\begin{example}[\toricex]\label{ex:toric-hodge}
Let $\X$ be a toric log manifold as in \autoref{ex:toric-def}, compactifying the torus group $\Xo = (\Gm)^n$ with Lie algebra $\ft$.   Note that we have  a canonical isomorphism $\HdR{\Xo}\cong \wedge^\bullet \ft^\vee$; concretely the cohomology is freely generated by the elements $\dlog{x_i} \in \coH[0]{\forms[1]{\X}}$ where $x_1,\ldots,x_n : \Xo \to \Gm$ are toric coordinates (i.e.~a basis for the character lattice).  

To compute $\KdR{\X}$ and the Poisson Hodge filtration, we use the canonical quasi-isomorphism of mixed complexes $(\wedge^{-\bullet} \ft^\vee,0,0) \hookrightarrow \Mix{\X}$ from \autoref{ex:toric-mixed}.   It induces an isomorphism $\KdR{\X}\cong\wedge^{-\bullet}\ft^\vee((u))$ that commutes with the contraction of polyvectors into forms.  Hence if $\ps \in \wedge^2\ft \cong \coH[0]{\der[2]{\X}}$ is any Poisson structure on $\X$, we have an identification
\[
\Fps[p]\KdR{\X} \cong e^{\hookps/u}\rbrac{u^p \wedge^{2p-\bullet}\ft^\vee + u^{p+1} \wedge^{2p+2-\bullet}\ft^\vee + \cdots} \cong e^{\hookps}\rbrac{\wedge^{\ge 2p-\bullet}\ft^\vee},
\]
where the second isomorphism is evaluation at $u=1$.

In particular, if $\ps = \sum_{ij}\lambda_{ij}x_ix_j\cvf{x_i}\wedge\cvf{x_j}$, then
\[
\Fps[n]\KdR[-n]{\X} = \CC\cdot \sbrac{e^{\hookps/u} \dlog{x_1}\wedge\cdots \wedge \dlog{x_n}} \subset \KdR[-n]{\X}
\]
is a line, generated by the differential form
\begin{align*}
e^{\hookps/u} \dlog{x_1}&\wedge\cdots \wedge \dlog{x_n} =  \dlog{x_1}\wedge\cdots \wedge \dlog{x_n} \\
&\ \ \ + u^{-1} \sum_{i < j}(-1)^{(j-1)(i-1)}\lambda_{ij} \dlog{x_1}\wedge\cdots \wedge \widehat{\dlog{x_i}}\wedge\cdots \wedge \widehat{\dlog{x_j}} \wedge \cdots \wedge \dlog{x_n} \\
&\ \ \ + \cdots \,.
\end{align*}

Note that we may therefore recover $\ps$ from the position of the line $\Fps[n]\KdR[-n]{\X}$ relative to the decomposition $\KdR[-n]{\X}=\bigoplus_{j}\HdR[2j-n]{\X}$, namely the projection of $\Fps[n]\KdR[-n]{\X}$ to $\HdR[n]{\X}\oplus \HdR[n-2]{\X}$ is identified with the graph of the contraction operator $\hookps : \coH[0]{\forms[n]{\X}} \to \coH[0]{\forms[n-2]{\X}}$,  which determines $\ps$.
\end{example}

\begin{example}[\logsympex]\label{ex:log-symp-Hodge}
Let $\X$ be a log manifold of dimension $2n$ with a log symplectic form $\omega$ as in \autoref{ex:log-symp-def}. 
 Then the Hodge filtration has the following description, due to Kontsevich~\cite[Proposition 1.36.1]{Kontsevich2008a}.   Note that the map $\forms[-\bullet]{\X} \to \forms[\bullet]{\X}((u))$ defined by $\alpha \mapsto \frac{\alpha}{u^j}$ for $\alpha \in \forms[j]{\X}$ has degree zero.  Hence we may define a degree-zero $\CC((u))$-linear map
\[
L_\omega : \sMix[\bullet]{\X}((u)) \to \forms{\X}[2n]((u))
\]
by $L_\omega(\alpha) = u^{n}\frac{\omega}{u}\wedge \frac{\alpha}{u^j}$ for $\alpha \in \sMix[-j]{\X}=\forms[j]{\X}$.  Setting $u=1$, it is just multiplication by $\omega$, so it is analogous to the Lefschetz operator of a K\"ahler form. One checks as in \cite[Proposition 1.36.1]{Kontsevich2008a} that $
e^{u\hookps} * e^{-\hookps/u} = e^{-L_\omega}$, so that we have a commutative diagram\footnote{Note that Kontsevich uses the convention $\delta_\ps=[\iota_\ps, d]$. Since $\ast^2=1$ \cite[Lemma 2.1.2]{Brylinski1988}, Kontsevich's equation $e^{\omega\wedge -}=e^{\iota_\ps} \ast e^{-\iota_\sigma}$ gives this commuting diagram in our conventions.} of isomorphisms of sheaves of filtered dg $\CC((u))$-modules
\[
\begin{tikzcd}
	(\F\sMix[\bullet]{\X}((u)),\delps+u\dd)  \ar[r,"e^{\hookps/u}"] \ar[d,"*"] &
 (\Fps\sMix[\bullet]{\X}((u)),u\dd) \ar[d,"e^{-L_\omega}"] \\
 (\F\forms{\X}[2n]((u)),\dd+u\delps) \ar[r,"e^{u\hookps}"] & 
( \F\forms{\X}[2n]((u)),\dd)\,, 
\end{tikzcd}
\]
where in all corners but the top right, the filtration $F^\bullet$  is the Hodge filtration of the corresponding mixed complex.  Passing to hypercohomology and setting $u=1$, we obtain an isomorphism 
\[
\Fps[p]\KdR[j]{\X,\ps} \cong \bigoplus_{l \le n-p} e^{\omega}\HdR[j+2l]{\X} \subset \KdR[j]{\X}
\]
where we have identified $\omega$ with its image in $\HdR[2]{\X}$ under the canonical inclusion of the Hodge subspace $\F[2]\HdR[2]{\X}\cong \coH[0]{\forms[2]{\X}}$.  In particular, if $\X$ is connected, we have a line
\[
\Fps[n]\KdR[0]{\X} = \CC \cdot e^{\omega} \subset \KdR[0]{\X}
\]
from which we recover $\omega$ from the projection to $\HdR[0]{\X}\oplus \HdR[2]{\X}$ as the graph of the operator $\omega\wedge -$.
\end{example}

\subsection{The weight filtration}

Recall from Deligne~\cite{Deligne1971a} that the logarithmic forms $\forms{\X} = \logforms{\X}$ carry an increasing filtration
\[
\cdots \subset \W[j]\forms{\X}\subset \W[j+1]\forms{\X}\subset \cdots \subset \forms{\X}
\]
given by the order of poles, called the \defn{weight filtration}. Equivalently, we have $\W[j]\forms{\X} = \forms[j]{\X}\wedge\forms[\bullet-j]{\cX}$.  We denote by
\[
\W\sMix{\X,\ps} = \W\forms[-\bullet]{\X}
\]
the corresponding weight filtration on $\sMix{\X,\ps}$.   Recall that the weight filtration  is preserved by the de Rham differential $\dd$. It is also compatible with the Poisson bivector, in the following way.
 
\begin{lemma}\label{lem:bivector-preserves-weight}
The filtration $\W\forms{\X}$ is preserved by the contraction by any logarithmic polyvector field.  Thus $\W\sMix{\X,\ps}$ is a filtration by mixed subcomplexes.
\end{lemma}

\begin{proof}
Consider first the case of a vector field.  If $\xi \in \cT{\X} \subset \cT{\cX}$, and $j \ge 0$, we have
\[
\hook{\xi} \W[j]\forms[\bullet]{\X} = \hook{\xi}\rbrac{\forms[j]{\X} \wedge \forms[\bullet-j]{\cX}} \subset \forms[j-1]{\X}\wedge\forms[\bullet-j]{\cX} + \forms[j]{\X} \wedge \forms[\bullet-j-1]{\cX} = \W[j]\forms[\bullet-1]{\X}.
\]
Since the algebra $\der{\X}$ of polyvectors is generated over $\cO{\cX}$ by $\cT{\X}$, and the filtration $\W\forms{\X}$ is $\cO{\cX}$-linear, it follows that the weight filtration is preserved by contraction with an arbitrary logarithmic polyvector.  In particular, $\hookps$ preserves $\W\forms{\X}$, and hence so does the differential $\delps = [\dd,\hookps]$.
\end{proof}

This allows us to define the weight filtration on $\KdR{\X,\ps}$ as in Deligne's classical mixed Hodge theory~\cite{Deligne1971a}:
\begin{definition}
The \defn{weight filtration on $\KdR{\X,\ps}$} is defined by
\[
\W[j]\KdR[n]{\X,\ps} = \img\rbrac{ \coH[n]{\rsect{\W[j-n]\sMix{\X,\ps}((u))},\delps + u\dd} \to \KdR{\X,\ps}}\, .
\]
\end{definition}

In light of \autoref{lem:bivector-preserves-weight}, we have the following.
\begin{lemma}
The operator $e^{\hookps/u}$ on $\sMix{\X,\ps}((u))$ induces an isomorphism
\[
(\KdR{\X,\ps},\F,\W) \to (\KdR{\X},\Fps,\W)
\]
of bifiltered vector spaces.
\end{lemma}
Thus, unlike the Hodge filtration, the weight filtration is independent of the Poisson bivector.  Note that the indexing convention is such that under the canonical isomorphism $\KdR{\X}\cong\bigoplus_l \HdR[\bullet+2l]{\X}$, we have
\[
\W[j]\KdR{\X} = \bigoplus_{l} \W[j+l]\HdR[\bullet +2l]{\X}\,.
\]
where on the right hand side, $\W$ denotes Deligne's weight filtration in cohomology.

\begin{example}[\toricex]\label{ex:toric-weight}
For $\X$ a toric log manifold with Lie algebra $\ft$ and toric coordinates $x_1,\ldots,x_n$, we have $\HdR{\X} \cong \wedge^\bullet \ft^\vee \cong \CC[\dlog{x_1},\ldots,\dlog{x_n}]$, and the weight in cohomology is such that $\HdR[j]{\X}$ is concentrated in weight $2j$.  Hence, in this case, the weight filtration on $\KdR{\X}$ corresponds, up to a re-indexing, to the increasing filtration by the degree in cohomology.  For instance, we have
\[
\W[2j]\KdR[0]{\X}=\W[2j+1]\KdR[0]{\X} = \bigoplus_{l \le j} \HdR[2l]{\X}
\]
for all $j$.
\end{example}

As explained by Deligne \cite[\S 3.1.5]{Deligne1971a}, the associated graded of $\W \forms{\X}$ has the following description.  Let $\cX_j$ be the manifold parameterizing pairs of a point $p \in \cX$ and an ordered tuple $(Y_1,\ldots,Y_j)$ of distinct branches of $\bdX$ at $p$.   Then $\cX_j$ carries a free action of the symmetric group $S_j$ by reordering the $Y$s.  The map $\phi_j : \cX_j \to \cX$ that forgets the $Y$s is a closed immersion, whose image is the locus where $\bdX$ has multiplicity $j$ (the locus of $j$-fold intersection of distinct branches of $\bdX$).  Note that if $\ps$ is a Poisson structure on $\X$, then it is tangent to each such intersection, and thus lifts uniquely to an $S_j$-invariant Poisson structure $\ps_j$ on $\cX_j$.

Taking iterated residues along local branches of $\bdX$ gives a map
\begin{align}
\res : \W[j] \forms{\X} \to \phi_{j*}\forms[\bullet-j]{\cX_j} \label{eq:res-forms}
\end{align}
which induces an isomorphism of complexes
\[
\grW[j]\forms{\X} \cong \rbrac{\phi_{j*}\forms[\bullet]{\cX_j} \otimes \sgn_j}^{S_j}[-j]
\]
where $\sgn_j$ denotes the sign representation of $S_j$.  It is straightforward to check that the residue map \eqref{eq:res-forms} commutes with the natural action of the logarithmic polyvectors $\der{\X}$ on both sides by contraction.  Consequently, it gives an isomorphism
\[
\grW[j]\sMix{\X,\ps} \cong (\phi_{j*}\sMix{\cX_j,\ps_j}\otimes \sgn_j)^{S_j}[j]
\]
of sheaves of mixed complexes.

\subsection{Spectral sequences}
We may now associate to $\sMix{\X,\ps} = (\rforms{\X},\delps,\dd)$ three natural spectral sequences.  Firstly, we have the usual spectral sequence
\[
E_1^{-p,q} = \coH[q]{\forms[p]{\X}} \Rightarrow \KDol[q-p]{\X,\ps}
\]
for the hypercohomology of $(\sMix{\X,\ps},\delps)$. 
Meanwhile, the Hodge filtration induces the Hodge--de Rham spectral sequence 
\[
E_1 =  \KDol{\X,\ps}((u)) \Rightarrow \grF\KdR[\bullet]{\X,\ps},
\]
and the weight filtration induces the weight spectral sequence
\[
E_1^{-j,n+j} = \rbrac{\KdR[n]{\cX_j,\ps|_{\cX_j}} \otimes \sgn_j}^{S_j}\Rightarrow \grW\KdR{\X,\ps}\,.
\]

The following result, proven below, is fundamental:
\begin{proposition}\label{prop:degen}
For any  log Poisson manifold $(\X,\ps)$, the following statements hold:
\begin{itemize}
\item The hypercohomology and Hodge--de Rham spectral sequences both degenerate at $E_1$.
\item The weight spectral sequence degenerates at $E_2$.
\end{itemize}
\end{proposition}

In particular, for all $n$ and $j$, we have a canonical isomorphism
\[
\grF[j] \KdR[n]{\X,\ps} \cong \KDol[n-2j]{\X,\ps}\,,
\]
and by choosing a splitting of the hypercohomology filtration on $\KDol{\X,\ps}$, we obtain a \emph{non-canonical} isomorphism
\[
\KDol[n]{\X,\ps} \cong \bigoplus_{q-p=n} \coH[q]{\logforms[p]{\X}} = \KDol[n]{\X},
\]
further cementing the analogy between $\KDol{-}$ and Dolbeault cohomology.

\begin{proof}[Proof of \autoref{prop:degen}]
We will reduce the proof to the corresponding statement for ordinary logarithmic de Rham cohomology, proven by Deligne in \cite{Deligne1971a}.

For the first statement, note that since $\cX$ is proper and the sheaf $\forms{\X} = \logforms{\X}$ is coherent, the hypercohomology spectral sequence is finite-dimensional, which implies that $\KDol{\X,\ps}$ is finite-dimensional, and hence the Hodge--de Rham spectral sequence is finite-dimensional in each degree.  Considering the dimensions of the $E_1$ and $E_\infty$ pages, we obtain inequalities
\[
\sum_{q-p=j}\dim \coH[q]{\forms[p]{\X}} \ge \dim \KDol[j]{\X,\ps}
\]
and
\[
\sum_{j \in \ZZ} \dim \KDol[\ell+2j]{\X,\ps} \ge \dim \KdR[\ell]{\X}\,.
\]
The degeneration of both spectral sequences at $E_1$ is then equivalent to the statement that both inequalities are actually equalities.  This, in turn, is equivalent to the statement that
\[
\sum_{q-p=\ell \modtwo}\dim \coH[q]{\forms[p]{\X}} = \dim \KdR[\ell]{\X}\, ,
\] 
which is an immediate consequence of the $E_1$-degeneration of the classical Hodge--de Rham spectral sequence
\[
E_1^{p,q} = \coH[q]{\forms[p]{\X}} \Rightarrow \HdR[p+q]{\X}
\]
established in~\cite[Corollaire 3.2.13(ii)]{Deligne1971a}.

For the second statement, note that since conjugation by $\exp(\hookps/u)$ preserves the weight filtration on $\sMix{\X}((u))$, the problem reduces to the case $\ps = 0$.  In that case, the spectral sequence is simply the two-periodization of the usual weight spectral sequence for logarithmic de Rham cohomology, which degenerates at $E_2$ by \cite[Corollaire 3.2.13(iii)]{Deligne1971a}.
\end{proof}

\subsection{Flag varieties}\label{sec:flags}
It follows immediately from \autoref{prop:degen} that for a fixed log  manifold $\X$, the dimension of the Hodge filtration $\Fps$ is independent of $\ps$:
\begin{corollary}\label{cor:hodge-nums-const}
If $(\X,\ps)$ is a log Poisson manifold, then
\[
\dim \Fps[p]\KdR[n]{\X} = \dim \F[p]\KdR[n]{\X}
\]for all $p,n \in \ZZ$.
\end{corollary}

As a consequence, we have a well-defined map from the space of Poisson structures on $\X$ to a suitable flag manifold, which we describe as follows.  Let $b_{n}^p := \dim \F[p]\KdR[n]{\X}$ and let $b_n = (b_n^p)_{p\in \ZZ}$. Note that by the Bott periodicity isomorphism~\autoref{eq:bott}, we have
\[
b_{p}^{n} = b_{p+1}^{n+2}
\]
for all $p,n$.  Let $\Flag[n]{\X}$ be the variety parameterizing flags of shape $b_n$ in the vector space $\KdR[n]{\X}$.  Then for any Poisson structure $\ps$, the filtration $\Fps\KdR[n]{\X}$ gives a point in $\Flag[n]{\X}$, and hence we may define a map
\[
\mapdef{\phi_n}{\bA^1}{\Flag[n]{\X}}
{\hbar}{\F_{\hbar\ps}}
\]
where $\bA^1=\CC$ is the complex  line.  

Note that the vector space $\KdR[n]{\X} \cong \PcoHdR[n]{\X}$ carries a canonical action of the multiplicative group $\Gm$, defined by the property that the subspace $\HdR[n+2j]{\X}$ has weight $-j$, i.e.~$\hbar \in \Gm$ acts by 
\begin{align*}
\hbar \adm \alpha := \hbar^{-j} \alpha \qquad \textrm{for }\alpha \in \HdR[n+2j]{\X} \subset \KdR[n]{\X}\,.  
\end{align*}
We then have an induced action of $\Gm$ on the flag manifold $\Flag[n]{\X}$, which extends to an action of the multiplicative monoid $\bA^1$ by the valuative criterion for properness.  The map $\phi_n$ then has the following crucial property:

\begin{lemma}\label{lem:equivariant}
The map $\phi_n$ is $\Gm$-equivariant: we have
\begin{align*}
\F[p]_{\hbar \ps}\KdR[n]{\X} = \hbar \adm \Fps[p]\KdR[n]{\X} 
\end{align*}
for every $\hbar \in \bA^1$ and every Poisson structure $\ps$ on $\X$.
\end{lemma}

\begin{proof}
We first lift the action of $\Gm$ on $\KdR{\X}$ to an action at cochain level.  Using the Dolbeault resolution, we have
\[
\Mix{\X}((u)) = \Dolb[\bullet]{\X}((u)) = \bigoplus_{j,p,q \in \ZZ} u^j \Dolb[p,q]{\X}\,,
\]
where $u^j\Dolb[p,q]{\X}$ has degree $q-p+2j$.   Consider the action $\diamond$ of $\Gm$ on the vector space $\Dolb[\bullet]{\X}((u))$ for which $u^j\Dolb[p,q]{\X}$ has weight $j-p$.  Then the operators $\delb$ and $u\del$ have weight zero, so that the differential $\delb+u\del$ is preserved.  Note that cocycles in $u^j\Dolb[p,q]{\X}$ project to the summand $\HdR[p+q]{\X} \subset \K[q-p+2j]{\X}$, on which $\Gm$ acts with weight equal to $\tfrac{1}{2}((q-p+2j)-(p+q))= j-p$.  Hence the action $\diamond$ on $\Dolb[\bullet]{\X}((u))$ induces the action $\diamond$ on $\KdR{\X}$, as desired. Furthermore, since $u$ has weight one, the action preserves the filtration $\F$ by the powers of $u$, i.e.~$\hbar \diamond \F = \F$ for all $\hbar \in \Gm$. 

Now observe that the operator $\hookps$ sends $u^j\Dolb[p,q]{\X}$ to $u^j\Dolb[p-2,q]{\X}$, and therefore has weight two, so that $\hookps/u$ has weight one, i.e.~we have the commutation relation
\[
(\hbar\adm-) \circ \frac{\hookps}{u}  = \frac{\hook{\hbar \ps}}{u}\circ (\hbar \adm -)
\]
as operators on $\Dolb[\bullet]{\X}((u))$.  Therefore
\[
\hbar\diamond  \Fps = \hbar \diamond (e^{\hookps/u}\F) = e^{\hbar\hookps/u}\hbar\diamond \F = e^{\hbar\hookps/u}\F = \F_{\hbar\ps}\,,
\]
as desired.
\end{proof}

\subsection{The Hodge decomposition and Bott periodicity}
To set notations, recall that a \defn{pure $\RR$-Hodge structure of weight $n$} is a $\CC$-vector space $\V_\dR$ equipped with a real structure (i.e.~an antilinear involution $v \mapsto \overline{v}$), and a decreasing filtration $\F = \F\V_\dR$ by $\CC$-linear subspaces, satisfying the \defn{opposedness axiom}, which states that the subspaces
\[
\V^{p,q} := \F[p]\V_\dR \cap \overline{\F[q]{\V_\dR}}
\]
for $n=p+q$ give a decomposition $\V_\dR = \bigoplus_{p+q=n}\V^{p,q}$.  A \defn{mixed $\RR$-Hodge structure} is the data of a complex vector space $\V_\dR$, equipped with a real structure, a decreasing filtration $\F\V_\dR$ and an increasing filtration $\W=\W\V_\dR$ such that
\begin{itemize}
\item $\W$ is preserved by the real structure, i.e.~$\W = \overline{\W}$;
\item $\F$ induces a pure Hodge structure of weight $j$ on $\grW[j]\V_\dR$ for all $j \in \ZZ$.
\end{itemize}
If $\V$ is a mixed $\RR$-Hodge structure, its \defn{Tate twist} $\V(j)$ is the vector space $\V_\dR$ equipped with the real structure $v \mapsto (-1)^j\overline{v}$, and the filtrations defined by
\[
\F[k]\V(j)_\dR:=\F[k+j]\V_\dR  \qquad \W[k]\V(j)_\dR:= \W[k+2j]\V_\dR
\]
for all $j,k \in \ZZ$.

If $\X$ is a  log manifold, then the de Rham cohomology $\HdR{\X}$  carries a mixed Hodge structure by a theorem of Deligne~\cite[Theorem 3.2.5]{Deligne1971a}; the weight filtration is induced by $\W\forms{\X}$, the Hodge filtration is induced by the truncations $\F[p]\forms{\X} = \forms[\ge p]{\X}$, and the real structure is given by complex conjugation of differential forms.  We denote this mixed Hodge structure by  $\coH{\X;\RR}$.  One readily checks that the filtrations $\W\KdR{\X}$ and $\F\KdR{\X}$ described above correspond precisely to sums of Tate twists of $\HdR{\X}$ under the isomorphism $\KdR{\X}\cong \bigoplus_{j\in\ZZ}\HdR[\bullet+2j]{\X}$ of \eqref{eq:K-vs-H}: more precisely, there is a unique real structure on $\KdR{\X}$, making it into a mixed $\RR$-Hodge structure $\KR{\X}$ such that
\[
\KR{\X} \cong \bigoplus_{j \in \ZZ} \coH[\bullet+2j]{\X;\RR}(j).
\]
As we will recall below, it can be naturally refined to an integral lattice given by topological K-theory, which will match well with the Tate twists. 

We incorporate Poisson structures by deforming $\F$:
\begin{definition}
If $\ps$ is a Poisson structure on $\X$, we denote by $\KR{\X,\ps}$ the data of the bifiltered vector space $(\KdR{\X},\Fps,\W)$ together with the canonical real structure on $\KdR{\X}$ as above.
\end{definition}

 The main theorem of this section is the following.
\begin{theorem}\label{thm:R-MHS}
	Let $(\X,\ps)$ be a log Poisson manifold. Then $\KR{\X,\ps}$ is a mixed $\RR$-Hodge structure, with the following properties:
\begin{enumerate}
\item\label{it:functor}\defn{Functoriality:} If $\phi : (\X,\ps) \to (\Y,\eta)$ is any morphism of log Poisson manifolds, then the pullback in de Rham cohomology defines a morphism $\phi^* : \KR{\Y,\eta}\to\KR{\X,\ps}$ of mixed $\RR$-Hodge structures.
\item \label{it:weak-invar} \defn{Invariance under weak equivalence:} If $\phi : (\X,\ps) \to (\Y,\eta)$ is a weak equivalence of log Poisson manifolds, then $\phi^*: \KR{\Y,\eta} \to \KR{\X,\ps}$ is an isomorphism.
\item\label{it:Bott} \defn{Bott periodicity:} For every $j \in \ZZ$, multiplication by $u^j$ gives a natural isomorphism
\[
\KR{\X,\ps} \cong \KR[\bullet+2j]{\X,\ps}(j).
\]
\item\label{it:equivariance}\defn{Equivariance:} For every $\hbar \in \RR^\times$, the action $\hbar \diamond (-)$ on $\KdR{\X}$, gives an isomorphism
\[
\KR{\X,\ps} \cong \KR{\X,\hbar \ps}.
\]
\item\label{it:wt-bound} \defn{Weight bounds:} If $\X$ has no boundary strata of codimension greater than $k$, i.e.~$\cX_{k+1} = \varnothing$, then the weights of $\KR[n]{\X,\ps}$  are concentrated in the interval $[n,n+k]$, that is, we have
\[
\grW[j]\KR[n]{\X,\ps} = 0 \qquad \textrm{ for }j \notin [n,n+k].
\]
\item\label{it:lowest-wt} \defn{Description of the lowest weight:} The lowest weight-graded piece is given by
\[
\grW[n]\KR[n]{\X,\ps} = \img\rbrac{\KR[n]{\cX,\ps} \to \KR[n]{\X,\ps}}\,.
\]
\item\label{it:pure} \defn{Compactness implies purity:} If $\X$ is compact K\"ahler, i.e.~the boundary $\bdX$ is empty, then $\KR[n]{\X,\ps}$ is a pure Hodge structure of weight $n$.
\end{enumerate}
\end{theorem}

\begin{proof}
 Statement (\ref{it:Bott}) follows immediately from the definition of the filtrations and the real structure. For statement (\ref{it:equivariance}), the compatibility of the action with the Hodge filtrations is the content of \autoref{lem:equivariant}, so it suffices to check compatibility with the real structure.  But if $\hbar \in \CCx$, the operator $\hbar \diamond (-)$ acts on each summand $\HdR[n+2j]{\X}$ by scalar multiplication by a power of $\hbar$.  Hence if $\hbar \in \RR^\times$, it acts by a real scalar, and therefore preserves any real structure on $\HdR[n+2j]{\X}$, as desired.

 Statement (\ref{it:functor}) is immediate from the functoriality of the mixed complex  and of the real structure in cohomology. 
 Assuming that $\KR{\X,\ps}$ is a mixed $\RR$-Hodge structure,  statement (\ref{it:weak-invar}) follows: restriction of de Rham cohomology classes to the interior gives an isomorphism $\KdR{\X} \cong \bigoplus_{j}\HdR[\bullet+2j]{\Xo}$, and if $\phi$ is a weak equivalence then $\phi^\circ : \Xo \to \Yo$ is an isomorphism, so that the morphism $\phi^*$ of mixed $\RR$-Hodge structures is an isomorphism at the level of vector spaces, hence an isomorphism.

It thus remains to prove that $\KR{\X,\ps}$ is a mixed $\RR$-Hodge structure and establish the statements about the weights.   To do so, we will first treat the case in which $\X = \cX$ is a compact K\"ahler manifold, i.e.~we prove statement (\ref{it:pure}).  For this, we need to show that $\KdR[n]{\X}\cong\Fps[p]\oplus \bFps[q]$ whenever $n = p+q+1$.  But complex conjugation is continuous, so this is an open condition on flags in $\KdR{\X}$, and it holds for $\ps=0$ by classical Hodge theory.  Hence it holds also for $\ps$ in an open ball $\U$ around the origin in the finite-dimensional vector space $\coH[0]{\der[2]{\X}}$.  But if  $\ps$ is \emph{any} Poisson structure on $\X$, there exists a constant $\hbar \in \RR^\times$ such that $\hbar \ps \in \U$, so the result follows from statement \eqref{it:equivariance}, which we have already proven.

Now let $\X$ be arbitrary, and consider the weight spectral sequence for $\KdR{\X,\ps}$; its $E_1$ page is
\[
E_1^{-j,n+j}=(\KdR[n]{\cX_j,\ps_j} \otimes \sgn_j)^{\symgp{j}} \Rightarrow \KdR{\X,\ps}.
\]
Since this spectral sequence is the two-periodization of the weight spectral sequence for cohomology, its differential preserves the real structure.  By \autoref{prop:degen}, the spectral sequence degenerates at $E_2$, so it follows from \cite[Corollaire 1.3.17]{Deligne1971a} (see also \cite[Lemma 3.11 and Theorem 3.12]{Peters2008}) that the filtration on $\grW\KdR{\X,\ps}$ induced by $\F$ coincides with the filtration on $E_2$ induced by the filtration on $E_1$, which is exactly the direct sum of the Hodge  filtrations  on the direct summands $(\KdR[n]{\cX_j,\ps_j}\otimes\sgn_j)^{\symgp{j}}$.  As proven above, under the isomorphism $e^{\hook{{\ps_j}}/u}$ with $\KdR[n]{\cX_j}$, the latter defines a pure Hodge structure of weight $n$.  Hence $\grW[j] \KdR{\X,\ps} \cong E_2$ is the cohomology of a complex of pure Hodge structures of weight $j$ and is thus pure of weight $j$, so that $\KR{\X,\ps}$ is a mixed $\RR$-Hodge structure.  Statements (\ref{it:wt-bound}) and (\ref{it:lowest-wt}) now follow immediately as in classical mixed Hodge theory \cite[\S3.2]{Deligne1971a} by considering the weights and degrees on the $E_1$ page.
\end{proof}

\subsection{Compact support}

Let $(\X,\ps)$ be a log Poisson manifold.  We now explain how to introduce a version of de Rham theory with compact support, which is Poincar\'e dual to the Hodge structure defined above.

Consider the sheaf
\[
\forms{\X}(-\bdX) = \logforms{\X}(-\bdX) \subset \forms{\cX}
\]
of forms that can be written locally as $f \omega$ where $f$ is a local defining equation for $\bdX$ and $\omega$ is a logarithmic form.  One readily checks that it is exactly the sheaf of holomorphic forms on $\cX$ whose pullback to $\bdX$ is identically zero. Hence its hypercohomology computes the relative de Rham cohomology $\HdR{\cX,\bdX}$, or equivalently, the compactly supported cohomology of the interior $\HdRc{\Xo}$; see, e.g.~\cite[\S3.3]{Brown2021}.  Hence we denote it by $\HdRc{\X}$ and call it the \defn{compactly supported de Rham cohomology of $\X$}.  The vector space $\HdRc{\X}$ carries a mixed Hodge structure that is Poincar\'e dual to $\coH{\X;\RR}$: in brief, the Hodge filtration is induced by the truncations $\F[p] \forms{\X}(-\bdX) = \forms[\ge p]{\X}(-\bdX)$, and the weight spectral sequence has $E_1$ page given by the alternating sum of the pullbacks $\HdR{\cX_j}\to\HdR{\cX_{j+1}}$. 

Consider the subsheaf
\[
\sMixc{\X,\ps} := \forms[-\bullet]{\X}(-\bdX) \subset \sMix{\X,\ps}\,.
\] It is a $\der{\X}$-submodule by $\cO{\cX}$-linearity, and is readily seen to be preserved by the de Rham differential, so that it defines a sheaf of mixed subcomplexes.  Passing to derived global sections we obtain a morphism of mixed complexes
\[
\Mixc{\X,\ps} \to \Mix{\X,\ps}
\]
and we make the following definition.
\begin{definition}
The \defn{compactly supported periodic cyclic Poisson homology} $\KdRc{\X,\ps}$ is the periodic cyclic cohomology of $\Mixc{\X,\ps}$.
\end{definition}
Applying \autoref{lem:exp-ps}, we deduce that the operator $e^{\hookps/u}$ identifies the periodic hypercohomology of $\sMixc{\X,\ps}$ with $\KdRc{\X}$ exactly as in the non-compactly supported case, so that we obtain a canonical mixed $\RR$-Hodge structure 
\[
\KRc{\X,\ps} = (\KdRc{\X},\Fps,\W,\overline{(-)})\,.
\]

\subsection{Poincar\'e duality}\label{sec:deRham-dual}

The natural Poincar\'e duality pairing between cohomology with and without compact supports is expressed by integration of compactly supported forms.  However, note that in our model for $\KdRc{\X}$ we work with logarithmic forms whose support is not compact in general; rather they are required to be smooth on $\cX$ and vanish when pulled back to the boundary. As observed in \cite[Remark 3.3]{Brown2021}, this is enough to deduce that Poincar\'e duality is still expressed by the obvious integral.

More precisely, define the \defn{Mukai involution} by
\begin{align*}
\omega \mapsto \omega^\vee := (-1)\diamond \omega \label{eq:mukai}\,,
\end{align*}
i.e.~$(-)^\vee$ acts with eigenvalue $(-1)^j$ on the summand $\HdR[n+2j]{\X}\subset \KdR[n]{\X}$.  As in the proof of \autoref{lem:equivariant}, this lifts to the Dolbeault resolution $\Dolb[\bullet]{\X}((u))$, so we may define a $\CC((u))$-bilinear map
\[
\pairnaive{-,-}_\dR : \Dolb[\bullet]{\X}((u)) \otimes_{\CC((u))} \Dolbc{\X}((u)) \to \CC((u)) \qquad 
\nu \otimes \omega \mapsto \int_{\cX} \nu^\vee \wedge\omega\,.
\]
It is well defined because the zeros of $\omega$ cancel the poles of $\nu^\vee$ so that the product $\nu^\vee \wedge \omega$ is a smooth form on the compact manifold $\cX$, and the integral converges. Since $\cX$ is even-dimensional as a real manifold, forms of opposite parity are orthogonal, and the pairing defines a morphism of complexes by Stokes' theorem.

The compatibility with the Poisson Hodge filtrations is expressed by the following result, which in the case $\X=\cX$ is  a special case of Poincar\'e duality for compact generalized K\"ahler manifolds~\cite{Cavalcanti2006,Gualtieri2004}.  Indeed, the argument for the latter reduces to linear algebra pointwise in each tangent space, and hence it works equally well when $\cX \neq \X$.

\begin{proposition}
The pairing has the following property:
\[
\pairnaive{\Fps[p],\Fps[q]}_\dR \subset \F[p+q]\CC((u)) = u^{p+q}\CC[[u]]\,.
\]
\end{proposition}

\begin{proof}
	Suppose that $\nu = e^{\iota_\ps/u}\nu_0$ and $\omega = e^{\iota_\ps/u}\omega_0$, where $\nu_0 \in \F[p]\Dolb[\bullet]{\X}((u))$ and $\omega_0 \in \F[q]\Dolb[\bullet]{\X}((u))$ lie in the corresponding components of the classical Hodge filtration.  One can verify that the top-degree components of $\nu^\vee \wedge \omega$ and $\nu_0^\vee \wedge\omega_0$ are equal; for instance, this is a consequence of the orthogonality of the Clifford action on spinors as in \cite[\S 2]{Cavalcanti2006} and \cite[\S 1.1]{Gualtieri2011}.  Hence we have $\int_{\cX}\nu^\vee \wedge\omega = \int_{\cX}\nu_0^\vee \wedge\omega_0$, so that the result reduces to the classical case $\ps=0$, where it holds by orthogonality of the Dolbeault decomposition on forms.
\end{proof}

In this way we obtain Poincar\'e duality for our Hodge structures:
\begin{corollary}\label{cor:Poincare-duality}
The pairing $\pairnaive{-,-}_\dR$ on forms induces a nondegenerate pairing 
\[
\pairnaive{-,-}_{\dR} : \KR{\X,\ps} \otimes \KRc[-\bullet]{\X,\ps} \to \RR
\]
of mixed $\RR$-Hodge structures.
\end{corollary}

\begin{remark}
    Note that by \autoref{thm:R-MHS} part (\ref{it:equivariance}), the Mukai involution gives an isomorphism $\K{\X,\ps}\cong\K{\X,-\ps}$.  Hence the standard integration pairing $\nu \otimes \omega \mapsto \int_{\cX}\nu\wedge\omega$ gives a duality between $\KdR{\X,\ps}$ and $\KdRc{\X,-\ps}$.
\end{remark}

\subsection{Gysin maps}
As an application of Poincar\'e duality, we show that the ``wrong way'' maps induced by proper morphisms are compatible with mixed $\RR$-Hodge structures.  For this, we adapt the notion of proper maps to the logarithmic setting, as follows.

\begin{definition}
A map $f : \X \to \Y$ of log manifolds is \defn{proper} if the induced map of interiors $f^\circ : \Xo \to \Yo$ is a proper map of topological spaces (with the classical analytic topology). 
\end{definition}
This has several equivalent reformulations.  For this, we recall that a sequence of points $p_1,p_2,\ldots$ in a topological space is said to ``escape to infinity''  if for every compact subset $S$, we have $p_n \notin S$ for $n \gg 0$.  Note that if $\X$ is a log manifold, then a sequence in $\Xo$ escapes to infinity in this sense if and only if all of its accumulation points in $\cX$ lie in $\bdX$.   

\begin{lemma}
For a morphism $f: \X\to \Y$ of log manifolds, the following conditions are equivalent:
\begin{enumerate}
\item $f$ is proper
\item $\bdX \subset f^{-1}(\bdY)$
\item $f^{-1}(\Yo) \subset \Xo$
\item  For every sequence $x_1,x_2,\ldots \in \Xo$ that escapes to infinity, the image sequence $f(x_1),f(x_2), \ldots \in \Yo$ also escapes to infinity.
\end{enumerate}
\end{lemma}

\begin{proof}
The equivalence of (2) and (3) is immediate since $\Xo = \cX \setminus \bdX$ and $\Yo = \cY\setminus \bdY$. 
 The equivalence of (1) and (4) is standard point-set topology (using that $\Xo$ and $\Yo$ are metrizable).  It therefore suffices to check that (4) implies (2), and (3) implies (4).

To see that (4) implies (2), suppose that $p \in \bdX$ and (4) holds. We claim that $p\in f^{-1}(\bdY)$.  Indeed, choose a sequence $p_1,p_2,\ldots \in \Xo$ converging to $p$, hence escaping to infinity in $\Xo$.  By properness, $f(p_1),f(p_2),\ldots \in \Yo$  escapes to infinity and hence $f(p) = \lim_n f(p_n)$ lies in $\bdY$, so that $p \in f^{-1}(\bdY)$.   

Finally, to see that (3) implies (4), suppose $f^{-1}(\Yo)\subset \Xo$ and let $p_1,p_2,\ldots \in \Xo$ be a sequence escaping to infinity.  We claim that $f(p_1),f(p_2),\ldots \in \Yo$ escapes to infinity.  Indeed, if not, then there is an accumulation point $q \in \Yo$. Without loss of generality, we have $\lim_n f(p_n) = q$.  Since $\cX$ is compact, $p_1,p_2,\ldots$ has an accumulation point $\tilde q \in \cX$, and by continuity $f(\tilde q) = q \in \Yo$.  Therefore $\tilde q \in f^{-1}(\Yo)\subset \Xo$, which is a contradiction since $p_1,p_2,\ldots,$ escapes to infinity in $\Xo$.
\end{proof}

From part (2) of the lemma, and the definition of $\sMixc{\X,\ps}\subset\sMix{\X,\ps}$, we immediately deduce the following.
\begin{corollary}
If $f : (\X,\ps) \to (\Y,\eta)$ is a proper morphism of log manifolds, then the pullback on differential forms defines a morphism of mixed complexes $f^{-1}\sMixc{\Y,\eta} \to\sMixc{\X,\ps}$.
\end{corollary}

\begin{corollary}\label{cor:deRham-Gysin}
If $f : (\X,\ps) \to (\Y,\eta)$ is a proper morphism of  log Poisson manifolds, then the pullback on compactly supported cohomology and its Poincar\'e dual define morphisms
\[
f^!_\dR : \KRc{\Y,\eta} \to \KRc{\X,\ps} \qquad f^\dR_! : \KR{\X,\ps} \to \KR{\Y,\eta}
\]
of mixed $\RR$-Hodge structures.
\end{corollary}

\begin{proof}
It suffices to observe that the pullback in compactly supported cohomology is induced by the pullback on the mixed complexes $\sMixc{\Y,\eta}$.  To see this, note that the pullback on compactly supported cohomology is induced by the pullback on compactly supported $C^\infty$ forms.  After periodization, the inclusion of the latter into the Dolbeault resolution of  $\sMixc{\X,\ps}((u))$ is a quasi-isomorphism, thanks to \cite[Proposition 3.10]{Brown2021}.  The result follows since this inclusion commutes with pullbacks along proper maps.
\end{proof}

\subsection{Characteristic classes of vector bundles}

Let $\E \to \X$ be a log vector bundle.  By a \defn{characteristic class of $\E$} we mean a class in $\KR[0]{\X}$ defined by a polynomial in the Chern classes of the vector bundle $\Eo \to \Xo$ via the canonical isomorphism $\KR[0]{\X}\cong \bigoplus_{j\in\ZZ} \coH[2j]{\Xo;\RR(j)}$.  

\begin{theorem}\label{thm:chern-classes}
Let $\E \to \X$ be a coabelian log Poisson vector bundle, and let $\gamma \in \KR[0]{\X}$ be any characteristic class of $\E$.  Then the multiplication operator $\gamma \cup - : \KR{\X,\ps} \to \KR{\X,\ps}$ is a morphism of mixed $\RR$-Hodge structures and similarly for $\KRc{\X,\ps}$.
\end{theorem}

\begin{proof}
Since the weight filtration is independent of $\ps$, to see that it is preserved, it suffices to treat the case $\ps=0$.  In this case it is standard: since the vector bundle $\Eo \to \Xo$ extends to a vector bundle $\E' \to \cX$, its characteristic classes lie in the image of the map $\KdR{\cX} \to \KdR{\X}$, which is $\W[0]\KdR{\X}$.  Since the weight filtration is compatible with multiplication, it follows that $\gamma \cup -$ preserves $\W\KdR{\X}$.

For the Hodge filtration, we reduce the problem to the case in which $\E$ has rank one using the projective bundle $p : \PP(\E \oplus \cO{\X}) \to \X$, as follows.  Recall that every characteristic class of $\E$ can be expressed as a polynomial in the Segre classes $s_k(\E)$ for $k \ge 1$, which act on cohomology by the formula
\[
s_k(\E)\cup \omega = p_!( H^{k+r-1} p^*\omega )\,,
\]
where $p_!$ and $p^*$ are the  Gysin pushforward and ordinary pullback, respectively, $H \in \coH[2]{\PP(\E\oplus \cO{\X})}$ is the first Chern class of the fibrewise hyperplane bundle, and $r$ is the rank of $\E$.  Since $p_!$ and $p^*$ preserve Hodge structures, it suffices to prove that multiplication by $H$ does, which gives the desired reduction to rank one.

So suppose that $\E$ is a logarithmic Poisson line bundle and let $\sL$ be the associated invertible sheaf as discussed in \autoref{ex:log-line-bundle}, we have a flat Poisson connection $\nabla : \sL \to \cT{\X}\otimes \sL$ taking values in the logarithmic tangent sheaf.

We must show that multiplication by $c_1(\sL)$ preserves $\Fps$, or equivalently, that the conjugation of this operation by $e^{\hookps/u}$ preserves the $u$-adic filtration on the periodic complex $(\Dolb[\bullet]{\X}((u)),\delb + \delps + u\del)$ up to homotopy.   For this, choose a $C^\infty$ unitary connection $\nabla_0$ on $\sL$ compatible with the holomorphic structure. Let $\omega = \nabla_0^2$ be the curvature; thus $\omega$ is a closed $(1,1)$-form representing $c_1(\sL)$.  Then the operator $-\pss \nabla_0 : \Dolb[0,0]{\X;\sL} \to \Dolb[0,0]{\cT{\X}\otimes\sL}$ on $C^\infty$  differs from $\nabla$ by a $C^\infty$ section $Z$ of $\cT{\X}$.  Considering the components of the curvatures $(\delb + \nabla)^2$ and $(\delb -\pss \nabla_0)^2$ in $\Dolb[0,1]{\cX;\cT{\X}}$, we find
\[
(-\pss \otimes 1) \omega = 
(-\pss \nabla_0)^2 = (\delb+\nabla+Z)^2 = \delb Z + [\ps, Z]\, .
\]
Now consider $\omega$ as an operator on forms by the wedge product.  Since $\omega$ has type $(1,1)$ and $\ps$ has type $(2,0)$, the Leibniz rule for the contraction implies that the commutator of $\omega\wedge -$ and $\hook{\ps}$ is determined by the contraction of the $(1,0)$-part of $\omega$ into $\ps$; namely, we have
\[
[\omega\wedge - , \hookps -] = \hook{(-\pss \otimes 1) \omega } = \hook{(\delb + [\ps,-])Z} = [\delb+\delps,\hook{Z}]\,.
\]
We therefore have 
\[
[\hookps,[\hookps,\omega\wedge -]] =  [\hookps,[\delb+\delps,\hook{Z}]] = [[\hookps,\delb+\delps],\hook{Z}] + [\delb+\delps,[\hookps,\hook{Z}]] = 0\,,
\]
where we have used the Jacobi identity for the commutator, and the identities $[\hookps,\hook{Z}]=[\hookps,\delb]=[\hookps,\delps]=0$. 
Consequently,
\begin{align*}
e^{-\hookps/u}(\omega\wedge - ) e^{\hookps/u} &= (\omega \wedge -) - u^{-1}[\hookps,\omega\wedge -] + \tfrac{u^{-2}}{2} [\hookps,[\hookps,\omega \wedge - ]] + \cdots \\
&= (\omega \wedge -) - u^{-1}[\delb+\delps,\hook{Z}] + 0 + \cdots \\
&= u^{-1}[\delb+\delps+u\del,-\hook{Z}] \mod \End{\Dolb[\bullet]{\X}}{[[u]]} \\
&\sim 0 \mod \End{\Dolb[\bullet]{\X}}{[[u]]}\,,
\end{align*}
where $\sim$ denotes cochain homotopy in $\End{\Dolb[\bullet]{\X}}((u))$.  Therefore multiplication by $\omega$ preserves the $u$-adic filtration up to homotopy, and moreover the homotopy preserves the Dolbeault resolution of the subcomplex $\sMixc{\X,\ps}\subset\sMix{\X,\ps}$.  Hence the result follows by taking cohomology.
\end{proof}

In particular, since the canonical bundle is always a coabelian log Poisson line bundle (\autoref{ex:log-line-bundle}), we have the following.
\begin{corollary}\label{cor:canonical-bundle-R-hodge}
For any log Poisson manifold $(\X,\ps)$ and any $t \in \RR$, multiplication by the element $e^{t c_1(\X)} \in \KdR[0]{\X}$ is an automorphism of the mixed $\RR$-Hodge structures $\KR{\X,\ps}$ and $\KRc{\X,\ps}$.
\end{corollary}

\section{Integral mixed Hodge structures}
\label{sec:lattices}

In this section, we promote the $\RR$-Hodge structure  $\KR{\X,\ps}$ to a full mixed Hodge  structure by incorporating a suitable integral lattice.  We then study its functoriality and explain how various standard tools for calculation of mixed Hodge structures in classical algebraic geometry carry over to our setting.

\subsection{Lattices}
\label{sec:Z-MHS} We recall the following notion, to set our notations.

\begin{definition}
A \defn{lattice} in a mixed $\RR$-Hodge structure $(\V_\dR,\F,\W,\overline{(-)})$ is a pair $(\V_\Bet,c)$ consisting of a finitely generated abelian group $\V_\Bet$ and a $\ZZ$-linear map $c : \V_\Bet \to \V_\dR$  called the \defn{comparison map}, with the following properties:
\begin{itemize}
	\item The induced map $c_\CC : \V_\Bet \otimes_\ZZ \CC \to \V_\dR$  of $\CC$-vector spaces is an isomorphism.
\item The isomorphism $c_\CC$ identifies the real structure of $\V_\dR$ with the real structure on $\V_\Bet \otimes_\ZZ \CC$ defined by $\ZZ$-linear extension of the complex conjugation on $\CC$.
\item The weight filtration on $\V_\dR$ lifts to $\V_\Bet \otimes \QQ$, i.e.~the $\QQ$-linear subspaces
\[
\W[j](\V_\Bet \otimes_{\ZZ} \QQ) := c^{-1}_\QQ(\W[j]\V)
\]
are such that $c$ induces a $\CC$-linear isomorphism $\W[j](\V_\Bet \otimes_{\ZZ} \QQ) \otimes_\QQ \CC \to \W[j]\V$.
\end{itemize}
A \defn{mixed Hodge structure} is a tuple $\V=(\V_\dR,\V_\Bet,\F,\W,c)$ consisting of 
a mixed $\RR$-Hodge structure equipped with a lattice.
\end{definition}
We will sometimes refer to a mixed Hodge structure $\V$ as an \emph{integral} mixed Hodge structure to distinguish it from the underlying mixed $\RR$-Hodge structure, which we will denote by $\V_\RR$.

With the obvious notion of tensor product and duals, integral mixed Hodge structures form a rigid abelian tensor category.  The \defn{$j$th Tate structure $\ZZ(j)$} is the pure mixed Hodge structure of weight $-2j$ induced by the inclusion $(\tipi)^j\ZZ \hookrightarrow \CC$ and the \defn{$j$th Tate twist} of a mixed Hodge structure $\V$ is $\V\otimes \ZZ(j)$.

Note that the Hodge filtration does not enter the axioms of a lattice. This is convenient for us: in our construction of the mixed $\RR$-Hodge structure $\KR{\X,\ps}$, the bivector $\ps$ only enters through the Hodge filtration, so the problem of constructing a lattice reduces to the case $\ps=0$.  In fact, this allows us to construct multiple natural lattices which become isomorphic over $\QQ$, as follows.

\subsubsection{The de Rham lattice} Note that since we have the isomorphism
\[
\KR[n]{\X} \cong   \bigoplus_{j \in \ZZ} \coH[n+2j]{\Xo;\RR}(j) 
\]
as mixed $\RR$-Hodge structures, there is an obvious integral lattice, induced by the singular cohomology of the interior $\Xo$, with suitable Tate twists inserted.  This gives the \defn{de Rham lattice}
\[
\bigoplus_{j \in \ZZ} \coH[n+2j]{\Xo;\ZZ(j)} \to \KdR[n]{\X}
\] 
for any log manifold $\X$.  This lattice and its compactly supported counterpart are natural for morphisms of log Poisson manifolds.  They are also compatible with Gysin maps and Poincar\'e duality.  

\subsubsection{The Chern character lattice}
For a log manifold $\X$, let us denote by $\KB{\X}$ the topological K-theory of its interior $\Xo$, equipped with the classical analytic topology.  The Chern character gives a natural map
\begin{align*}
\ch : \KB[n]{\X} \to \bigoplus_{j \in \ZZ} \coH[n+2j]{\Xo;\QQ(j)} 
\end{align*}
which becomes an isomorphism after tensoring the K-theory with $\QQ$, by a theorem of Atiyah--Hirzebruch \cite[\S2.4]{Atiyah1961}.  Note that the Tate twists capture the conventional factors of $\tipi$ occurring in the Chern character in de Rham cohomology (defined by taking traces of powers of curvature forms); see also \cite{Getzler1993a} for the case of $\Ktop[1]{-}$, from which the twists in all other degrees are determined by Bott periodicity.  This defines the \defn{Chern character lattice}
\[
\ch : \KB{\X} \to \KdR{\X}
\]
for any log manifold $\X$.

\subsubsection{The charge lattice}

The Chern character is natural for pullbacks, but is not compatible with pushforwards and Poincar\'e duality; the correction is the content of the Riemann--Roch theorem. Let $\Ahat_{\X} \in \KdR[0]{\X}$ denote the $\Ahat$-class of the tangent bundle of $\X$; it is the multiplicative characteristic class  associated to the series
\[
\ahat(x) = \frac{x/2}{\sinh(x/2)} \in \QQ[[x]].
\]
Let $\Ahat_{\X}^{1/2} = 1 - \tfrac{1}{48}c_2(\X)+\cdots \in \KdR[0]{\X}$ be its (positive) square root. Note that the operator $\Ahat_{\X}^{1/2}\cup -$ on cohomology is real, and invertible over $\QQ$.  Moreover, since the tangent bundle extends to a vector bundle on $\cX$, we have $\Ahat_{\X}^{1/2} \in \W[0]\KdR[0]{\X}$, so that this operator preserves the weight filtration. Thus, multiplication by $\Ahat_{\X}^{1/2}$ takes lattices to lattices, and we may define the \defn{charge lattice}
\[
\charge_\X :=  \Ahat_\X^{1/2} \cup \ch : \KB{\X} \to \KdR{\X}
\]
which sends the class of a vector bundle $\E$ to the element $c_\X(\E)=\Ahat_\X^{1/2}\ch(\E)$, commonly known as its Mukai vector. For us, the charge lattice will be the most important, as it corresponds naturally to the K-theory of the quantization. We therefore make the following definition.

\begin{definition}
Let $\X$ be a log Poisson manifold.  The \defn{K-theory of $(\X,\ps)$} is the mixed Hodge structure $\K{\X,\ps}$ defined by the mixed $\RR$-Hodge structure $\KR{\X,\ps}$ and the charge lattice $\charge_\X : \KB{\X} \to \KdR{\X}$. 
\end{definition}

There is also a variant with compact supports.  Namely, the compactly supported K-theory $\KBc{\X}$ maps to $\KdRc{\X}$ by the Chern character, and the latter is a $\KdR{\X}$--module, so that multiplication by $\Ahat_{\X}^{1/2}$ still makes sense, and we may make the following definition.
\begin{definition}
The \defn{compactly supported K-theory of $(\X,\ps)$} is the mixed Hodge structure $\Kc{\X,\ps}$ defined by the mixed $\RR$-Hodge structure $\KRc{\X,\ps}$  and the charge lattice $\charge_\X : \KBc{\X} \to \KdRc{\X}$. 
\end{definition}

An important subtlety is that, while $\ch$ is a ring homomorphism, $\charge_\X$ is not a ring homomorphism.  Rather, since $\Ahat^{1/2}_\X\in\KdR{\X}$ is central, $\charge_\X$ is a $\KB{\X}$-bimodule homomorphism with respect to the Chern character, i.e.
\[
\charge_\X(e' \cdot e'')  = \ch(e') \charge_\X(e'') = \charge_\X(e')\ch(e'')
\]
for all $e',e''\in\KB{\X}$.
Combining this observation with \autoref{thm:chern-classes}, we deduce the following useful fact.
\begin{lemma}\label{lem:coabelian-K-mult}
If $(\E,\eta) \to (\X,\ps)$ is a coabelian vector bundle, then multiplication by $[\E]\in \KB{\X}$ defines an endomorphism of the mixed Hodge structure $\K{\X,\ps}$.  If $\rank \E=1$, then this endomorphism is an automorphism.
\end{lemma}
Similarly, \autoref{cor:canonical-bundle-R-hodge} has the following consequence.  As we shall see in subsequent work, this reflects the fact that (up to a shift in degree), tensoring by the canonical bundle is the Serre functor of the derived category, and the latter lifts to an autoequivalence of the category of sheaves on the quantization.
\begin{corollary}\label{cor:canonical-bundle-K}
	If $(\X,\ps)$ is a log Poisson manifold, then multiplication by the class of the canonical bundle $[\det \forms[1]{\X}] \in \KB[0]{\X}$ is an automorphism of the mixed Hodge structure $\K{\X,\ps}$.
\end{corollary}

\subsection{Functoriality}
Note that if  $f : (\X,\ps) \to (\Y,\eta)$ is a morphism of log Poisson manifolds, the pullback maps $f_{\Bet}^* : \KB{\Y}\to\KB{\X}$ and $f_\dR^* : \KdR{\Y} \to \KdR{\X}$ are intertwined by the de Rham and Chern homomorphisms, and therefore define a morphism of mixed Hodge structures with respect to those lattices.

However the maps $f^*_\Bet$ and $f^*_\dR$ are not intertwined by the charge homomorphism; rather, they differ by multiplication by $\Ahat_f^{1/2}$ where $\Ahat_f := \Ahat_{\X}/f^*_{\dR}\Ahat_{\Y}$ is the $\Ahat$-class of the relative tangent complex
\[
\cT{f}:= (\cT{\X} \stackrel{\dd f}{\to} f^*\cT{\Y})\,.
\]
This means that the homomorphism $f_\Bet^* \otimes \CC$  may fail to preserve the Poisson Hodge filtration, so that the Hodge structure $\K{\X,\ps}$ is not functorial for arbitrary morphisms.  A simple example of this phenomenon is the following.
\begin{example}\label{ex:projection-to-a-point}
Let $(\X,\ps) \to *$ be the projection to a point.  Note that $\K[0]{*} = \ZZ$ is a pure Hodge structure of Hodge type $(0,0)$: the Hodge filtration is $\F[0]\KdR{*}=\CC$ and $\F[1]\KdR{*}=0$.  Moreover $f_\Bet^*(1) = [\cO{\X}] \in \KB[0]{\X}$ is the class of the trivial line bundle.  Thus $f^*_\Bet$ induces a morphism of mixed Hodge structures if and only if $\charge_\X(\cO{\X}) \in \Fps[0]\KdR[0]{\X}$.  But
\[
\charge_\X(\cO{\X}) = \Ahat^{1/2}_\X \cup \ch(\cO{\X}) = \Ahat^{1/2}_\X \cup 1 = \Ahat^{1/2}_\X.
\]
We conclude that the projection to a point induces a morphism of mixed Hodge structures $f^*: \K{*} \to \K{\X,\ps}$ if and only if $\Ahat^{1/2}_\X \in \Fps[0]\KdR[0]{\X}$. 

For instance, suppose that $(\X,\ps)$ is a K3 surface equipped with the Poisson structure defined by a holomorphic symplectic form.  We have $\Ahat_\X^{1/2} =1 - \tfrac{1}{48}c_2(\X)$ and $c_2(\X) \neq 0$.  We always have $1= e^{\hookps/u}1 \in \Fps[0]$, but by \autoref{ex:log-symp-Hodge}, we have $\Fps[0]\cap \HdR[4]{\X} = 0$.  Therefore $c_2(\X) \notin \Fps[0]$, and hence $\Ahat_\X^{1/2} \notin \Fps[0]$, so the pullback along the projection $(\X,\ps) \to *$ does not induce a morphism of Hodge structures in this case.
\end{example}

As we shall see in a subsequent paper, the question of whether $c(\cO{\X}) \in \Fps[0]\KdR{\X}$ is related to whether the structure sheaf $\cO{\X}$ deforms to a sheaf on the canonical quantization of $(\X,\ps)$; this may fail because the quantization may be twisted by a gerbe, which is the source of the problem for K3 surfaces.

Thus with the charge lattice, we may only expect to retain functoriality for a special class of maps: those that lift to the canonical deformation quantization, or more generally those for which the Hodge-theoretic obstruction to such a lift vanish.  This has the advantage that it allows one to access nontrivial properties of the quantization from geometry.  This motivates the following definition.
\begin{definition}
Let $f :  (\X,\ps) \to (\Y,\eta)$ be a morphism of log Poisson manifolds.  We say that $f$ is \defn{K-quantizable} if $f_\Bet^*$ defines a morphism of mixed Hodge structures $f^* : \K{\Y,\eta}\to\K{\X,\ps}$, or equivalently,
\[
\Ahat_{f}^{1/2} f_\dR^* : \KR{\Y,\eta} \to \KR{\X,\ps}
\]
is a morphism of mixed $\RR$-Hodge structures.

We say that $f$ is \defn{properly K-quantizable} if it is proper, K-quantizable, and the pullback $f_\Bet^! $ defines a morphism of mixed Hodge structures $f^! : \K{\Y,\eta}\to\K{\X,\ps}$, or equivalently,
\[
\Ahat_{f}^{1/2}f_\dR^! : \KdRc{\Y} \to \KdRc{\X}
\]
is a morphism of mixed $\RR$-Hodge structures.
\end{definition}

The following useful condition is immediate from the definition, and \autoref{thm:chern-classes}.
\begin{lemma}\label{lem:ahat-kquant}
If $[\cT{f}] \in \KB[0]{\X}$ lies in the $\QQ$-subalgebra generated by coabelian vector bundles over $(\X,\ps)$, then $f$ is K-quantizable.  If, in addition, $f$ is proper, then it is properly K-quantizable.
\end{lemma}

\subsection{Open embeddings and \'etale maps}

By an \defn{open embedding} (respectively, an \defn{\'etale map}) we mean a morphism of log Poisson manifolds $f :(\X,\ps) \to (\Y,\eta)$ that restricts to an open embedding (resp.~\'etale map) of the interiors.  Clearly every open embedding is \'etale.  Moreover,  the relative tangent complex of an \'etale map is trivial in the interior.  Hence by \autoref{lem:ahat-kquant} we have the following.
\begin{corollary}
    All \'etale maps are K-quantizable.
\end{corollary}

In particular, this applies to the natural map $\X \to \cX$ modelling the inclusion of the interior.  Combined with the description of the lowest weight part of $\K{\X,\ps}_\RR$ from \autoref{thm:R-MHS} part (\ref{it:lowest-wt}), we deduce the following.
\begin{corollary}
The open embedding $\X\to\cX$ induces a morphism of mixed Hodge structures $\K[n]{\cX,\ps}\to\K[n]{\X,\ps}$ whose image is  $\grW[n]\K[n]{\X,\ps}$, for all $n\in\ZZ$.
\end{corollary}

On the other hand, every weak equivalence is an open embedding, which immediately implies the following corollary. 
\begin{corollary}
    If $\phi : (\X,\ps) \to (\Y,\eta)$ is a weak equivalence, then the pullback $\phi^*:\K{\Y,\eta}\to\K{\X,\ps}$ is an isomorphism of mixed Hodge structures.
\end{corollary}
Thus if $(\U,\ps)$ is any  holomorphic Poisson manifold that admits a logarithmic model, we obtain a mixed Hodge structure on K-theory depending only on the weak equivalence class of the model.  In some cases, one can obtain stronger statements of independence:
\begin{example}[\twopureex]\label{ex:two-pure-uniqueness}
    Suppose $\U$ is a  two-pure smooth algebraic variety (\autoref{ex:two-pure-formality}) such as a torus (\autoref{ex:toric-mixed}), so that its cohomology is the space of global logarithmic forms.  Since Deligne's mixed Hodge structure is functorial, the condition that a global holomorphic form on $\U$ is logarithmic is independent of the choice of algebraic log model.  Moreover, the action of $\hookps$  on such a form is determined by its restriction to the interior, by continuity.  Hence, for two-pure Poisson varieties, we obtain a mixed Hodge structure $\K{\U,\ps}$ that depends only on the algebraic structure of $\U$ and the Poisson structure $\ps$. It would be interesting to know if a similar statement holds for all smooth algebraic Poisson varieties that admit logarithmic models.
\end{example}

\subsection{Zeros of the Poisson structure}

While maps \emph{to} the zero Poisson structure need not be K-quantizable, maps \emph{from} the zero structure always are:
\begin{lemma}
\label{ex:zero-of-sigma}
If $\Y$ is a log manifold equipped with the zero Poisson structure, then every morphism $\Y \to (\X,\ps)$ is K-quantizable.
\end{lemma}
\begin{proof}
Every holomorphic vector bundle on $\cY$ can be equipped with the zero Poisson structure, making it coabelian,  so the result follows from \autoref{lem:ahat-kquant}.
\end{proof}

Consider the case in which $\X$ is connected.  Then we have a canonical $\ZZ$-linear functional 
\[
\chi : \KB[0]{\X} \to \ZZ
\]
sending the class of a bounded complex of vector bundles $\calE^\bullet$ on $\Xo$ to its fibrewise Euler characteristic
\[
\chi(\calE^\bullet) = \sum_j(-1)^j \rank(\calE^j)\,.
\]
(If $\X$ is disconnected, we have such a functional for every connected component.)  Note that this is exactly the morphism given by pullback to a point in $\Xo$, which is independent of the point by homotopy invariance.   In particular, the induced map on the de Rham side
\[
\chi\otimes \CC = \chi_{\dR} : \KdR[0]{\X}\to \CC
\]
is the projection onto the zeroth cohomology. Note that since $\chi$ comes from a map of log manifolds, it automatically preserves the weight filtration. Moreover if the Poisson structure on $\X$ vanishes at $p$ then $\chi_\dR$ will preserve the Poisson Hodge filtration by \autoref{ex:zero-of-sigma}, or equivalently, it will annihilate $\Fps[1]\KdR[0]{\X}$.  Thus we have the following.

\begin{corollary}\label{cor:zeros-chi}
	If $\ps$ vanishes at a point in the interior $\Xo$, then $\chi$ is a morphism of mixed Hodge structures, i.e.~$\chi_\dR(\beta)= 0$ for all $\beta \in \Fps[1]\KdR[0]{\X}$.
\end{corollary}

This gives a Hodge-theoretic obstruction to the existence of zeros of the Poisson structure, which is, in general, nontrivial:
\begin{example}[\toricex]
Consider a toric log Poisson manifold $(\X,\ps)$ with toric coordinates $x_1,\ldots,x_n$ so that $\ps = \sum \lambda_{ij}x_ix_j\cvf{x_i}\wedge\cvf{x_j}$ for a skew-symmetric matrix of constants $\lambda_{ij} \in \CC$. Note that the rank of $\ps$ in $\Xo = (\Gm)^n$ is constant, so $\ps$ vanishes in the interior if and only if it is identically zero.  By \autoref{ex:toric-hodge}, the subspace $\Fps[1] \subset \KdR[0]{\X}$ contains the elements $\beta_{ji} := u\dlog{x_j}\wedge\dlog{x_i} + \lambda_{ij}$, for $i < j$.  Thus, the images $\chi_{\dR}(\beta_{ji}) = \lambda_{ij}$ are exactly the components of the  bivector.   We deduce that $\chi$ is a morphism of mixed Hodge structures if and only if $\ps$ has a zero in $\Xo$, in accordance with \autoref{cor:zeros-chi}.
\end{example}

\subsection{Reduced K-theory}
\label{ex:reduced-K}
From \autoref{ex:projection-to-a-point}, the projection to a point $(\X,\ps) \to *$ is K-quantizable if and only if $\Ahat^{1/2}_\X \in \Fps[0]\KdR[0]{\X}$, in which case $\K{*}$ embeds as a mixed Hodge substructure in $\K{\X,\ps}$, so we may define the \defn{reduced K-theory} as the quotient
\[
\tK{\X,\ps} := \K{\X,\ps}/\K{*}
\]
with the induced mixed Hodge structure.  Note that if $\Xo$ contains a zero of the Poisson structure, we obtain a splitting $\K{\X,\ps}\cong \tK{\X,\ps}\oplus \ZZ$ by \autoref{cor:zeros-chi}, but in general such a splitting as a direct sum of Hodge structures need not exist; see, e.g. \autoref{ex:P2-quantum}.

\subsection{The canonical involution and the opposite Poisson structure}

The $\CC$-linear Mukai involution $(-)^\vee = (-1) \diamond (-)$ on $\KdR{-}$ from \autoref{sec:deRham-dual} has a natural K-theoretic interpretation:  for a vector bundle $E$ we have $\ch[j](E) = (-1)^j\ch[j](E^\vee)$ where $E^\vee$ is the dual bundle, so that $\ch(E^\vee) = \ch(E)^\vee$.  Since the $\Ahat$ class is even, it is invariant under $(-)^\vee$ and we deduce that
\[
c_\X(E)^\vee = c_\X(E^\vee)
\]
for all $E \in \KB{\X}$.  Combining this fact with the $\RR^\times$-equivariance of $\KR{\X,\ps}$ from \autoref{thm:R-MHS} part (\ref{it:equivariance}), we deduce the following.
\begin{proposition}
For any log Poisson manifold $(\X,\ps)$ the canonical involution $E \mapsto E^\vee$ on K-theory  induces isomorphisms of mixed Hodge structures
\[
\K{\X,\ps} \cong \K{\X,-\ps} \qquad \Kc{\X,\ps} \cong \Kc{\X,-\ps}\,.
\]
\end{proposition}

\subsection{The index pairing and duality}\label{sec:duality}

Poincar\'e duality for topological K-theory is induced by the index pairing
\[
\pair{-,-}_\Bet : \KB{\X} \times \KBc[-\bullet]{\X} \to \ZZ \qquad (E,F) \mapsto \textrm{index}(E^\vee \otimes F)\,,
\]
where $\textrm{index}$ denotes the index of a compactly supported K-theory class, given by the K-theoretic pushforward to a point $\KBc[0]{\X} \to \KBc[0]{*} \cong \ZZ$.   This agrees with the analytic index defined using Dirac operators, by the Atiyah--Singer index theorem.

The index pairing is \emph{not} compatible with the integration pairing $(-,-)_\dR$ on $\KdR{\X}$ from \autoref{sec:deRham-dual}.  Rather, let us recall that the \defn{C\u{a}ld\u{a}raru--Mukai pairing} on $\KdR{\X}$ is the modification of the integration pairing defined by
\[
\pair{\omega,\nu}_\dR := \pairnaive{e^{-c_1(\X)/2}\omega,\nu}_\dR\,.
\]
By \autoref{cor:canonical-bundle-R-hodge}, $e^{-c_1(\X)/2}$ is an automorphism of $\KR{\X,\ps}$.  Combining this with Poincar\'e duality for $\pairnaive{-,-}_\dR$ (\autoref{cor:Poincare-duality}), we deduce that $\pair{-,-}_\dR$ is also a pairing of mixed $\RR$-Hodge structures. 
\begin{proposition}
The charge lattice is isometric with respect to the C\u{a}ld\u{a}raru--Mukai pairing, i.e.~
\[
\pair{E,F}_{\Bet} = \pair{\charge_\X(E),\charge_\X(F)}_\dR
\]
for all $E\in \KB{\X}$ and $F \in \KBc{\X}$.  Therefore the index pairing defines a perfect pairing of integral mixed Hodge structures
\[
\pair{-,-} : \K{\X,\ps} \otimes_\ZZ \Kc[-\bullet]{\X,\ps} \to \ZZ\,.
\]
\end{proposition}
\begin{proof}
This is the topological counterpart of the discussion in \cite[\S3]{Caldararu2005a}.  We have
\begin{align*}
\pair{\charge_\X(E),\charge_\X(F)}_\dR &= \int_{\cX} e^{-c_1(\X)/2}\charge_\X(E^\vee)\charge_\X(F) \\
&= \int_{\cX} e^{-c_1(\X)/2}\Ahat_\X \ch(E^\vee\otimes F) \\
&= \textrm{index}(E^\vee\otimes F)
\end{align*}
by the differentiable Riemann--Roch formula~\cite{Atiyah1959}.
\end{proof}

\subsection{Relative K-theory}

Let $f : (\Y,\eta) \to  (\X,\ps)$ be a K-quantizable morphism. Its \defn{relative K-theory} $\K{f,\ps,\eta}$ is defined as follows.  First, the lattice is given by the relative topological K-theory $\KB{f}$, which is the K-theory of the cone of the map $f^\circ : \Yo \to \Xo$.   On the de Rham side, this corresponds to the homotopy fibre of the map
\[
(\alpha_f \cup - ) \circ f^*_\dR : \Mix{\X}((u)) \to \Mix{\Y}((u))
\]
of periodic complexes, where $\alpha_f$ is any representative of the class $\Ahat_{f}^{1/2}$ in de Rham cohomology.  By definition of K-quantizability, this map preserves the weight and Hodge filtrations up to homotopy, and hence induces a mixed Hodge structure on the relative de Rham theory, as desired.  Evidently there is a dual version of this construction, replacing $\K{-}$ with $\Kc{-}$ and $f^*$ with $f_*$. 

In this way, we obtain long exact sequences of mixed Hodge structures 
\[
\begin{tikzcd}
\K{f,\ps,\eta} \ar[r] & \K{\X,\ps} \ar[r,"f^*"] & \K{\Y,\eta} \ar[r]& \K{f,\ps,\eta}[1]
\end{tikzcd}
\]
and dually
\[
\begin{tikzcd}
\Kc{f,\ps,\eta}[-1] \ar[r] & \Kc{\Y,\eta} \ar[r,"f_*"] & \Kc{\X,\ps} \ar[r]& \Kc{f,\ps,\eta}\,.
\end{tikzcd}
\]
In the particular case in which $f:\Y\to\X$ is the embedding of a log Poisson submanifold, we denote the relative K-theory by
\[
\K{\X,\Y,\ps} := \K{f,\ps,\ps|_\Y}
\]
and similarly in the compactly supported case.

\subsection{K-theory of coabelian vector bundles}\label{sec:coabelian-K}

Throughout the present \autoref{sec:coabelian-K},  $(\E,\eta)\to(\X,\ps)$ is a coabelian vector bundle in the sense of \autoref{sec:coabelian-constructions},  with underlying holomorphic vector bundle $\E'\to\cX$.

\begin{proposition}
The following statements hold.
\begin{enumerate}
\item The projection $p : \E \to \X$ is K-quantizable.
\item The zero section $i : \X \to \E$ is properly K-quantizable.
\item The projection $q : \PP(\E) \to \X$ is properly K-quantizable.
\end{enumerate}
\end{proposition}
\begin{proof}
Let $\calE$ be the locally free sheaf of sections of $\E'$.  The relative tangent complexes are given by
\[
\cT{p} = p^*\calE \qquad \cT{i} = \calE[-1] \qquad \cT{q} \cong \calE (1) / \cO{\PP(\E)}
\] 
where the third isomorphism uses the relative Euler sequence on $\PP(\E')$, so the result follows from \autoref{lem:ahat-kquant}.
\end{proof}

From this basic fact, we derive the compatibility of our Hodge structure with various classical constructions.

\begin{corollary}[Homotopy invariance] The pullback in K-theory gives mutually inverse isomorphisms $p^* : \K{\X,\ps} \to \K{\E,\eta}$ and $i^* : \K{\E,\eta} \to \K{\X,\ps}$.
\end{corollary}

\begin{corollary}[Thom isomorphism]\label{cor:thom}
The Thom isomorphism for K-theory gives an isomorphism of mixed Hodge structures
\[
\begin{tikzcd}
 \K{\E,\E\setminus 0,\eta} \ar[r,"\sim"] & \K{\X,\ps}\,.
\end{tikzcd}
\]
\end{corollary}

\begin{proof}
    The Thom isomorphism in the interiors is obtained by composing the isomorphisms in K-theory induced by maps of pairs
    \[
      (\Eo,\Eo \setminus 0) \to   (\PP(\Eo \oplus \cO{\Xo}),\PP(\Eo\oplus \cO{\Xo})\setminus 0) \leftarrow (\PP(\Eo \oplus \cO{\Xo}),\PP(\Eo))
    \]
    with the proper pushforward along $\PP(\Eo\oplus \cO{\Xo}) \to \Xo$.  From the discussion above, all of these maps have  K-quantizable or properly K-quantizable logarithmic models, as needed.
\end{proof}

Recall, in addition, that we have coabelian line bundles $\cO{\PP(\E)}(j)$ for all $j \in \ZZ$, and multiplication by their K-theory classes give Hodge automorphisms of $\K{\PP(\E),\PP(\eta)}$ by \autoref{lem:coabelian-K-mult}.  Hence our Hodge structure is compatible with the usual formula for the K-theory of a projective bundle:
\begin{corollary}[Projective bundle formula]
Let $r = \rank \E$.  Then the map
\[
\begin{tikzcd}[row sep=2]
\K{\X,\ps}^{\oplus r} \ar[r] &  \K{\PP(\E),\PP(\eta)} \\
(\beta_0,\ldots,\beta_{r-1}) \ar[r,mapsto] & \sum_{0 \le j < r} [\cO{\PP(\E)}(j)] \cup \beta_j
\end{tikzcd}
\]
is an isomorphism of mixed Hodge structures.  
\end{corollary}

As the very special case in which the base is a point, we have the following.
\begin{corollary}\label{cor:P^n-hodge}
If $\ps$ is any Poisson structure on the projective space $\PP^d$, then we have isomorphisms of mixed Hodge structures
\[
\K[n]{\PP^d,\ps} \cong \K[n]{\PP^d} \cong  \begin{cases}
\ZZ(-j)^{\oplus (n+1)}  & n=2j \\ 0 & \textrm{otherwise}
\end{cases}\,,
\]
so that $\K{\PP^d,\ps}$ is independent of $\ps$.
\end{corollary}

\begin{remark}\label{rmk:P^n-K-basis}
Concretely, there are two natural bases for $\K[0]{\PP^d}$: the successive powers $1,t,\ldots,t^{d}$ where $t = [\cO{\PP^d}(-1)]$ is the class of the tautological bundle, or the elements $e_0,e_1,\ldots,e_d$, where $e_j$ is the  class of the structure sheaf of a codimension-$j$ linear subspace.  They are related by the change of basis formulae $e_j=(1-t)^j$, and have Hodge type $(0,0)$ for any Poisson structure on $\PP^d$.  The  $e_j$ basis will be useful in calculations below.
\end{remark}

\subsection{K-theory of coabelian submanifolds}

Let $ f : (\Y,\eta) \hookrightarrow (\X,\ps)$ be the embedding of a closed coabelian submanifold, and let  $(\N,\ps_\N)$ denote the coabelian normal bundle as in \autoref{ex:coabelian-normal}.   The relative tangent complex is the shifted normal bundle $\cT{f} \cong \mathcal{N}[-1]$ so that \autoref{lem:ahat-kquant} gives the following.
\begin{lemma}
The embedding $f$ is properly K-quantizable.
\end{lemma}
We shall use this fact to establish the compatibility of our Hodge structure with various constructions associated to $\X, \Y$ and $\N$.

\subsubsection{Blowups}  Let $\Bl{\X}{\Y}$ be the blowup as in \autoref{sec:compactifications}, equipped with the induced Poisson structure $\tilde \ps$.  We have a Cartesian diagram of log manifolds
\begin{equation}
\begin{tikzcd}
\PP(\N)\ar[d,"p"] \ar[r,"j"] &\Bl{\X}{\Y} \ar[d,"b"] \\
\Y \ar[r,"i"] & \X
\end{tikzcd}\label{eq:blowup-square}
\end{equation}
whose K-theories are related as follows.
\begin{proposition}[Blowup formula]
For a coabelian submanifold $\Y\subset \X$, all maps in the diagram \eqref{eq:blowup-square} are properly K-quantizable.  Therefore we have canonical isomorphisms of mixed Hodge structures
\[
\begin{tikzcd}
\K{\Bl{\X}{\Y},\tilde \ps} & \ar[l,"b^*+j_!"'] \K{\X,\ps} \oplus \frac{\K{\PP(\N),\ps_\N}}{p^*\K{\Y,\eta}} \cong \K{\X,\ps}\oplus \K{\Y,\eta}^{\codim \Y-1} 
\end{tikzcd}
\]
where the second isomorphism is obtained from the projective bundle formula.
\end{proposition}

\begin{proof}  
The horizontal maps in \eqref{eq:blowup-square} are embeddings of coabelian submanifolds and $p$ is a projective bundle, so it remains to confirm that $b$ is properly K-quantizable.  The relative tangent complex of $b$ is trivial away from $\Y$ and isomorphic to the relative tangent sheaf of $p$ over $\Y$, so that we have the standard identity
\[
[\cT{b}] = [j_*(p^*\cN{} - \cO{\PP(\N)}(-1))]
\]
in the K-theory of $\Bl{\X}{\Y}$; see, e.g.~\cite[Lemma 15.4, parts (i) and (iv)]{Fulton1998}. Let $\sL := \cO{\Bl{\X}{\Y}}(\PP(\N))$ be the invertible sheaf defining the exceptional divisor; it is a coabelian line bundle by \autoref{ex:log-line-bundle}.  Then the normal bundle of the exceptional divisor is $j^*\sL|_{\PP(\N)} \cong \cO{\PP(\N)}(-1)$, which combined with the Grothendieck--Riemann--Roch formula implies that all characteristic classes of  $\cT{b}$ are polynomials in $c_1(\sL)$ and the classes of the form $j_! p^*\gamma$ where $\gamma \in \KdR{\X}$ is a characteristic class of the normal bundle $\N$.

It therefore suffices to prove that the operations $(j_! p^* \gamma) \cup b^*(-)$ preserve  Hodge structures.  For this, suppose that $\beta \in \KdR{\X}$.  We have
\[
(j_!p^*\gamma) \cup b^*\beta = j_!(p^*\gamma\cup j^*b^*\beta) = j_!(p^*\gamma \cup p^*i^*\beta) = j_!p^*(\gamma\cup i^*\beta)\,.
\]
The result now follows since $\gamma$ is a characteristic class of the coabelian  vector bundle $\N$, and the morphisms $j$, $p$ and $i$ are properly K-quantizable.
\end{proof}

\subsubsection{Specialization maps}

Suppose that $\Y \subset (\X,\ps)$ is a closed coabelian submanifold.  In this subsection, we abuse notation and think of $\bA^1 = \PP^1 \setminus \{\infty\}$ as a log manifold, equipped with the zero Poisson structure.  Then $\Y \times \bA^1$ and $\X\times \{0\}$ are coabelian submanifolds of $\X\times\bA^1$.

\begin{definition}
The \defn{degeneration to the normal cone} is the log Poisson manifold $\tX = \Bl{\X\times \bA^1}{\Y\times\bA^1} \setminus (\X\times \{0\})'$ where $'$ denotes the strict transform.
\end{definition}

The degeneration to the normal cone comes equipped with a natural morphism $\tX \to \bA^1$ whose fibre over 1 is identified with $\X$ and whose fibre over 0 is identified with the logarithmic normal bundle $\N$ of $\Y\subset \X$.  In addition, the strict transform of $\Y \times \bA^1$ in the blowup gives an embedding  $\tY := \Y \times \bA^1 \hookrightarrow \tX$.  We thus have a commutative diagram of K-quantizable maps:
\begin{equation}
\begin{tikzcd}
\N\setminus \Y  \ar[r] \ar[d] &\widetilde{\X} \setminus \tY \ar[d] & \ar[l] \X \setminus \Y\ar[d]\\
\N  \ar[r] &\widetilde{\X} & \ar[l]  \X
\end{tikzcd}\,.\label{eq:normal-cone}
\end{equation}

\begin{lemma}\label{lem:normal-cone}
The maps in \eqref{eq:normal-cone} induce an isomorphism of mixed Hodge structures
\[
\K{\X, \X\setminus \Y,\ps}  \cong \K{\N, \N\setminus \Y, \ps_\N}\,.
\]
\end{lemma}
\begin{proof}
We need to show that the horizontal maps in \eqref{eq:normal-cone} induce isomorphisms on the relative K-theory of the vertical pairs.  Since the maps are K-quantizable, the compatibility with Hodge structures is immediate, so it suffices to verify that they induce an isomorphism on $\KB{-}$.  But this is standard:  excising the complement of a $C^\infty$ tubular neighbourhood of $\tY^\circ$ in $\tX^\circ$, we can assume that $\tX^{\circ}\to\bA^1$ is topologically trivial, so the result follows from homotopy invariance.
\end{proof}

This allows us to make the following definition.
\begin{definition}\label{def:specialization}
For a closed coabelian submanifold $(\Y,\eta)\subset (\X,\ps)$, the \defn{specialization map} is the isomorphism of mixed Hodge structures
\[
 \K{\X ,\X\setminus \Y,\ps} \stackrel{\sim}{\rightarrow} \K{\Y,\eta}
\]
defined as the composition of the isomorphism from \autoref{lem:normal-cone} with the K-theoretic Thom isomorphism from \autoref{cor:thom}. 
\end{definition}

\subsubsection{The Gysin sequence}

For a closed coabelian submanifold $\Y \subset \X$, we may compose the specialization map of \autoref{def:specialization} with the connecting homomorphism $\K{\X\setminus\Y} \to \K{\X,\X\setminus\Y}[1]$ to obtain the \defn{K-theoretic residue map}
\[
\res  : \K{\X\setminus\Y,\ps} \to \K{\Y,\eta}[1]\,.
\]
\begin{remark}
Due to the appearance of the Thom isomorphism in the definition of the specialization map, the K-theoretic residue map differs from the usual residue map in de Rham cohomology by a correction involving the Todd class. 
\end{remark}
\begin{remark}
In contrast with the residue maps in cohomology, there are no Tate twists inserted here; these are taken care of by the twists defining $\K{-}$.
\end{remark}

Recall that the Gysin long exact sequence for the embedding $\Y\hookrightarrow \X$ is the long exact sequence for the pair $(\X,\X\setminus \Y)$, but with the relative K-theory replaced by $\K{\Y}$ using the specialization map.  We thus arrive at the following.

\begin{corollary}[Gysin sequence]\label{thm:gysin}
Suppose that $i : (\Y,\eta) \to (\X,\ps)$ is a closed coabelian submanifold with complement  $j : \X\setminus \Y \to \X$.  Then the K-theoretic Gysin sequence defines a long exact sequence of mixed Hodge structures
\[
\begin{tikzcd}
\K{\Y,\eta}\ar[r,"i_!"] & \K{\X,\ps} \ar[r,"j^*"] & \K{\X\setminus \Y,j^*\ps} \ar[r,"\res"] & \K{\Y,\eta}[1]\,.
\end{tikzcd}
\]
\end{corollary}

\begin{example}[\surfaceex]\label{ex:P2-gysin}
As a special case of \autoref{ex:surface-constr}, let $\X$ be the log manifold $\PP^2 \setminus \Y$ where $\Y$ is a smooth cubic curve.  Thus $\Y$ has genus one (i.e.~it is an elliptic curve, but without a preferred basepoint).  Let $\ps$ be a Poisson structure on $\X$. 
 Using the isomorphism $\K{\PP^2,\ps}\cong\K{\PP^2}$ from \autoref{cor:P^n-hodge}, the Gysin sequence associated to the maps
\[
\begin{tikzcd}
\Y \ar[r,"i"] & \PP^2 & \PP^2\setminus \Y \ar[l,"j"']
\end{tikzcd}
\]	
takes the form
\[
\xymatrix{
\cdots \ar[r] & \K[0]{\Y} \ar[r] & \K[0]{\PP^2} \ar[r] & \K[0]{\X,\ps}\ar[r] & \K[1]{\Y} \ar[r] & 0 \ar[r] & \cdots\,.
}
\]

Since $\Y$ is topologically a two-torus, the Chern character and charge homomorphism agree, and give an isomorphism $\K[1]{\Y}\cong\coH[1]{\Y;\ZZ}$.  Meanwhile, $\K[0]{\Y} \cong \ZZ\cdot \{[\cO{\Y}],[\cO{p}]\}$ where $p \in \Y$ is a point and $\K[0]{\PP^2} \cong \ZZ \{e_0,e_1,e_2\}$ where $e_j$ is the class of a codimension-$j$ linear subspace as in \autoref{rmk:P^n-K-basis}.   Since $i_![\calE] = [i_*\calE]$ for any coherent sheaf $\calE$ on $\Y$, a straightforward calculation using Koszul resolutions shows that
\begin{align*}
i_![\cO{p}] &= e_2 & i_![\cO{\Y}] &= 3(e_1-e_2)\,.
\end{align*}
Hence by exactness we find that
\[
\img(j^*) = \ZZ j^*e_0 \oplus \tfrac{\ZZ}{3\ZZ} j^*e_1 \subset \K[0]{\X}\,,
\]
so that we have an exact sequence of mixed Hodge structures
\begin{equation}
\xymatrix{
0\ar[r] & \ZZ \oplus \ZZ/3\ZZ \ar[r] & \K[0]{\PP^2\setminus \Y,\ps} \ar[r] & \coH[1]{\Y;\ZZ} \ar[r] & 0\,,
}\label{eq:P^2-extension}
\end{equation}
which gives the associated graded
\[
\grW[n] \K[0]{\PP^2\setminus \Y,\ps} = \begin{cases}
\ZZ \oplus \ZZ/3\ZZ  & n=0 \\
\coH[1]{\Y;\ZZ} & n=1 \\
0 & \textrm{otherwise}
\end{cases}\,.
\]
Note that since $\coH[1]{\Y;\ZZ}$ is a free abelian group, the exact sequence of abelian groups \eqref{eq:P^2-extension} is always split.  However, the sequence typically does not split in the category of mixed Hodge structures; we shall return to this in \autoref{ex:P2-quantum} below.
\end{example}

\begin{example}[\sklyaninex]\label{ex:sklyanin-gysin}
In a similar vein, consider a Sklyanin Poisson structure on $\PP^3$ as in \autoref{ex:sklyanin-constr} with vanishing set an elliptic normal curve $\Z$ given by the base locus of a pencil of quadrics.  Arguing as in the previous example, the Gysin sequence for $\Z \subset \PP^3$ gives an extension
\[
\xymatrix{
0 \ar[r] & \ZZ \{j^* e_0, j^* e_1\} \oplus  \frac{\ZZ}{4\ZZ} j^*e_2 \ar[r] & \K[0]{\PP^3 \setminus \Z,\ps} \ar[r] & \coH[1]{\Z;\ZZ}\ar[r] & 0\,.
}
\]
Since the maps $\PP^3 \setminus \Z \to \PP^3 \to *$ are K-quantizable, and the Poisson structure on $\PP^3 \setminus \Z$ vanishes at the singular points of the pencil, we deduce from \autoref{ex:reduced-K} that $\K{\PP^3\setminus \Z,\ps} \cong \tK{\PP^3\setminus \Z,\ps} \oplus \K{*}$ and hence the zeroth K-theory is determined by the simpler exact sequence
\[
\xymatrix{
0 \ar[r] & \ZZ \oplus \ZZ/4\ZZ \ar[r] & \tK[0]{\PP^3 \setminus \Y,\ps} \ar[r] & \coH[1]{\Z;\ZZ}\ar[r] & 0\,,
}
\]
similar to the surface case.
\end{example}

%% file: poisson-mhs-periods.tex
\section{Semiclassical period maps}
\label{sec:periods}
We now consider the period maps that describe the behaviour of the Hodge structures  $\K{\X,\ps}$ in families, leading to the notion of quantum parameters and the results on the Torelli problem from the introduction.  

\subsection{Families and the Kodaira--Spencer map}
Let $\sfS$ be a complex analytic space. A \defn{family of log Poisson manifolds $\famX$ over $\sfS$} is a complex analytic space $\famXc$ equipped with a smooth and proper morphism $\famXc \to \sfS$ with K\"ahler fibres, a divisor $\fambdX \subset \famXc$ that is normal crossings relative to $\sfS$, and a relative logarithmic Poisson bivector $\ps \in \coH[0]{\der[2]{\famX/\sfS}}$.  Each fibre $\X_s := f^{-1}(s) \subset \famX$ for $s \in \sfS$ is then a coabelian submanifold, and the underlying $C^\infty$ manifolds form a locally trivial fibre bundle over $\sfS$.  

By standard Poisson deformation theory~\cite{Kim2020}, we have a Kodaira--Spencer map
\[
\KodSpenc : \cT{\sfS} \to R^2f_*(\der{\famX/\sfS},[\ps,-])\,,
\]
where the right hand side is the fibrewise Poisson cohomology sheaf of the family. Concretely, using the fibrewise Dolbeault resolution of $\der{\famX/\sfS}$ , the relative Poisson cohomology is computed by 
	\begin{align*}
 Rf_*(\der{\famX/\sfS},[\ps,-]) \cong f_*(\mathcal{A}_{\famX/\sfS}^{0,\bullet} \otimes_{\mathcal{O}_{\sfS}} \wedge^{\bullet}\mathcal{T}_{\famX/\sfS}^{1,0}, \overline{\partial} + [\sigma,-])\,.
	\end{align*}
 Picking a $C^\infty$ local trivialization around any point $s_{0}\in \sfS$, the nearby fibres are diffeomorphic to  $\XX_{0}:=f^{-1}(s_{0})$, and with respect to this trivialization, the family of log manifolds is  determined by 
 a smooth family of Maurer--Cartan elements $(\sigma(s), I(s)) \in \Dolb[0,0]{\wedge^{2}\cT{\X_0}} \oplus \Dolb[0,1]{\cT{\X_0}}$.  The Kodaira--Spencer map is then represented at cochain level by the derivative of $(\ps(s),I(s))$ :
 \[
 \kappa|_{s_0} = 
 [\dd_\sfS(\sigma(s),I(s))|_{s_0}] \in \mathsf{T}^*_{s_0}\sfS \otimes \coH[2]{\der{\X_0},[\ps,-]}\,.
 \]
 
\subsection{Variations of Hodge structures}\label{subsec:variations}
Given a family of log Poisson manifolds $(\famX,\ps)/\sfS$ as above, we denote by $\sKB{\famX}$ the sheaf associated to the topological K-theory presheaf $\U \mapsto \KB{f^{-1}(\U) \cap \famXo}$.  Since $\famX$ is topologically locally trivial, the sheaf $\sKB{\famX}$ is locally constant and the weight filtration on the fibres gives a canonical filtration of $\sKB{\famX}$ by locally constant subsheaves $\W\sKB{\famX}$.

We denote by $\sKDol{\famX}$ (resp.~$\sKdR{\famX}$) the corresponding sheaves of fibrewise Poisson homology (resp.~periodic cyclic Poisson homology).  These are obtained by taking cohomology (resp. periodic cohomology) of the mixed complex $\sMix{\famX/\sfS,\ps} = (\forms[-\bullet]{\famXc/\sfS}(\log\fambdX),\delps,\dd)$ fibrewise, giving relative versions of the hypercohomology and Hodge--de Rham spectral sequences:
\[
R^qf_*\forms[p]{\famX/\sfS} \Rightarrow \sKDol[q-p]{\famX} \qquad \sKDol{\famX}((u)) \Rightarrow \sKdR{\famX}\,.
\]
It follows from \autoref{prop:degen} together with the cohomology and base change theorem that $\sKDol{\X}$ and $\sKdR{\X}$ are locally free sheaves of $\cO{\sfS}$-modules, and the Hodge filtration on the latter has associated graded
\[
\grF \sKdR{\X} \cong \sKDol{\X}((u))\,.
\]
In addition, the relative charge homomorphism gives a map
\[
\charge_\famX : \sKB{\famX} \to \sKdR{\famX}
\]
inducing an isomorphism $\sKB{\famX}\otimes \cO{\sfS} \cong \sKdR{\famX}$.

The Hodge filtration $\Fps$ defines a holomorphic section of the relevant flag bundle
\[
\phi : \sfS \to \Flag[]{\sKdR{\famX}}
\]
and its covariant derivative relative to the local system $\sKB{\famX}$ defines an $\cO{\sfS}$-linear map
\[
\nabla \phi: \cT{\sfS} \to \phi^*\cT{\Flag[]{\sKdR{\famX}}/\sfS}
\]
called the \defn{infinitesimal period map}.
We claim that it satisfies the Griffiths transversality condition, i.e.~that $\nabla\phi$ takes values in the subbundle
\[
\bigoplus_p \sHom{\Fps[p],\Fps[p-1]/\Fps[p]}  \hookrightarrow \phi^*\cT{\Flag[]{\sKdR{\famX}}/\sfS}\,.
\]
To see this, note that $e^{\hookps/u}$ gives a canonical identification
\[
\bigoplus_p \sHom{\Fps[p],\Fps[p-1]/\Fps[p]} \cong \bigoplus_p \sHom{\sKDol[p]{\famX,\ps},\sKDol[p+2]{\famX,\ps}}\,.
\]
On the other hand, the contraction of fibrewise polyvectors into forms gives an action of $\der{\famX}$ on $\sMix{\famX/\sfS,\ps}$, which induces a morphism
\[
\iota : Rf_*(\der{\famX/\sfS},[\ps,-]) \to \sEnd{\sKDol{\famX,\ps}}\,.
\]
The following result is then the Poisson version of Griffiths' infinitesimal period relation for families of complex manifolds; when $\X = \cX$ is compact, it is a special case of Baraglia's results \cite{Baraglia14} on period maps for families of compact generalized K\"ahler manifolds.
\begin{proposition}\label{prop:InfPeriodMap}
The tuple $(\sKdR{\famX},\sKB{\famX},\F,\W,\charge)$ defines a variation of mixed Hodge structures, with infinitesimal period map given by
\[
\nabla \phi = \iota \circ \KodSpenc  : \cT{\sfS} \to \bigoplus_p \sHom{\sKDol[p]{\famX},\sKDol[p+2]{\famX}} \subset \phi^*\cT{\Flag[]{\sKdR{\famX}}}\,.
\]
\end{proposition}

\begin{proof}
We may assume that $\sfS$ is a disk with coordinate $s$ and trivialize the family near $\X_0=f^{-1}(s_0)$ as above, giving the family of Maurer--Cartan elements $(\ps(s), I(s))$.  We then have a family of filtrations
\begin{align*}
e^{\iota_{\sigma(s)}/u}F^{\sbt}_{I(s)}\Dolb[\bullet]{\X_0}((u))
	\end{align*}
	on the $C^\infty$ differential forms on $\XX_{0}$, where we note that the ordinary Hodge filtration $F^{\sbt}_{I(s)}\Dolb[\bullet]{\X_0}((u))$ will vary with the complex structure $I(s)$.  Identifying the ``difference'' between the flags $F_{I(s)}$ and $F_{I(s_0)}$ with linear operators in the usual way, we can formally apply the Leibniz rule to deduce that the derivative is 
	\begin{align*}
           \left.\frac{d}{ds}\right|_{s=s_0}e^{\hook{\ps(s)}/u}\F_{I(s)}&= 
		\left.\frac{d}{ds}\right|_{s=s_0}(e^{\iota_{\sigma(s)}/u}) + e^{\iota_{\sigma(s)}/u} \left.\frac{d}{ds}\right|_{s=s_0} \F_{I(s)} \\
  &= e^{\iota_{\sigma(s)}/u}\cdot \iota_{\frac{d}{ds} \left( I(s) +\iota_{\sigma(s)/u}\right) } \\
  &=e^{\iota_{\sigma(s)}/u} \iota_{\kappa} 
	\end{align*}
 as desired, where we have used the fact from Hodge theory \cite[\S 2]{Griffiths68} that $\tfrac{d}{ds}F^{\sbt}_{I(s)}$ is given by contracting with the Kodaira-Spencer class $\kappa(I(s)) \in \mathsf{H}^{1}(\mathcal{T}_{\XX_{0}})$ of the deformation of complex structure. 
\end{proof}

\begin{remark}\label{rmk:moduli}
To connect with the picture sketched in the introduction, take $\famX/\bS$ to be the universal family over the moduli stack $\MPois$ of log Poisson manifolds.  The latter is a derived analytic stack whose  tangent space at $(\X,\ps) \in \MPois$ is identified with $\coH[2]{\der[\ge1]{\X},[\ps,-]}$ by the Kodaira--Spencer map, while the moduli stack of Hodge structures $\MHodge$ has tangent space equal to that of the flag variety.  Hence the differential of the global period map  $\wp : \MPois \to \MHodge$ is exactly the contraction map
\[
\dd \wp = \iota : \coH[2]{\der[\ge1]{\X},[\ps,-]} \to \Hom{\KDol{\X,\ps},\KDol[\bullet+2]{\X,\ps}}.
\]
This can be used to check local injectivity of the period map in various cases of interest, the standard example being the following.
\end{remark}

\begin{example}\label{ex:logCY}
Let $\X$ be a \defn{log Calabi--Yau manifold}, i.e.~a log manifold $\X$ for which $\bdX\subset \cX$ is an anticanonical divisor.  Equivalently, the log canonical bundle $\forms[n]{\X}$ is trivial, where $n=\dim\X$.  

A choice of trivialization $\mu \in \coH[0]{\forms[n]{\X}}$ gives an isomorphism $\der{\X} \cong \forms[n-\bullet]{\X}$, so a bivector $\ps \in \coH[0]{\der[2]{\X}}$ is equivalent to an $n-2$ form $\hookps\mu \in \coH[0]{\forms[n-2]{\X}}$.  
Since global forms are closed, and $\hook{[\ps,\ps]}=[\hookps,[\dd,\hookps]]$, we deduce that $\ps$ is automatically Poisson.
Moreover, since $\dd(\hookps\mu)=0$, the form $\mu$ is invariant under Hamiltonian flows, i.e.~$\ps$ is automatically unimodular in the sense of Weinstein~\cite{Weinstein1997}.  Equivalently, the form $e^{\hookps}\mu$ is closed.  We then have the Hodge subspace
\[
\Fps[n]\KdR[-n]{\X} = \CC \cdot e^{\hookps}\mu \subset \bigoplus_{j \ge0} \coH[0]{\forms[n-2j]{\X}} \subset \KdR[-n]{\X}
\]
which is independent of $\mu$. It determines $\hook{\ps} \in \Hom{\coH[0]{\forms[n]{\X}},\coH[0]{\forms[n-2]{\X}}}$, and hence $\ps$, by projection to $\coH[0]{\forms[n]{\X}}\oplus\coH[0]{\forms[n-2]{\X}}$ as in the toric case (\autoref{ex:toric-hodge}). 

On the other hand, as explained by Katzarkov--Kontsevich--Pantev~\cite[Lemma 4.19]{Katzarkov2008}, the logarithmic Bogomolov--Tian--Todorov lemma implies that deformations of $\X$ as a log manifold are unobstructed, so that $\X$ has a universal deformation over an open subset $\sfS_0\subset \coH[1]{\cT{\X}}$, whose tangent space is identified with $\Hom{\coH[0]{\forms[n]{\X}},\coH[1]{\forms[n-1]{\X}}}$ by the infinitesimal period map. 

Combining these, we deduce that any log Calabi--Yau Poisson manifold $(\X,\ps)$ has a universal deformation $(\X_s,\ps_s)_{s\in\sfS}$ over the base $\sfS = \sfS_0 \times \coH[0]{\der[2]{\X}}$, and the period map of the resulting variation of mixed Hodge structures $\K{\X_s,\ps_s}$ gives an immersion
\[
\mapdef{\wp_\X}{\sfS}{\PP(\KdR[-n]{\X})}{s}{[\phi_s(e^{\hookps}\mu_s)]}
\]
where $\phi_s : \K{\X_s}\to\K{\X}$ denotes the  parallel transport of the Gauss--Manin connection.  This is a log Poisson counterpart of the results for compact generalized Calabi--Yau manifolds in \cite{Baraglia14,Huybrechts2005a}.
\end{example}

%% file: poisson-mhs-adams.tex
\subsection{Adams-equivariance}\label{subsec:Adams}

Recall that the Adams operations $\psi_n,n\in\ZZ$ give a natural action of the multiplicative monoid $(\ZZ,\cdot)$ on topological K-theory $\KB{-}$, uniquely determined by the condition that $\psi_n([L]) = [L^n]$ for every $n \in \ZZ$ and every topological line bundle $L$.  It extends to a linear action of the multiplicative group $\QQx$ on $\KB{-}\otimes \QQ$.  To match our earlier conventions, we will consider the ``inverse'' action, 
\[
\diamond : \QQx \times \KB{-}\otimes \QQ \to \KB{-}\otimes \QQ
\]
determined uniquely by 
\[
\frac{1}{n} \diamond [L] = [L^n]\,.
\]
Then the Chern character is equivariant, in the sense that 
\[
\ch(\lambda \diamond e) = \lambda \diamond\ch(e)\,,
\]
where the action $\diamond$ on $\KdR{-}$ is as in \autoref{sec:flags}.
As a result, the charge homomorphism is equivariant if and only if $\Ahat_\X = 1$.  In this case, by equivariance of Poisson--Hodge filtrations (\autoref{lem:equivariant}), the mixed Hodge structure $\K{\X,\ps}$, together with the action $\diamond$, gives the prototypical example of the following.

\begin{definition}
An \defn{Adams-equivariant mixed Hodge structure} is a mixed Hodge structure $\V = (\V_\dR,\V_\Bet,\F,\W,c)$ together with compatible linear actions
\[
\begin{tikzcd}[row sep=1em]
 \QQx \ar[r,hook] & \Gm \\
 \V_\Bet \otimes \QQ \ar[r,hook,"c"]\ar[loop above] & \V_\dR\ar[loop above]
\end{tikzcd}
\]
both denoted by $\adm$, such that the following conditions hold:
\begin{enumerate}
\item The action of $\QQx$ preserves the weight filtration $\W$.
\item The family of filtrations $\F_\hbar := \hbar \adm \F$ for $\hbar \in \Gm$ satisfies the Griffiths transversality condition.
\item The limit $F_0 := \lim_{\hbar \to 0} \hbar\adm \F$ is a Hodge filtration, i.e.~satisfies the opposedness axiom relative to $\W$.
\end{enumerate}
\end{definition}

Every Adams-equivariant mixed Hodge structure $\V$ comes equipped with a canonical filtration $\SF \V$ obtained by diagonalizing the action of $\QQx$:
\begin{definition}
For an Adams-equivariant mixed Hodge structure $\V$ and an integer $l$, the $l$th \defn{Adams eigenspace of $\V$} is the subspace 
\[
\cl[l]{\V}_\QQ \subset \V_\Bet \otimes \QQ
\]
on which $\QQx$ acts with weight $-l$.  We denote by $\SF \V_\Bet \otimes \QQ$ the increasing filtration by the Adams eigenvalue, i.e.
\[
\SF[j] \V_\Bet \otimes \QQ = \bigoplus_{l \le j} \cl[l]{\V}_\QQ = \set{ v \in \V_\Bet\otimes \QQ }{ \hbar^j \cdot(\hbar \adm v) \textrm{ has a limit as }\hbar \to 0}\,.
\]
We denote by $\SF \V_\Bet$ the induced filtration on the lattice $\V_\Bet$, and by
\[
\cl[l]{\V} := \grS[l]\V
\]
the associated graded, equipped with filtrations induced by $\F$ and $\W$.
\end{definition}

\begin{remark}
The name ``Adams eigenspace'' is justified by the fact that if $\V = \K{\X}$ is the K-theory of a log  manifold with $\Ahat_\X=1$, then $\cl[l]{\V}_\QQ = \cl[l]{\KB{\X}}_\QQ$ is the usual weight-$l$ Adams eigenspace in K-theory, which maps to $\HdR[\bullet+2l]{\X}$ via the Chern character.  Note, however, that the filtration $\SF$ is different from the usual Adams filtration on $K$-theory occurring in the Atyiah--Hirzebruch spectral sequence; the latter corresponds to the decreasing filtration by $l$, rather than the increasing one, i.e.\ they are Poincar\'e dual.  Note that since the filtration $\SF{}$ is defined using the rationalization, it ignores torsion in K-theory.  If $\coH{\X;\ZZ}$ is torsion-free, then the Atiyah--Hirzebruch filtration splits $\SF$ and we can identify $\cl[l]{\KB[n]{\X}} \cong \coH[n+2l]{\X;\ZZ(l)}$.
\end{remark}

Note that the definition implies that the tuple
\[
\V_{\hbar} := (\V_\dR, \V_\Bet,\F_\hbar,\W,c)
\]
with the same action of the multiplicative group, is an Adams-equivariant mixed Hodge structure for all $\hbar \in \bA^1$, giving a variation of mixed Hodge structures over $\bA^1$ that is ``$\Gm$-equivariant''.  The following is an analogue, for mixed Hodge structures, of the Rees construction relating filtered vector spaces to equivariant vector bundles over $\bA^1$:
\begin{proposition}
	The subgroups $s_j\V_\Bet \subset \V_\Bet$ for $j \in \ZZ$ define a filtration of $\V_\hbar$ by mixed Hodge substructures for all $\hbar \in \bA^1$, and the associated graded variation of Hodge structures is constant, i.e.~we have the equality of mixed Hodge structures
\[
\V_\hbar^{(j)} = \V_0^{(j)}
\]
for all $\hbar \in \bA^1$ and $j \in \ZZ$.
\end{proposition}

\begin{proof}
It suffices to treat the case $\hbar=1$, so that $\V_\hbar=\V$.  We proceed by induction on $j$.  Since $\V_\Bet \otimes \QQ$ is finite-dimensional, we have $\SF[j]\V_\Bet \otimes \QQ = 0$ for $j \ll 0$, in which case the statement is vacuous.  Thus suppose that $\SF[j-1]\V$ is a mixed Hodge substructure for some $j \in \ZZ$ with associated graded isomorphic to $\bigoplus_{l < j} \V^{(l)}_0$.  We must prove that the same holds for $\SF[j]\V$.  

By definition, the putative weight and Hodge filtrations on $\SF[j]\V_\Bet\otimes \QQ$ and $\SF[j]\V_\dR$ are the preimages of $\W$ and $\F$ under the inclusions in $\V_\Bet\otimes \QQ$ and $\V_\dR$, respectively, and similarly for $\SF[j-1]$.  Hence the inclusion $\SF[j-1]\V \to \SF[j]\V$ strictly preserves both filtrations, so by \cite[Criterion 3.10]{Peters2008}, it suffices to prove that the quotient map $\SF[j]\V \to \V^{(j)}_0$ strictly preserves the filtrations.  That it preserves $\W$ is immediate from the definition: $\W$ is constant.  Hence we must check that the image of $\F_{\hbar=1} \SF[j]\V$ is sent strictly to $\F_0 \V^{(j)}$.  For this, it suffices to note that if $v \in \SF[j]\V$, then the image of $v$ in the associated graded is equal to that of the limit $\lim_{\hbar \to 0}\hbar^j \cdot (\hbar \adm v)$.  The result then follows since
\[
\lim_{\hbar\to 0}\hbar^j (\hbar \adm \F_1) = \lim_{\hbar \to 0}\hbar^j\F_\hbar = \lim_{\hbar \to 0} \F_\hbar = \F_0
\]
by equivariance of the Hodge filtration.
\end{proof}

Since the associated graded of an Adams-equivariant mixed Hodge structure $\V$ is independent of $\hbar$, we denote it simply by 
\[
\V^{(j)} = \V_\hbar^{(j)} = \V_0^{(j)}.
\]
We then have for each $\hbar \in \bA^1$ and each $j \in \ZZ$, a short exact sequence
\[
\xymatrix{
0 \ar[r] & \cl[j-1]{\V} \ar[r] & \SF[j]\V_\hbar / \SF[j-2]\V_\hbar \ar[r] & \cl[j]{\V} \ar[r]& 0
}
\]
defining an extension class
\[
e^\V_j(\hbar) \in \ExtMHS[1]{\cl[j]{\V},\cl[j-1]{\V}}
\]
in the abelian category of mixed Hodge structures, which we analyze as follows.

\subsubsection{Extensions of Hodge structures}
Recall from \cite{Carlson1980} and \cite[\S3.5.1]{Peters2008} that for mixed Hodge structures $A$ and $B$, the extension group $\ExtMHS[1]{A,B}$ is a (possibly non-Hausdorff) abelian complex Lie group, whose group of connected components is $\ExtZ[1]{A_\Bet,B_\Bet}$, and whose identity component is given by the ``Jacobian''
\[
J(A,B) := \ExtMHS[1]{A',B'} \cong \frac{\W[0]\Hom[\CC]{A_\dR,B_\dR}}{\F[0]\W[0]\Hom[\CC]{A_\dR,B_\dR}+\W[0]\Hom[\ZZ]{A_\Bet',B_\Bet'}}
\]
where $A_\Bet'$ and $B_\Bet'$ denote the quotients of $A_\Bet$ and $B_\Bet$ by their torsion subgroups. We will write the group operation in $\ExtMHS[1]{A,B}$ multiplicatively, even though it corresponds to addition in the vector space $\Hom[\CC]{A_\dR,B_\dR}$.  

Note that the Lie algebra of $J(A,B)$ is the vector space
\[
\fj(A,B) := \frac{\W[0]\Hom[\CC]{A_\dR,B_\dR}}{\F[0]\W[0]\Hom[\CC]{A_\dR,B_\dR}}\,.
\]
It has a decreasing filtration $\F\fj(A,B)$ given by the image of the Hodge filtration on $\Hom{A_\dR,B_\dR}$.    An element  $x \in \fj(A,B)$ corresponds to an infinitesimal deformation of the trivial extension $A \oplus B$.   A straightforward calculation shows that if the resulting infinitesimal family of mixed Hodge structures satisfies the  Griffiths transversality condition, then $x$ is the image of the corresponding infinitesimal period map, and in particular we have
\[
x \in \F[-1]\fj(A,B) \cong \grF[-1] \W[0]\Hom{A_\dR,B_\dR}.
\]
For details, we refer the reader to the recent preprint \cite[\S2.3]{Aguilar2024}, where this was derived independently.

\subsubsection{Quantum parameters}\label{sec:qparam}
Returning to the case of an Adams-equivariant Hodge structure $\V$, note that since our extension classes $e^\V_j(\hbar) \in \ExtMHS{\cl[j]{\V},\cl[j-1]{\V}}$ vary continuously in $\hbar$, they are determined by the class $e^\V_j(0)$, which is the identity if $\cl[j]{\V}$ is torsion-free, and the elements
\[
q_{\V;j}(\hbar) := \frac{e_j^\V(\hbar)}{e^\V_j(0)} \in J(\cl[j]{\V},\cl[j-1]{\V})\,.
\]
We collect these into a function $q_\V = (q_{\V;j})_{j \in \ZZ} : \bA^1_{\hbar} \to \Q{\V}$ where 
\[
\Q{\V} :=  \prod_{j \in \ZZ} J(\cl[j]{\V},\cl[j-1]{\V})
\]
is the corresponding product of Jacobians.  Its Lie algebra
\[
\fq(\V): = \prod_{j\in\ZZ} \fj(\cl[j]{\V},\cl[j-1]{\V})
\]
has the induced filtration $\F \fq(\V) = \prod_j \F \fj(\cl[j]{\V},\cl[j-1]{\V})$.
\begin{definition}
The value $q_\V(\hbar = 1)\in \Q{\V}$ is called the \defn{quantum parameter} of the Adams-equivariant mixed Hodge structure $\V$.
\end{definition}

Let us denote by
\[
\iota_\V := \left.\frac{\dd q_\V(\hbar)}{\dd\hbar}\right|_{\hbar=0} \in \fq(\V)
\]
the derivative at $\hbar=0$ of the quantum parameter.  Then the behaviour of $q_\V$ is characterized as follows:

\begin{proposition}
For an Adams-equivariant mixed Hodge structure $\V$, we have
\[
\iota_\V \in \F[-1]\fq(\V)\,,
\]
and the function $q_\V:\bA^1_{\hbar} \to \Q{\V}$ is the one-parameter subgroup generated by $\iota_\V$, i.e. we have
\[
q_\V(\hbar) = \exp(\hbar \iota_\V) \in \Q{\V}
\]
for all $\hbar$.
\end{proposition}

\begin{proof}	
That $\hook{\V}$ lies in $\F[-1]\fq(\V)$ is immediate since an Adams-equivariant mixed Hodge structure satisfies the Griffiths transversality condition by definition, so let us prove that $q_\V$ is a one-parameter subgroup.

Since $0 \in \bA^1$ is sent to the identity in $\Q{\V}$, the map $q_\V$ lifts to the universal cover $\widetilde{\Q{\V}}\cong \fq(\V)$. We must prove that this lift is a straight line through the origin with constant velocity.  But this map is evidently equivariant with respect to the action of $\Gm$ on $\widetilde{\Q{\V}}$ induced by its action on $\V_\dR$. Considering the eigenvalues, we see that $\Gm$ acts on the vector space $\widetilde{\Q{\V}}$ by linear rescalings.  Hence the equivariance implies that the map $\bA^1 \to \widetilde{\Q{\V}}$ is linear, as desired. 

\end{proof}

For a log manifold $\X$ with $\Ahat_\X=1$, we denote by 
\[
\Q{\X} := \Q{\K[0]{\X}}\times \Q{\K[1]{\X}}
\]
the space of quantum parameters of its K-theory and by $\fq(\V)$ its Lie algebra; the  part relevant for quantum parameters is
\begin{align}
\F[-1]\fq(\V) \cong \prod_{p,q} \W[0]\Hom{\coH[q]{\forms[p]{\X}},\coH[q]{\forms[p-2]{\X}}}\,. \label{eq:lie-Q(X)}
\end{align}
Since the associated graded  $\grS\K{\X,\ps}\cong \grS\K{\X}$ is independent of $\ps$, we have a canonical quantum parameter
\[
q(\ps) \in \Q{\X}
\]
for every Poisson structure on $\X$, whose infinitesimal generator is  identified with the contraction operator $\hookps$, thanks to \autoref{prop:InfPeriodMap}.  Combining the results above, we obtain the following characterization of the ``global Torelli problem for quantum parameters'' on a fixed log manifold $\X$:
\begin{theorem}\label{thm:qparam-torelli}
Let $\X$ be a log manifold with $\Ahat_\X=1$.
\begin{enumerate}
\item If $\ps$ is a Poisson structure on $\X$, then
\[
q(\ps) = \exp(\hookps) \in \Q{\X}\,,
\]
where $\hookps$ is viewed as an element in $\fq(\V)$ via the isomorphism \eqref{eq:lie-Q(X)}.  
\item If $\ps$ and $\ps'$ are Poisson structures on $\X$, then $q(\ps)=q(\ps')$ if and only if the operators $\hookps,\hook{\ps'}$ on Dolbeault cohomology differ by an integral element, i.e.~there exists an element
\[
\eta \in \prod_{j}\prod_{n=0,1} \W[0]\Hom[\ZZ]{\K[n]{\X}^{(j)},\K[n]{\X}^{(j-1)}} 
\]
such that 
\[
\hookps - \hook{\ps'} = c(\eta) \mod \F[0]\End{\KdR{\X}}\,.
\]
\item If $\X$ is log Calabi--Yau, then the natural map
\[
\coH[0]{\der[2]{\X}} \to \Q{\X}
\]
sending a Poisson structure on $\X$ to its quantum parameter is an immersion of complex Lie groups.
\end{enumerate}
\end{theorem}

\subsection{Purity}\label{sec:purity} Note that for an Adams-equivariant mixed Hodge structure, the filtration $\SF \V$, and hence the quantum parameter $q_\V(1)$, depends \emph{a priori} on the additional structure of the $\QQ^\times$ action on $\V$.  It is therefore not intrinsic to the underlying mixed Hodge structure in general.  However, in several examples of interest, this filtration is a  re-indexing of the weight filtration, and is therefore essentially intrinsic.  For instance, the following is a natural sufficient (but not necessary) condition; see, e.g.~the introduction of \cite{Cirici2020} for some additional context regarding this notion. 
\begin{definition}
	Let $\X$ be a log manifold and suppose that $\alpha \in \{\tfrac{3}{2},2\}$.   We say that $\X$ is \defn{$\alpha$-pure} if the Hodge structure $\coH[n]{\X;\QQ}$ is pure of weight $\alpha n$ for all $n$.
\end{definition}

Note that the condition of $\alpha$-purity immediately implies that the Adams eigenspace $\K[n]{\X}^{(j)}\otimes \QQ \cong \coH[n+2j]{\X;\QQ(j)}$ is pure of weight $\alpha n + 2j(\alpha-1)$, so that the successive eigenspaces $\K[n]{\X}^{(j)}$ and $\K[n]{\X}^{(j+1)}$ differ in weight by
\[
2(\alpha-1)= \begin{cases}
1 & \alpha = \tfrac{3}{2} \\
2 & \alpha = 2
\end{cases}
\]
and thus $\SF[j]\K[n]{\X} = \W[\alpha n+2j(\alpha-1)]\K[n]{\X}$ is a re-indexing of the weight filtration.  (These are the only values of $\alpha$ for which this phenomenon occurs, which is why we do not consider other values as in \cite{Cirici2020}.)  We then have the following.
\begin{proposition}\label{prop:pure-params}
If $\X$ is a log manifold that is $\alpha$-pure with $\alpha \in \{\tfrac{3}{2},2\}$, then the following statements hold.
\begin{enumerate}
\item We have $\Ahat_{\X} = 1$, so that $\K{\X,\ps}$ is Adams-equivariant for all Poisson structures $\ps$ on $\X$. 
\item\label{it:3/2-pure} If $\alpha = \tfrac{3}{2}$, then $\Q{\X}$ is a compact complex torus, i.e.~is isomorphic to $\CC^n/\Lambda$ for some $n \ge 0$ and some lattice $\Lambda \subset \CC^n$.
\item\label{it:2-pure} If $\alpha = 2$, then $\Q{\X}$ is an affine algebraic torus, i.e.~is isomorphic to $(\CCx)^n$ for some $n \ge 0$.
\end{enumerate}
\end{proposition}

\begin{proof}
For (1), note that we have $\W[0]\K[0]{\X} = \SF[0]\K[0]{\X} \cong \coH[0]{\X;\ZZ}$.  But $\Ahat_{\X}$ is a characteristic class of a holomorphic bundle on $\cX$, so it has weight zero.  It must therefore lie in $\HdR[0]{\X}$, and thus be equal to its constant term, which is 1.  

For (2), note that if $\alpha = \tfrac{3}{2}$, then the space of quantum parameters is the product of the Jacobians of the Hodge structures $\Hom{\K{\X}^{(j)},\K{\X}^{(j-1)}}$ which are pure of weight $-1$.  Hence the result follows from \cite[Example 3.30]{Peters2008}.

Finally, for (3), note that if $\alpha=2$ then standard bounds on Hodge numbers~\cite[Corollaire 3.2.15]{Deligne1971a} imply that the only nontrivial Hodge number of $\coH[j]{\X;\QQ}$ is $h^{j,j}$, so that the cohomology is pure Tate.  Hence the space of quantum parameters is a product of copies of the groups $\ExtMHS{\ZZ(j),\ZZ(j+1)}$, each of which is isomorphic to $\CCx$ by \cite[Example 3.34 (1)]{Peters2008}.
\end{proof}

\subsection{Examples of period maps}\label{sec:examples-of-periods} We now give some examples that illustrate the general structural features of the period map derived above.

\begin{example}[\toricex]\label{ex:toricPeriods}
Let $\X$ be a toric log manifold of dimension two with toric coordinates $(x,y)$ so that $\Xo \cong (\CCx)^2$; the higher-dimensional case is similar and is discussed in \autoref{sec:qtori} below. The Chern character gives an isomorphism
\[
\KB[0]{\X} \cong \cl[0]{\KB[0]{\X}} \oplus \cl[1]{\KB[0]{\X}}  \cong \ZZ \oplus \coH[2]{\X;\ZZ(1)} 
\]
sending the class of a vector bundle to its rank and first Chern class.  Meanwhile $\K[1]{\X}$ is concentrated in Adams weight 0 and thus has no nontrivial quantum parameters.  We conclude that the space of quantum parameters for $\X$ is 
\[
\Q{\X} \cong \frac{\Hom[\CC]{\HdR[2]{\X},\HdR[0]{\X}}}{\Hom[\ZZ]{\coH[2]{\X;\ZZ(1)},\coH[0]{\X;\ZZ}}} \cong \frac{\coH[0]{\forms[2]{\X}}^\vee}{\coH[2]{\X;\ZZ(1)}^\vee} 
\]
and the quantum parameter $q(\sigma)$ of a Poisson structure $\ps$ is the image of the contraction operator $\hookps : \coH[0]{\forms[2]{\X}} \to \coH[0]{\forms[0]{\X}}$, viewed as a linear functional on global logarithmic two-forms.

To understand this concretely, we note that a basis for the K-theory is given by the unit
\[
1 = [\cO{\X}] \in \cl[0]{\KB[0]{\X}}
\]
and an element
\[
\beta = [\calE]-1 \in \cl[1]{\KB[0]{\X}}\,,
\]
where $\calE$ is any line bundle whose first Chern class generates $\coH[2]{\X;\ZZ(1)}$, e.g.~the Deligne line bundle~\cite{Deligne1991}; see also~\cite[p.~15]{Brylinski2000}.  The latter is a non-algebraic holomorphic bundle, whose first Chern class in de Rham cohomology is
\[
c(\beta) = \tfrac{1}{\tipi} \dlog{y}\wedge\dlog{x} \in \coH[0]{\forms[2]{\X}} \cong \HdR[2]{\X}\,.
\]
Now if $\ps = \lambda xy\cvf{x}\wedge\cvf{y}$ for some $\lambda \in \CC$, then the contraction operator is given by
\[
\hookps c(\beta) = \hook{\lambda xy\cvf{x}\wedge\cvf{y}} \rbrac{\tfrac{1}{\tipi}\dlog{y}\wedge\dlog{x}} = \tfrac{\lambda}{\tipi} \in \CC\,,
\]
which is an integer if and only if $e^\lambda = 1$.  Hence we have an isomorphism  $\Q{\X} \cong \CCx$ under which the quantum parameter is given by
\[
q(\ps) = e^{\lambda} \in \CCx.
\]
In this way, we see explicitly that $q(\ps)$ determines $\ps$ up to shifting $\lambda$ by an integral multiple of $\tipi$, in accordance with \autoref{thm:qparam-torelli}, and that $\Q{\X}$ is an affine algebraic torus, in accordance with \autoref{prop:pure-params}(\ref{it:2-pure}).
\end{example}

\begin{example}[\surfaceex]\label{ex:P2-quantum}
Let $\Y \subset \PP^2$ be a smooth cubic curve and let $\X = \PP^2 \setminus \Y$ be its complement, viewed as a log manifold. We examined its Hodge structure in \autoref{ex:P2-gysin}.  It is $\tfrac{3}{2}$-pure, so that $\SF = \W$, and the quantum parameter is the class of the extension \eqref{eq:P^2-extension} induced by the Gysin sequence. We therefore have a canonical isomorphism
\[
\Q{\X} \cong \Jac{\Hom{\coH[1]{\Y;\ZZ},\ZZ}} \cong \frac{\HdR[1]{\Y}^\vee}{\F[0](\HdR[1]{\Y}^\vee) + \coH[1]{\Y;\ZZ}^\vee} \cong \frac{\coH[0]{\forms[1]{\Y}}^\vee}{\Hlgy[1]{\Y;\ZZ}}\,,
\]
identifying $\Q{\X}$ with the Albanese torus of $\Y$, which is a compact complex torus in accordance with  \autoref{prop:pure-params}(\ref{it:3/2-pure}).  It is noncanonically isomorphic to $\Y$ itself, and since two smooth cubic curves are abstractly isomorphic if and only if they are projectively equivalent, we deduce that $\X$ is determined, up to isomorphism of log surfaces, by the torus $\Q{\X}$ obtained from $\grS{\K[0]{\X}}= \grW{\K[0]{\X}}$.

Note that the pairing of vector fields and forms gives an isomorphism of vector spaces $\coH[0]{\forms[1]{\Y}}^\vee \cong \coH[0]{\cT{\Y}}$.  The flow of vector fields then identifies $\Q{\X}$ with the identity component of $\Aut{\Y}$.
 Tracing through the definitions, we find that the quantum parameter of a Poisson structure $\ps$ corresponds to the automorphism
\[
q(\sigma) = \exp(\zeta) \in \Aut{\Y}\,,
\]
where $\zeta  \in \coH[0]{\cT{\Y}}$ is the vector field on $\Y$ that is dual to the residue of the logarithmic two-form $\ps^{-1} \in \coH[0]{\forms[2]{\X}}$; this was proposed in \cite[\S1.31]{Kontsevich2008a} as the parameter defining the quantization.  The vector field $\zeta$ has a more familiar description: up to an overall sign, it is the restriction to $\Y$ of the modular vector field of $\ps$ in the sense of Weinstein~\cite{Weinstein1997}.  Put differently, $\zeta$ is the connection vector field for the Poisson line bundle $(\det \forms[1]{\PP^2})|_\Y$ on $\Y$. Note that $\zeta$ uniquely determines $\ps$, since it depends linearly and nontrivially on $\ps$, and the space of Poisson structures is one-dimensional.

As a result, we have the following global Torelli property for elliptic Poisson planes: if $(\X,\ps)$ and $(\X',\ps')$ are two log Poisson surfaces as above, with associated genus-one curves $\Y,\Y'$ and vector fields $\zeta,\zeta'$, then there exists an isomorphism of mixed Hodge structures $\K[0]{\X,\ps}\cong \K[0]{\X',\ps'}$  respecting the canonical embedding $\ZZ\hookrightarrow \K[0]{-}$ if and only if there exists an isomorphism $\phi : \X \to \X'$ of log surfaces that intertwines the automorphisms $\exp(\zeta)\in\Aut{\Y}$ and $\exp(\zeta')\in\Aut{\Y'}$.  This mirrors the fact that the quantizations of these structures are the Feigin--Odesskii--Sklyanin~\cite{Feigin1989} elliptic algebras, whose defining relations are determined by an elliptic curve and an automorphism thereof.

Similar results hold for other rational log Calabi--Yau surfaces obtained from this example or the previous one by blowing up/down, although these are no longer $\alpha$-pure.  Namely, by results of Friedman~\cite{Friedman1984a,Friedman2016} and Gross--Hacking--Keel~\cite{Gross2015},  the mixed Hodge structure on the second cohomology, or equivalently the reduced K-theory, more or less determines the log surface.  Meanwhile, the quantum parameter determines the exponential of the modular vector field of the Poisson structure, mirroring the parameters defining the quantizations of these surfaces in the sense of \cite{VandenBergh2001,Rains2016}.
\end{example}

\begin{example}[\sklyaninex]
The strategy of the previous example, based on the Gysin sequence, applies also to higher-dimensional Poisson manifolds, even when the log manifold is not $\alpha$-pure.  For instance, consider an elliptic normal curve $\Z \subset \PP^3$ as in \autoref{ex:sklyanin-constr}. In this case, the complement $\X := \PP^3\setminus \Z$ is not $\alpha$-pure; indeed  from the Gysin sequence, one sees that the cohomology in degrees 0, 2 and 4 has weights $0,2$ and $5$, respectively.  However, we still have $\Ahat_\X=1$. 

The summands $\K[0]{\X}^{(j)} = \coH[2j]{\X;\ZZ(j)} \cong \coH[2j]{\PP^3;\ZZ(j)}\cong \ZZ$ for $j=0,1$ have no nontrivial extensions, so $\K[0]{\X}^{(0)} = \coH[0]{\X;\ZZ}$ does not contribute to the quantum parameter, and hence the latter depends only on the reduced K-theory.  The relevant extension is then determined by the Gysin sequence from \autoref{ex:sklyanin-gysin}, which expresses $\tK[0]{\X,\ps}$ as an extension of $\coH[1]{\Y;\ZZ}$ by $\ZZ \oplus \ZZ/4\ZZ$.   Thus, once again, we have a global Torelli property: $\grW \tK[0]{\PP^3\setminus \Z,\ps}$ determines  the log threefold $\PP^3 \setminus \Z$ up to isomorphism, and the quantum parameter determines the exponentiation of a vector field on the curve, mirroring the parameters defining the noncommutative Sklyanin algebras~\cite{Sklyanin1982}. Similar results can be derived for many other families of Poisson threefolds as in \autoref{ex:threefold-constr}.
\end{example}

%% file: poisson-mhs-qtori.tex
\section{The quantum torus}
\label{sec:qtori}

In this section, we give a precise relation between the quantum parameters of a Poisson structure and the parameters appearing in its canonical deformation quantization, in the special case of toric structures.  In other words, we compute the canonical quantization of torus-invariant Poisson structures on $(\CCx)^n$.

The strategy is as follows: first, we consider the universal family of \emph{noncommutative} tori given by the standard generators and relations presentation.  We associate to this family a variation of mixed Hodge structures via periodic cyclic homology and Getzler's Gauss--Manin connection, and we establish a global Torelli property. Then we compare this variation with that of the universal family of Poisson tori using the cyclic formality theorem of~\cite{CFW11,Shoikhet03,WWChains}.  It implies that the quantum parameter of a Poisson structure is  exactly the ``$q$-parameter'' appearing in the defining relation for its quantization.

The results in this section were developed in the first author's MSc thesis~\cite{Lindberg2020}.  We will therefore summarize the main results and provide precise references for details of the calculations as needed; we apologize that this section is therefore less self-contained than the rest of the paper, but we hope that the main ideas will come through.

\subsection{The parameter space}

Recall that an $n\times n$ matrix $q = (q_{ij})_{i,j}$ is \defn{multiplicatively skew-symmetric} if $q_{ii}=1$ and $q_{ij}=q_{ji}^{-1}$ for all $i,j$.  Such matrices form a Zariski-locally closed subset of all matrices, which we denote by
\[
\QP[n] \subset \CC^{n\times n}.
\]
We denote by $\allones[n] \in \QP[n]$ the matrix whose entries are all equal to one.  Let $L = \ZZ^n$ be the free abelian group of rank $n$, with basis $l_1,\ldots,l_n$.  Note that $q$ may equivalently be viewed as a homomorphism $\wedge^2 L \to \CCx$, defined by $l_i\wedge l_j \mapsto q_{ij}$.  Thus we have an isomorphism of Lie groups
\[
\QP[n] \cong (\CCx)^{n \choose 2} \cong \Hom{\wedge^2 L , \CC^\times},
\]
which will be useful below.
\subsection{Quantum tori}
Given a multiplicatively skew-symmetric matrix $q \in \QP[n]$, the corresponding (algebraic) \defn{quantum torus of dimension $n$} is the associative $\CC$-algebra generated by invertible elements $x_1,\ldots,x_n$ satisfying the ``$q$-commutation relations'' $x_i x_j = q_{ij} x_j x_i$, i.e.
\[
A_q := \frac{\CC{\abrac{x_1^{\pm1},\ldots,x_n^{\pm1}}}}{(x_ix_j - q_{ij}x_jx_i)_{i,j}}\,.
\]
When $q = \allones[n]$, we obtain the usual coordinate ring of the torus $(\CCx)^n$:
\[
A_{\allones[n]} \cong \cO{}((\CCx)^n) = \CC[x_1^{\pm 1},\ldots,x_n^{\pm n}]\,.
\]
For arbitrary $q$, the monomials $x_1^{k_1}\cdots x_n^{k_n}$ form a $\CC$-basis of $A_q$, so that the algebras $A_q$ form a flat family of algebras over the parameter space $\QP[n]$, which we may view as a deformation quantization of $(\CCx)^n$.  We denote by
\[
\cA \in \QCoh{\QP[n]}
\]
the corresponding quasi-coherent sheaf of $\cO{\QP[n]}$-algebras.  Note that its underlying $\cO{\QP[n]}$-module is trivial, but the algebra structure in the fibres varies as a regular function of $q$. Moreover, each $A_q$ has a canonical action of the torus by algebra automorphisms, defined by rescaling the generators so that the monomial $x_1^{k_1}\cdots x_n^{k_n}$ is a basis for the subspace of $A_q$ with torus weight $(k_1,\ldots,k_n) \in\ZZ^n$.  Thus the trivialization of $\cA$ as an $\cO{\QP[n]}$-module is torus-equivariant.

\subsection{Hodge--de Rham theory} Let $\KdR{A_q} := \HP[-\bullet]{A_q}$ denote the periodic cyclic homology of the $\CC$-algebra $A_q$; that is, the periodic cohomology of the non-positively graded Hochschild mixed complex
\[
\Mix[\bullet]{A_q} := ( A_q^{\otimes (1-\bullet)},b,B)\,,
\]
where $b$ is the Hochschild differential and $B$ is the Connes--Tsygan operator (see \cite[2.5.13]{Loday1998} and \cite{Connes1985, Tsygan1983}). Thus, as for the periodic cohomology of any mixed complex, we have a Hodge filtration
\[
\F\KdR{A_q}
\]
given by the powers of the periodizing variable $u$ as above. It was calculated explicitly in \cite[\S 3.2]{Lindberg2020}, using formulae from \cite{Wam97,Yas17}; we summarize the results as follows.

Let $\ft$ be the Lie algebra of the torus $(\CCx)^n$; it acts on $A_q$ by derivations.  Concretely, if $e_1,\ldots,e_n \in \ft$ is the basis given by the left invariant vector fields $\logcvf{x_1},\ldots,\logcvf{x_n}$, then we have
\[
e_i(x_1^{k_1}\cdots x_n^{k_n}) = k_i x_1^{k_1}\cdots x_n^{k_n},
\]
for all $(k_1,\ldots,k_n) \in \ZZ^n$, independent of $q$.

Let $\tau :  A_q \to  \CC$ be the $\CC$-linear functional defined by projection onto the torus invariant part.  Concretely,
\[
\tau(x_1^{k_1}\cdots x_n^{k_n}) = \begin{cases}
1 & k_1=\cdots = k_n = 0 \\
0 & \mathrm{otherwise}
\end{cases}\,.
\]
Define a pairing
\[
\abrac{-,-} : \wedge^k \ft \otimes_\CC \Mix[-k]{A_q} \to \CC
\]
by the formula
\[
\abrac{\xi_1\wedge\cdots \wedge \xi_k, a_0\otimes\cdots\otimes a_k} =\frac{1}{k!} \det\rbrac{\tau(a_0\xi_i(a_j)}_{i,j=1}^k\,.
\]
The results of \cite[\S3.1]{Lindberg2020} can then be summarized as follows.
\begin{proposition}\label{prop:HP-bundle}
The pairing $\abrac{-,-}$ and the inclusion of the $(\CCx)^n$-invariant subcomplex induce isomorphisms
\[
\KdR{A_q} \cong \KdR{A_q}^{(\CCx)^n} \cong \rbrac{\wedge^{-\bullet}\ft^\vee}((u))
\]
compatible with the Hodge filtrations.  When $q = \allones[n]$, this isomorphism is equal to the composition
\[
\begin{tikzcd}
\KdR{A_{\allones}} = \HP[-\bullet]{A_{\allones}} \ar[r,"\sim"] &  \KdR{(\CCx)^n}
\ar[r,"\sim"] & \bigoplus_j \wedge^{\bullet+2j}\ft^\vee
\end{tikzcd}
\]
of the Hochschild--Kostant--Rosenberg isomorphism for $A_{\allones} = \cO{}((\CCx)^n)$ with the isomorphism of \autoref{ex:toric-hodge} (for any log model of the torus).
\end{proposition}

In this way, the periodic cyclic homology groups of the algebras $A_q$ assemble into a locally free sheaf $\sKdR{\cA}$ over the parameter space $\QP[n]$, isomorphic to the trivial bundle with fibre $(\wedge^{-\bullet} \ft)((u))$.

\subsection{Getzler's Gauss--Manin connection}

In \cite{Getzler93}, Getzler constructed a flat connection on the periodic cyclic homology of a family of associative algebras over a formal disk; it is the noncommutative  analogue of the Gauss--Manin connection for smooth fibre bundles, and satisfies the Griffiths transversality condition. For quantum tori, it was computed in \cite[\S 3.2.2]{Lindberg2020}, building on the computations in \cite[\S 7]{Yas17} and the results of \cite[\S 4.4]{GS12}.

\begin{proposition}
Under the isomorphism of \autoref{prop:HP-bundle}, Getzler's Gauss--Manin connection on $\sKdR{\cA}$ is identified with the connection
\[
\nabla := \dd + u^{-1} \sum_{1\le i< j \le n} \dlog{q_{ij}} \hook{e_i \wedge e_j}
\]
on the trivial bundle $(\wedge^{-\bullet}\ft)((u))$, where $e_1,\ldots,e_n \in \ft$ is the canonical basis of left-invariant vector fields from above.
\end{proposition}
Note that in this model, the Griffiths transversality condition $\nabla \F[p]\subset\F[p-1]$ amounts to the fact that the connection involves $u^{-1}$ but no higher powers thereof.

Since the operators $\hook{e_i\wedge e_j}$ pairwise commute, the parallel transport of the connection is easily computed.  Namely, flat sections of $\sKdR{\cA}$ are identified with $\wedge^{-\bullet}\ft((u))$-valued functions of the form
\[
\exp\rbrac{ u^{-1}\sum_{i<j} \log q_{ij} \hook{e_i\wedge e_j}}v = \prod_{i < j}q_{ij}^{\hook{e_i\wedge e_j}/u}\cdot v
\]
for $v \in \wedge^{-\bullet} \ft((u))$.  Note that these functions are multi-valued, i.e.~Getzler's connection has monodromy.  More precisely, the monodromy representation based at $\allones[n]\in\QP[n]$ is given by 
\[
\mapdef{\rho}{\pi_1(\QP[n],\allones[n])}{\Aut{\wedge^{-\bullet}\ft((u))}}
{\gamma_{ij}}{\exp\rbrac{\frac{\tipi}{u} \hook{e_i\wedge e_j}}}
\]
where $\gamma_{ij}$ is the generator of $\pi_1(\QP[n],\allones[n])$ given by the family of matrices $q(t)$ whose entries are all constant except $q_{ij}(t) = q_{ji}(t)^{-1} = e^{\iu t}$  for $t \in [0,2\pi]$.

\subsection{Integral lattice}
We now incorporate an integral lattice.  Note that any reasonable choice must be covariant constant with respect to the Gauss--Manin connection.  Hence we can do this in a ``cheap'' way, by parallel transport of the lattice defined at the basepoint $q = \allones[n] \in \QP[n]$, but for this to work,  we need to verify that the lattice is preserved by the monodromy of the connection.  This is the approach we take here, but it can be reinterpreted intrinsically via the topological K-theory functor of Blanc~\cite{Blanc2016}; see \cite[Theorem 3.3.7]{Lindberg2020} in the case of families over formal disk. See also \cite{Elliott1984} for a similar result in $C^\ast$-algebraic context.

For $q = \allones[n]$, we have $A_q = \cO{}((\CCx)^n)$, and the topological K-theory is given by
\[
\KB[0]{(\CCx)^n} \cong \bigoplus_j \coH[2j]{\X;\ZZ(j)} \cong \bigoplus_j (\wedge^{2j} L)(-j)\,,
\]
where we have identified the free abelian group $L = \ZZ^n$ with $\coH[1]{(\CCx)^n;\ZZ(1)}$, so that the basis element $l_i \in L$ corresponds to the class of the logarithmic differential form $\dlog{x_i}$.  
 A similar isomorphism holds for the other K-theory groups.  This gives a lattice
\[
\KB{(\CCx)^n} \to \KdR{(\CCx)^n} \cong \bigoplus_j \wedge^{\bullet+2j}\ft^\vee \cong \KdR{A_{\allones}}
\]
via the Chern character (or equivalently, the charge) and the isomorphism of \autoref{prop:HP-bundle}. Note that the operator $\tipi \hook{e_i\wedge e_j}$ on $\wedge^\bullet \ft^\vee$ preserves this lattice and its weight filtration, and hence the image of  $\KB{(\CCx)^n}$ in $\KdR{A_q}$ via parallel transport of the Gauss--Manin connection is independent of the choice of path from $\allones$ to $q$, though the exact map does depend on this choice.

\begin{example}
    In the case $n=2$, the lattice in $\KdR[0]{A_q}$ at a point $q \in \QP \cong \CCx$ has a $\ZZ$-basis given by the elements $1$ and $\tfrac{1}{2\pi i} (u e_{1}\wedge e_{2} - \log q)$.  The latter vector depends on a choice of branch for $\log q$, but the filtered subgroup generated by these two vectors is independent of the choice.
\end{example}

In summary we have the following:
\begin{proposition}
The periodic cyclic homology of the universal noncommutative torus $\cA$, with its Hodge filtration and Gauss--Manin connection, lifts uniquely to a variation of mixed Hodge structures $\shfK{\cA}$ over $\QP$, whose fibre at $q = \allones$ is identified with $\K{(\CCx)^n}$ via the charge and Hochschild--Kostant--Rosenberg maps.
\end{proposition}

\subsection{Extraction of the quantum parameter}
Note that since the monodromy of $\shfK{\cA}$ is unipotent with respect to the weight filtration, the associated graded variation $\grW\shfK[0]{\cA}$ is trivial, so that parallel transport gives canonical isomorphisms
\[
\grW[j]\K[0]{A_q} \cong \grW[j] \K[0]{A_{\allones}} \cong \grW[j] \K[0]{(\CCx)^n} \cong  (\wedge^{2j}L)(-j)
\]
for all $q \in \QP$ and all $j\in \ZZ$.  In particular, for each $q \in \QP[n]$, we have a canonical extension
\begin{equation}
\begin{tikzcd}
0 \ar[r] & \ZZ \ar[r] & \W[2]\K[0]{A_q} \ar[r] &  (\wedge^2 L)(-1) \ar[r] & 0\,. 
\end{tikzcd}\label{eq:qtorus-ext}
\end{equation}
But the group of such extensions is
\begin{align}
\ExtMHS[1]{(\wedge^2 L)(-1),\ZZ}  \cong \Hom{\wedge^2 L,\CCx} \cong \QP
\label{eq:torus-Q-ext}
\end{align}
since $L$ is free and $\ExtMHS[1]{\ZZ(-1),\ZZ}\cong \CCx$.  A straightforward calculation then gives the following.
\begin{lemma}
Under the canonical isomorphism \eqref{eq:torus-Q-ext},  the class of the extension \eqref{eq:qtorus-ext} is given by
\[
[\W[2]\K[0]{A_q}] = q\,.
\]
\end{lemma}
\begin{corollary}[Global Torelli for noncommutative tori]\label{cor:qtori-torelli}
The period map 
\[
	\mapdef{\wp}{\QP[n]}{\ExtMHS[1]{(\wedge^2 L)(-1),\ZZ}}
{q }{[\W[2]\K[0]{A_q}]}
\]
is an isomorphism of complex Lie groups.
\end{corollary}

\subsection{Comparison with the canonical quantization}

We now compare the algebraic quantum tori above with the canonical quantization in the sense of Kontsevich.  To this end, observe that the universal cover of $\QP[n]$, as a map of pointed spaces, is given by the \emph{entrywise} exponential map
\[
\mapdef{\exp}{(\fq,0)}{(\QP[n],\allones[n])}
{(\lambda_{ij})_{ij}}{ (e^{\lambda_{ij}})_{ij}}
\]
where $\fq = \mathrm{Lie}(\QP) \cong \wedge^2 \CC^n$ is the space of skew-symmetric matrices.  Note that this map should not be confused with the \emph{matrix} exponential.

On the other hand, $\fq \cong \wedge^2 \CC^n= \wedge^2 \ft$ is canonically identified with the space of Poisson structures on any toric log compactification $\X$ of $(\CCx)^n$.   Hence we have a corresponding universal family
\[
(\famX,\ps) \to \fq
\]
of Poisson structures on $\X$, with $\famX = \X \times \fq$.  The quantum parameters of these structures include, as components, the class of the extension
\[
\begin{tikzcd}
0\ar[r] & \coH[0]{\X;\ZZ}=\ZZ \ar[r] & \W[2]\K[0]{\X,\ps} \ar[r] & \coH[2]{\X;\ZZ(1)}\cong (\wedge^2 L)(-1) \ar[r] & 0\,.
\end{tikzcd}
\]
Concretely, if $\lambda = (\lambda_{ij}) \in \fq$, the associated Poisson structure is 
\[
\ps_\lambda = \sum_{ij} \lambda_{ij}x_ix_j\cvf{x_i}\wedge\cvf{x_j}\,.
\]
Then, by pulling back the computations of the quantum parameters in \autoref{ex:toricPeriods} along the coordinate projections $(\CCx)^n \to (\CCx)^2$, we deduce that the extension  is classified by the element
\[
 ( e^{\lambda_{ij}})_{i,j} \in \QP\,.
\]
One readily checks that the full Hodge structure is completely determined by this component of the quantum parameter, so that $\K{A_q}\cong \K{\X,\ps_\lambda}$ if and only if $q = \exp(\lambda)$, i.e.~we have the following.
\begin{lemma}\label{lem:exp-variation}
There is a unique isomorphism
\[
\shfK{\famX,\ps} \cong \exp^*\shfK{\cA}
\]
of variations of mixed Hodge structures over $\fq$, that reduces to the Hochschild--Kostant--Rosenberg isomorphism $\KdR{\famX|_0} \cong \KdR{\X} \cong \KdR{A_{\allones}}$  at the origin.
\end{lemma}

On the other hand, in \cite{Kontsevich2003}, Kontsevich constructs a canonical formal deformation quantization for Poisson structures on open subsets of affine space, given by a $\CC[[\hbar]]$-bilinear product on $\CC[x_1,\ldots,x_n][[\hbar]]$ that is equivariant for linear changes of coordinates.  In particular, if $\ps$ is an invariant Poisson structure on the torus $(\CCx)^n$, then its canonical quantization must be isomorphic, by an $\hbar$-adically continuous isomorphism of  $\CC[[\hbar]]$-algebras, to a formal quantum torus:
\begin{align}
A'_{\hbar \ps} \cong \frac{\CC[[\hbar]]\abrac{x_1,\ldots,x_n}}{x_ix_j-q_{ij}(\hbar\ps)x_jx_i} \label{eq:formal-qtorus}
\end{align}
for some uniquely determined formal power series
\[
q_{ij}(\hbar\ps)  = 1 + \hbar \ps_{ij} + \cdots \in \CC[[\hbar]]\,,
\]
as discussed in the introduction.  More invariantly, Kontsevich's canonical quantization gives an isomorphism of the formal completions
\[
\quant : \widehat{\fq} \to \widehat{\QP[n]}
\]
so that the series of matrices $(q_{ij}(\hbar\ps))_{ij}$ is the composition of $\quant$ with the formal completion of the canonical family $\hbar \mapsto \hbar \ps$ over $\bA^1$.   Then the pullback $\quant^*\shfK{\mathcal{A}}$ gives the variation of mixed Hodge structures over the formal scheme $\widehat \fq$ associated to the family of formal noncommutative tori \eqref{eq:formal-qtorus}.   This is related to the other variations above as follows:

\begin{proposition}\label{thm:canonical-quant-mhs}
There is a unique isomorphism
\[
 \quant^* \shfK{\cA} \cong \left. \rbrac{\exp^* \shfK{\cA}}\right|_{\widehat \fq}
\]
as variations of mixed Hodge structures over $\widehat{\fq}$, which reduces to the identity on $\K{(\CCx)^n}$ at the origin.
\end{proposition}

\begin{proof}
By \autoref{lem:exp-variation}, it suffices to check that the HKR isomorphism at the origin extends to an isomorphism of variations $\shfK{\famX,\ps}|_\fq \cong \quant^*\shfK{\cA}$.  For this, we use
the cyclic formality theorem of Shoikhet~\cite{Shoikhet03} and Willwacher~\cite{WWChains}, which implies that the mixed complex $\Mix{A'_{\hbar\ps}}$ of the canonical quantization is quasi-isomorphic to the mixed complex of the affine Poisson variety $((\CCx)^n,\hbar\ps)$.  Moreover, by results of Cattaneo--Felder--Willwacher~\cite{CFW11}, this isomorphism trivializes the Getzler--Gauss--Manin connection.
\end{proof}

Put differently, \autoref{thm:canonical-quant-mhs} states that the compositions of $\quant$ and $\exp$ with the period map for noncommutative tori are equal.  But the latter is an isomorphism by the Torelli property (\autoref{cor:qtori-torelli}), so we must have $\quant = \exp|_\fq$ as maps of formal schemes.  In this way, we arrive at the following.
\begin{theorem}\label{thm:q=e^h}
The canonical quantization of the Poisson structure
\[
\ps = \sum_{i<j} \lambda_{ij} x_ix_j\cvf{x_i}\wedge\cvf{x_j}
\]
is isomorphic to the noncommutative torus \eqref{eq:formal-qtorus} with parameters
\[
q_{ij}(\hbar\ps) = e^{\hbar \lambda_{ij}}.
\]
\end{theorem}

\begin{remark}\label{rmk:grt}
    The argument given here was for the quantization defined by Kontsevich~\cite{Kontsevich2003}.  However, one can define deformation quantization with respect to any stable formality morphism in the sense of \cite{Dolgushev2021}.  The proof of \autoref{thm:canonical-quant-mhs} used only the following properties of Kontsevich's formality morphism: 1) it is equivariant for linear change of coordinates; 2) it is homotopic to a stable formality morphism that lifts to a stable formality morphism for cyclic chains; and 3) the lift to cyclic chains is compatible with the Getzler--Gauss--Manin connection.  Properties 1) and 2) hold for all stable formality morphisms: 1) is immediate from the definition, while 2) is the content of \cite[Theorem 1]{Willwacher2017}.  We expect that  3) also holds for all stable formality morphisms, but we are not aware of a reference for this fact.
\end{remark}

\subsection{Hodge classes and roots of unity}\label{sec:qtori-hodge-conj}
When the quantum parameter $q \in \QP[n]$ is torsion, i.e.~$q^j = \allones[n]$ for some $j \in \ZZ_{>0}$, the quantum torus $A_q$ has a large centre, given by the weight spaces that are divisible by $j$.   Note that by \autoref{cor:qtori-torelli}, this exactly corresponds to the case in which the extension class $[\W[2]\K[0]{A_q}]$ is $j$-torsion.  If we identify $\KB[0]{A_q}$ with the topological K-theory in the sense of Blanc~\cite{Blanc2016}, this has the following interpretation in terms of ``noncommutative cycle classes''.

For simplicity, we treat the case $n=2$, so we have the algebra
\[
A_q = \frac{\CC\abrac{x^{\pm 1},y^{\pm1}}}{(xy-qyx)}
\]
and the parameter $q \in \CCx$ corresponds to the class of the extension
\[
\begin{tikzcd}[column sep=4em]
0 \ar[r] & \ZZ  \ar[r,"{1\mapsto [A_q]}"] & \K[0]{A_q} \ar[r] & \ZZ(-1)\cdot \beta \ar[r] & 0\,,
\end{tikzcd}
\]
where  $\beta \in \grW[2]\K[0]{A_q}\cong\grW[2]\K[0]{(\CCx)^2}$ is as in \autoref{ex:toricPeriods}. By definition of the Yoneda product, this extension is $j$-torsion if and only if  there exists a morphism of mixed Hodge structures $ \phi : \K[0]{A_q} \to \ZZ$ such that $\phi([A_q]) = j$. 

To construct such a morphism, suppose that $q$ is a primitive $j$th root of unity.  Then, as is well known, the centre $Z =  Z(A_q)$ is generated by  $a^{\pm1} := x^{\pm j}$ and $b^{\pm 1} := y^{\pm j}$, and $A_q$ is Azumaya over $Z$.  In particular, if $p \in \Spec{Z}$ is a closed point,  then 
\begin{align*}
A_q \otimes_Z \cO{p} \cong \CC^{j\times j} \label{eq:q-torus-fibre}
\end{align*}
is the algebra of $j\times j$ matrices.  This induces a morphism of mixed Hodge structures
\[
- \otimes \cO{p} : \K[0]{A_q} \to \K[0]{\CC^{j\times j}} \cong \ZZ\cdot [V] \cong \ZZ\,,
\]
where $V \cong \CC^j$ is the fundamental representation of $\CC^{j\times j}$, and it sends $[A_q]$ to $[\CC^{j\times j}] = [V^{\oplus j}]= j$, as desired.

 Dually, we may view this as a splitting of the compactly supported K-theory
 \[
 \begin{tikzcd}
 0 \ar[r] & \ZZ(1)\beta^\vee \ar[r] & \Kc[0]{A_q} \ar[r] & \ZZ \ar[r] & 0\,,
 \end{tikzcd}
 \]
 which is equivalent to the data of an element in $\F[0]\KdR[0]{A_q}\cap \KB[0]{A_q}$ (i.e.~a Hodge class) that projects to $j \in \ZZ$.  This Hodge class exists if and only if $q^j=1$, and is provided by the $j$-dimensional representation above.   In summary, we could say that the ``noncommutative integral Hodge conjecture'' holds for quantum tori.